\documentclass[1 [leqno,11pt]{amsart}
\usepackage{amssymb, amsmath,amsmath,latexsym,amssymb,amsfonts,amsbsy, amsthm,mathtools,graphicx,CJKutf8,CJKnumb,CJKulem,color}
\numberwithin{equation}{section}

\allowdisplaybreaks
\let\al=\alpha

\let\d=\delta

\let\ep=\epsilon

\let\la=\lambda

\let\s=\sigma
\let\f=\frac

\let\om=\omega

\let\Om=\Omega

\let\na=\nabla
\let\th=\theta
\let\pa=\partial

\def\bbD{\mathbb{D}}
\def\bfD{\mathbf{D}}

\def\bfR{\mathbf{R}}

\def\tz{\tilde{z}}
\def\tom{\tilde{\omega}}
\def\trho{\tilde{\rho}}
\def\tG{\tilde{G}}
\def\tpsi{\tilde{\psi}}
\def\tchi{\tilde{\chi}}

\def\rmA{\mathrm{A}}

\def\rmD{\mathrm{D}}
\def\rmE{\mathrm{E}}

\def\rmJ{\mathrm{J}}

\def\rmf{\mathrm{f}}
\def\rmF{\mathrm{F}}
\def\rmR{\mathrm{R}}
\def\rmK{\mathrm{K}}

\def\rmM{\mathrm{M}}
\def\rmN{\mathrm{N}}
\def\rmT{\mathrm{T}}

\def\rmU{\mathrm{U}}
\def\rmw{\mathrm{w}}
\def\rmW{\mathrm{W}}

\def\R{\mathbf R}

\def\no{\noindent}

\def\eqdef{\buildrel\hbox{\footnotesize def}\over =}

\def\bbT{\mathbb{T}}

\newcommand{\beq}{\begin{equation}}
\newcommand{\eeq}{\end{equation}}
\newcommand{\ben}{\begin{eqnarray}}
\newcommand{\een}{\end{eqnarray}}
\newcommand{\beno}{\begin{eqnarray*}}
\newcommand{\eeno}{\end{eqnarray*}}
\newcommand{\udl}[1]{\underline{#1}}

\newtheorem{theorem}{Theorem}[section]

\newtheorem{lemma}[theorem]{Lemma}
\newtheorem{proposition}[theorem]{Proposition}
\newtheorem{corol}[theorem]{Corollary}
\newtheorem{remark}[theorem]{Remark}
\begin{document}
\begin{CJK*}{UTF8}{gkai}
\title{Nonlinear inviscid damping for a class of monotone shear flows in finite channel}
\author{Nader Masmoudi}
\address{Department of Mathematics, New York University in Abu Dhabi, Saadiyat Island, P.O. Box 129188, Abu Dhabi, United Arab Emirates. Courant Institute of Mathematical Sciences, New York University, 251 Mercer Street, New York, NY 10012, USA,}
\email{masmoudi@cims.nyu.edu}
\author{Weiren Zhao}
\address{Department of Mathematics, New York University in Abu Dhabi, Saadiyat Island, P.O. Box 129188, Abu Dhabi, United Arab Emirates.}
\email{zjzjzwr@126.com, wz19@nyu.edu}
\maketitle

\begin{abstract}
We prove the nonlinear inviscid damping for a class of monotone shear flows in $\mathbb{T}\times [0,1]$ for initial perturbation in Gevrey-$\f1s$($s>2$) class with compact support. The main idea of the proof is to use the wave operator of a slightly modified Rayleigh operator in a well chosen coordinate system. 
\end{abstract}

\section{Introduction}
We consider the 2D Euler system in the vorticity formulation with a background shear flow $(u(y), 0)$:
\beq
\label{eq:vorticity}
\left\{
\begin{aligned}
&\om_t+u(y)\pa_x\om-u''(y)\pa_x\psi+U\cdot\nabla \om=0,\\
&U=\na^{\bot}\psi, \quad \Delta\psi=\om,\quad \psi(t,x,0)=\psi(t,x,1)=0,\\
&\om|_{t=0}=\om_{in}(x,y).
\end{aligned}
\right.
\eeq
Here, $(x,y)\in \mathbb{T}\times [0,1]$, $\na^{\bot}=(-\pa_y,\pa_x)$ and $(U,\om)$ are periodic in $x$ variable with period normalized to $2\pi$. The physical velocity is $(u(y),0)+U$ where $U=(U^x,U^y)$ denotes the velocity perturbation and the total vorticity is $-u'(y)+\om$. 

In this work, we are interested in the long time behavior of \eqref{eq:vorticity} for small initial perturbations $\om_{in}$. In particular, we show that all sufficiently small perturbations in a suitable regularity class undergo `inviscid damping' and satisfy $(u(y),0)+U(t,x,y)\to (u(y)+u_{\infty}(y),0)$ as $t\to \infty$ for some $u_{\infty}(y)$ determined by the evolution. 

The field of hydrodynamic stability started in the nineteenth century with Stokes, Helmholtz, Reynolds, Rayleigh, Kelvin, Orr, Sommerfeld and many others. 
In \cite{Orr}, Orr observed an important phenomenon that the velocity will tend to 0 as $t\to \infty$. This phenomenon is so-called inviscid damping, which is the analogue in hydrodynamics of Landau damping found by Landau \cite{Lan}, which predicted the rapid decay of the electric field of the linearized Vlasov equation around homogeneous equilibrium. Mouhot and Villani \cite{MV} made a breakthrough and proved nonlinear Landau damping for the perturbation in Gevrey class(see also \cite{BMM}). In this case, the mechanism leading to the damping is the vorticity mixing driven by shear flow or Orr mechanism \cite{Orr}. See \cite{RZ17,Ry} for similar phenomena in various system.

Due to the presence of the nonlocal operator for general shear flows, the inviscid damping for general shear flows is a challenge problem even at linear level.
For the linear inviscid damping we refer to \cite{Zill,WZZ1,Jia,Jia2} for the results for general monotone flows. 
For non-monotone flows such as the Poiseuille flow and the Kolmogorov flow, another dynamic phenomena should be taken into consideration, which is so called the vorticity depletion phenomena, predicted by Bouchet and Morita \cite{BouMor} and later proved by Wei, Zhang and Zhao \cite{WZZ2,WZZ3}. See also \cite{BCV2017}. 

Due to the possible nonlinear transient growth, it is a challenging task extending linear damping to nonlinear damping. 
Even for the Couette flow there are only few results. Moreover, nonlinear damping is sensitive to the topology of the perturbation. 
Indeed, Lin and Zeng \cite{LZ} proved that nonlinear inviscid damping is not true for perturbations of the Couette flow in $H^s$ for $s<\f32$. 
Bedrossian and Masmoudi \cite{BM1} proved nonlinear inviscid damping around the Couette flow in Gevrey class $2_-$. 
Recently Deng and Masmoudi \cite{DM}  proved some instability for initial perturbations in Gevrey class $2_+$. 
We refer to \cite{IonescuJia,Jia2} and references therein for other related interesting results. 

Our main result is
\begin{theorem}\label{Thm: main}
Suppose $u(y)$ is a smooth function defined on $[0,1]$ which satisfies
\begin{itemize}
\item[1.]  (Monotone) There exists $c_0>0$ such that $u'\geq c_0>0$.
\item[2.] (Compact support) There exists $\th_0\in (0,\f{1}{10}]$ such that $\mathrm{supp}\, u''\subset [4\th_0,1-4\th_0]$. 
\item[3.] (Linear stability) The Rayleigh operator $u\mathrm{Id}-u''\Delta^{-1}$ has no eigenvalue and no embedding eigenvalue. 
\item[4.] (Regularity) There exist $K>1$ and $s_0\in (0,1)$ such that for all integers $m\geq 0$,
\beno
\sup_{y\in\R}\left|\f{d^m(u''(y))}{dy^m}\right|\leq K^m(m!)^{\f{1}{s_0}}(m+1)^{-2}.
\eeno
\end{itemize}
For all $1>s_0\geq s>1/2$ and $\la_{in}>0$, there exist $\la_{in}>\la_{\infty}=\la_{\infty}(\la_{in}, K,\th_0,s)>0$ and $0<\ep_0=\ep_0(\la_{in},\la_{\infty},\th_0,s)\leq \f12$ such that for all $\ep\leq \ep_0$ if $\om_{in}$ has compact support in $\mathbb{T}\times [3\th_0,1-3\th_0]$ and satisfies 
\beno
\|\om_{in}\|_{\mathcal{G}^{\la_{in}}}^2=\sum_{k}\int\left|\widehat{\om}_{in}(k,\eta)\right|e^{2\la_{in}|k,\eta|^s}d\eta\leq \ep^2,
\quad \int_{\mathbb{T}\times [0,1]}\om_{in}(x,y)dxdy=0
\eeno
then the smooth solution $\om(t)$ to \eqref{eq:vorticity} satisfies: 
\begin{itemize}
\item[1.] (Compact support) For all $t\geq 0$, $\mathrm{supp}\, \om(t)\subset \mathbb{T}\times [2\th_0,1-2\th_0]$. 
\item[2.] (Scattering) There exists $f_{\infty}\in \mathcal{G}^{\la_{\infty}}$ with $\mathrm{supp}\, f_{\infty}\subset \mathbb{T}\times [2\th_0,1-2\th_0]$ such that for all $t\geq 0$, 
\beq\label{eq: Scattering} 
\left\|\om(t,x+tu(y)+\Phi(t,y),y)-f_{\infty}(x,y)\right\|_{\mathcal{G}^{\la_{\infty}}}\lesssim \f{\ep}{\langle t\rangle},
\eeq
where $\Phi(t,y)$ is given explicitly by 
\beq\label{eq:Phi(t,y)}
\Phi(t,y)=\f{1}{2\pi}\int_0^t\int_{\mathbb{T}}U^x(\tau,x,y)dxd\tau=u_{\infty}(y)t+O(\ep^2),
\eeq
with $u_{\infty}=\pa_y<\Delta^{-1}f_{\infty}>$. 
\item[3.] (Inviscid damping) 
The velocity field $U$ satisfies
\begin{align}\label{eq: inviscid damping}
\left\|\f{1}{2\pi}\int U^x(t,x,\cdot)dx-u_{\infty}\right\|_{\mathcal{G}^{\la_{\infty}}}&\lesssim \f{\ep^2}{\langle t\rangle^2},\\
\left\|U^x(t)-\f{1}{2\pi}\int U^x(t,x,\cdot)dx\right\|_{L^2}&\lesssim \f{\ep}{\langle t\rangle},\\
\big\|U^y(t)\big\|_{L^2}&\lesssim \f{\ep}{\langle t\rangle^2}.
\end{align}
\end{itemize}
\end{theorem}
Let us give some remarks about Theorem \ref{Thm: main}:
\begin{itemize}
\item The compact support assumption of $u''$ is to prevent  boundary effects.
\item If the background flow satisfies $u''(y)\geq 0$, the Rayleigh operator $u\mathrm{Id}-u''\Delta^{-1}$ has no eigenvalue and no embedding eigenvalue. 
\item The regularity assumption means that $u''$ is in some Gevrey-$\f{1}{s_0}$ class. 
\item The solution will be less regular than the background flow $u(y)$. 
\end{itemize}
Let also mention a very recent paper by Ionescu and Jia \cite{IonescuJia2}, where a similar result to ours was proved with a different method. The two papers are independent. 

Compared to {\it (1.5)} in \cite{BM1}, \eqref{eq: Scattering} loses an $\ep$, due to the linear nonlocal effect. We can recover the $\ep^2$ by using the wave operator $\bbD_{u}$ defined in section 2: 
\begin{corol}\label{Rmk: better-scattering}
Let us define
\ben
\bbD_{u}\om(t,x,y)=\sum_{k=0}\bbD_{u,k}(\hat{\om}_k)(t,y)e^{ikx}.
\een
Then there exists $\rmf_{\infty}$ such that for all $t\geq 0$, 
\ben\label{eq:new-scattering}
\left\|(\bbD_{u}\om)\big(t,x+tu(y)+\Phi(t,y),y\big)-\rmf_{\infty}(x,y)\right\|_{L^2}\lesssim \f{\ep^2}{\langle t\rangle}.
\een
\end{corol}

In \eqref{eq:new-scattering}, we can improve $L^2$ to Gevrey class by modulating the uniform estimate of the wave operator $\bbD_{u,k}$ in Gevrey class. 

\subsection{Notation and conventions}
For $f(x,y)$ in Schwartz space with compact support in $(0,1)$, we define the Fourier transform in the first direction $\mathcal{F}_{1}f(k,y)$, the Fourier transform in the second direction $\mathcal{F}_{2}f(x,\eta)$ and the Fourier transform in both directions $\hat{f}_k(\eta)$ where $(k,\eta)\in \mathbb{Z}\times \R$, 
\begin{align*}
&\mathcal{F}_{1}f(k,y)=\f{1}{2\pi}\int_{\mathbb{T}}f(x,y)e^{-ixk}dx,\\
&\mathcal{F}_{2}f(x,\eta)=\int_{\R}f(x,y)e^{-iy\eta}dy,\\
&\hat{f}_k(\eta)=\f{1}{2\pi}\int_{\mathbb{T}\times \R}f(x,y)e^{-ixk-iy\eta}dxdy. 
\end{align*}
A convention we generally use is to denote the discrete $x$ (or $\tz, \, z$) frequencies as subscripts. By convention we always use Greek letters such as $\eta$ and $\xi$ to denote frequencies in the $y$, $u$ or $v$ direction and lowercase Latin characters commonly used as indices such as $k$ and $l$ to denote frequencies in the $x$, $\tz$ or $z$ direction (which are discrete). Another convention we use is to denote $\rmK, \rmM, \rmN$ as dyadic integers $\rmK, \rmM, \rmN\in \mathcal{D}$ where 
\beno
\mathcal{D}=\left\{\f12, 1, 2, ..., 2^j,...\right\}.
\eeno
When the sum is written with indices $\rmK, \rmM, \rmM', \rmN$ or $\rmN'$ it will always be over a subset of $\mathcal{D}$. 
This will be useful when defining Littlewood-Paley projections and paraproduct decompositions, see {\it section A.1} in \cite{BM1}. 

We use the notation $f\lesssim g$ when there exists a constant $C>0$ independent of the parameters of interest such that $f\leq Cg$ (we analogously $f\gtrsim g$ define). Similarly, we use the notation $f\approx g$ when there exists $C>0$ such that $C^{-1}g\leq  f \leq  Cg$. We sometimes use the notation $f\lesssim_{\al} g$ if we want to emphasize that the implicit constant depends on some parameter $\al$.

We will denote the $l^1$ vector norm $|k, \eta| = |k| + |\eta|$, which by convention is the norm taken in our work. Similarly, given a scalar or vector in $\R^d$  we denote
\beno
\langle v\rangle =\left(1+|v|^2\right)^{\f12}
\eeno
We use a similar notation to denote the $x$ (or $z$) average of a function: 
$f_0=<f>=\f{1}{2\pi}\int_{\mathbb{T}}f(x,y)dx=P_0f$. We also frequently use the notation $P_{\neq}f=f-f_0$. We denote the standard $L^p$ norm by $\|f\|_{L^p}$. For any $f$ with compact support in $(0,1)$, we make common use of the Gevery-$\f1s$ norm with Sobolev correction defined by
\beno
\|f\|_{\mathcal{G}^{\la,\s;s}}=\sum_k\int\left|\hat{f}_k(\eta)\right|^2e^{2\la|k,\eta|^s}\langle k,\eta\rangle^{2\s}d\eta.
\eeno
We refer to {\it section A.2} in \cite{BM1} for a discussion of the basic properties of this norm and some related useful inequalities. 
For $\eta\geq 0$, we define $\rmE(\eta)\in \mathbf{Z}$ to be the integer part. We define for $\eta\in \R$ and $1\leq |k|\leq \rmE(\sqrt{|\eta|})$ with $\eta k\geq 0$, $t_{k,\eta}=\Big|\f{\eta}{k}\Big|-\f{|\eta|}{2|k|(|k|+1)}$ and $t_{0,\eta}=2|\eta|$ and the critical intervals
\beno
\mathrm{I}_{k,\eta}=\left\{
\begin{split}
&[t_{|k|,\eta},t_{|k|-1,\eta}]\quad &\text{if}\ \eta k\geq 0 \ \text{and } 1\leq |k|\leq \rmE(\sqrt{|\eta|}),\\
&\emptyset \quad &\text{otherwise}.
\end{split}\right.
\eeno
For minor technical reasons, we define a slightly restricted subset as the resonant intervals
\beno
\mathbf{I}_{k,\eta}=\left\{
\begin{split}
&\mathrm{I}_{k,\eta}\quad &\sqrt{|\eta|}\leq t_{k,\eta},\\
&\emptyset \quad &\text{otherwise}.
\end{split}\right.
\eeno

We use $1_{A}$ be the characteristic function which means $1_{A}(x)=\left\{\begin{aligned}&1\quad x\in A\\&0\quad x\notin A\end{aligned}\right.$. 

We also use the smooth cut-off functions $\chi_1$ and $\chi_2$ in Gevrey-$\f{2}{s_0+1}$ class which satisfies 
\begin{align*}
&\mathrm{supp}\, \chi_1\subset [\th_0,1-\th_0],\quad 
\chi_1(y)\equiv 1\  \text{for}\  y\in [1-2\th_0,2\th_0],\\
&\mathrm{supp}\, \chi_2\subset \big[\f{\th_0}{2},1-\f{\th_0}{2}\big],\quad  
\chi_2(y)\equiv 1\  \text{for}\  y\in [\th_0,1-\th_0],
\end{align*} 
and there exist $K_1>1$ such that for all integers $m\geq 0$
\ben\label{eq: chi}
\sup_{y\in\R}\left|\f{d^m\chi_i(y)}{dy^m}\right|\leq K_1^m(m!)^{\f{2}{s_0+1}}(m+1)^{-2},\quad i=1,2. 
\een
One may refer \cite{IonescuJia,Jia2} for the construction of such cut-off functions. 

\subsection{Discussion}
The key idea of this paper is to use the wave operator in a well chosen coordinate system. 
We are able to reduce the length of the paper, since we use \cite{BM1} and \cite{WZZ1} as black boxes. There are four propositions (see Proposition \ref{prop: transport}, \ref{prop:reaction}, \ref{prop: remainder} and \ref{prop: coordinate}) which are identical with the propositions in \cite{BM1}. The elliptic estimate (see Proposition \ref{prop:elliptic}) is slightly different, since the linear change of coordinates twists the Laplace operator $\Delta_{x,y}$. In the energy estimate, we mainly focus on the new terms. One of them comes from the application of the wave operator (see Proposition \ref{prop:com}), the other is coming from the nonlocal term (see Proposition \ref{prop: nonlocal} and \ref{prop:nonlocal-ep}).

\section{Proof of Theorem \ref{Thm: main}}
Next we give the proof of Theorem \ref{Thm: main}, starting the primary steps as propositions which are proved in the subsequent sections. 

\subsection{Time-dependent norm}
We will use the same multiplier $\rmA(t,\na)$ introduced in \cite{BM1}. 
\beno
\rmA_k(t,\eta)=e^{\la(t)|k,\eta|^s}\langle k,\eta\rangle^{\s}\rmJ_{k}(t,\eta).
\eeno
The index $\la(t)$ is the bulk Gevrey-$\f{1}{s}$ regularity and will be chosen to satisfy
\beno
&&\la(t)=\f34\la_0+\f14\la',\quad t\leq 1\\
&&\dot{\la}(t)=-\f{\d_{\la}}{\langle t\rangle^{2\tilde{q}}}(1+\la(t)),\quad t\in (1,\infty)
\eeno
where $\la_0, \la'$ are parameters which depend on the regularity of the background flow $u(y)$ and chosen by the proof, $\d_{\la}\approx \la_0-\la'$ is a small parameter that ensures $\la(t)\geq \f{\la_0}{2}+\f{\la'}{2}$ for $t\geq 0$ and $\f12<\tilde{q}\leq \f{s}{8}+\f{7}{16}$ is a parameter chosen by the proof. 

Roughly speaking, we need the Gevrey bandwidth $\la(t)$ in the multiplier $\rmA$ to be sufficiently small (if $s=s_0$) and determined by the background flow so that the wave operator is bounded in $\{f\in C^{\infty}:~\|\rmA f\|_2<\infty\}$. See Proposition \ref{prop: kernel-wave-op}, Remark \ref{eq: wave-A} and Remark \ref{Rmk: wave bdd-A} for more details. 

The main multiplier for dealing with the Orr mechanism and the associated nonlinear growth is 
\beq\label{eq: J}
\rmJ_{k}(t,\eta)=\f{e^{\mu |\eta|^{\f12}}}{w_k(t,\eta)}+e^{\mu |k|^{\f12}},
\eeq
where 
\beq\label{eq: w_k(t,eta)}
w_k(t,\eta)=\left\{
\begin{aligned}
&w_k(t_{\rmE(\sqrt{\eta}),\eta},\eta)\quad t<t_{\rmE(\sqrt{\eta}),\eta}\\
&w_{\mathrm{NR}}(t,\eta)\quad t\in [t_{\rmE(\sqrt{\eta}),\eta},2\eta]\setminus \mathrm{I}_{k,\eta}\\
&w_{\mathrm{R}}(t,\eta)\quad t\in \mathrm{I}_{k,\eta}\\
&1\quad t\geq 2\eta.
\end{aligned}
\right.
\eeq
Here $(w_{\mathrm{R}}(t,\eta),w_{\mathrm{NR}}(t,\eta))$ is defined in the following way: 

Let $w_{\mathrm{NR}}$ be a non-decreasing function of time with $w_{\mathrm{NR}}(t,\eta)=1$ for $t\geq 2\eta$. For definiteness, we remark here that for $|\eta|\leq 1$, $w_{\mathrm{NR}}(t,\eta)=1$, which will be a consequence of  the definition. For $k=1,2,3,...,\rmE(\sqrt{\eta})$, we define
\beq\label{eq: w_NR-R}
\begin{split}
w_{\mathrm{NR}}&(t,\eta)=\left(\f{k^2}{\eta}\Big(1+b_{k,\eta}\Big|t-\f{\eta}{k}\Big|\Big)\right)^{C\kappa}w_{\mathrm{NR}}(t_{k-1,\eta},\eta),\\
&\forall t\in \mathrm{I}^{\mathrm{R}}_{k,\eta}=\left[\f{\eta}{k},t_{k-1,\eta}\right],
\end{split}
\eeq
and
\beq\label{eq: w_NR-L}
\begin{split}
w_{\mathrm{NR}}&(t,\eta)=\left(1+a_{k,\eta}\Big|t-\f{\eta}{k}\Big|\right)^{-1-C\kappa}w_{\mathrm{NR}}\left(\f{\eta}{k},\eta\right),\\
&\forall t\in \mathrm{I}^{\mathrm{L}}_{k,\eta}=\left[t_{k,\eta},\f{\eta}{k}\right].
\end{split}
\eeq
The constant $b_{k,\eta}$ is chosen to ensure that $\f{k^2}{\eta}\Big(1+b_{k,\eta}\Big|t_{k-1,\eta}-\f{\eta}{k}\Big|\Big)=1$, hence for $k\geq 2$
\beq\label{eq: b_k,eta}
b_{k,\eta}=\f{2(k-1)}{k}\left(1-\f{k^2}{\eta}\right)
\eeq
and $b_{1,\eta}=1-\f{1}{\eta}$ and similarly, 
\beq\label{eq: a_k,eta}
a_{k,\eta}=\f{2(k+1)}{k}\left(1-\f{k^2}{\eta}\right).
\eeq
On each interval $\mathrm{I}_{k,\eta}$, we define $w_{\mathrm{R}}(t,\eta)$ by
\begin{align}\label{eq: w_R-R}
&w_{\mathrm{R}}(t,\eta)=\f{k^2}{\eta}\left(1+b_{k,\eta}\Big|t-\f{\eta}{k}\Big|\right)w_{\mathrm{NR}}(t,\eta),\quad
\forall t\in \mathrm{I}^{\mathrm{R}}_{k,\eta}=\left[\f{\eta}{k},t_{k-1,\eta}\right],\\
\label{eq: w_R-L}
&w_{\mathrm{R}}(t,\eta)=\f{k^2}{\eta}\left(1+a_{k,\eta}\Big|t-\f{\eta}{k}\Big|\right)w_{\mathrm{NR}}(t,\eta),\quad
\forall t\in \mathrm{I}^{\mathrm{L}}_{k,\eta}=\left[t_{k,\eta},\f{\eta}{k}\right].
\end{align}
We also define $\rmJ^{\mathrm{R}}(t,\eta)$ and $\rmA^{\mathrm{R}}(t,\eta)$ to assign resonant regularity at every critical time: 
\beq\label{eq: J^R,A^R}
\begin{split}
&\rmJ^{\mathrm{R}}(t,\eta)
=\left\{\begin{aligned}
&e^{\mu |\eta|^{\f12}}w_{\mathrm{R}}^{-1}(t_{\rmE(\sqrt{\eta}),\eta},\eta)\quad t<t_{\rmE(\sqrt{\eta}),\eta}\\
&e^{\mu |\eta|^{\f12}}w_{\mathrm{R}}^{-1}(t,\eta)\quad t\in [t_{\rmE(\sqrt{\eta}),\eta},2\eta]\\
&e^{\mu |\eta|^{\f12}}\quad t\geq 2\eta,
\end{aligned}\right.\\
&\rmA^{\mathrm{R}}(t,k,\eta)=e^{\la(t)|\eta|^s}\langle\eta\rangle^{\s}\rmJ^{\mathrm{R}}(t,\eta).
\end{split}
\eeq

One may refer to \cite{BM1} for more basic properties of the multiplier $\rmA(t,\na)$.

\subsection{Wave operator}
The wave operator related to a self-adjoint operators is well known. Let $A, B$ be two self-adjoint operators in the Hilbert space $H$, then the wave operator $\bbD$ related to $A$ and $B$ satisflies
\beno
A\bbD=\bbD B
\eeno
It can be defined by
\beno
\bbD=\lim_{t\to \infty}e^{-itA}e^{itB}. 
\eeno
However the wave operator related to non-self-adjoint operators is usually not easy to construct and estimate. 
Recently, this was successfully used to solve important problems in fluid mechanics.
In \cite{LWZ}, the authors use the wave operator method to solve Gallay's conjecture on pseudospectral and spectral bounds of the Oseen vortices operator. In \cite{WZZ3}, the wave operator method was used to solve Beck and Wayne's conjecture. 

In order to absorb the nonlocal term, we also construct a wave operator $\bbD_{u,k}$ which is related to the Rayleigh operator $\mathcal{R}_{u,k}=u(y)\mathrm{Id}-u''(y)\Delta_k^{-1}$ where $\Delta_k=\pa_y^2-k^2$. We have the following proposition. 
\begin{proposition}\label{prop: general-wave}
Suppose $k\neq 0$, $u\in C^4([0,1])$ is a strictly monotone function and $\mathcal{R}_{u,k}$ has no eigenvalue and no embedding eigenvalue. Then there exists $\bbD_{u,k}$ and $\bbD^1_{u,k}$ such that
\ben\label{eq:DR=uD}
\bbD_{u,k}\mathcal{R}_{u,k}=u\bbD_{u,k},
\een
and for $f,g\in L^2(0,1)$
\ben\label{eq: id1}
\int_0^1\bbD_{u,k}(f)(y)\bbD^1_{u,k}(g)(y)dy=\int_0^1f(y)g(y)dy. 
\een
Moreover, there exist $C>1$ independent of $k$ such that 
\ben\label{eq: est1}
C^{-1}\leq \|\bbD_{u,k}\|_{L^2\to L^2}\leq C,\quad C^{-1}\leq\|\bbD^1_{u,k}\|_{L^2\to L^2}\leq C.
\een
\end{proposition}
Note that
\ben
\bbD^1_{u,k}=(\bbD_{u,k}^{-1})^*.
\een
The proof can be found in Section \ref{sect: The existence of wave operator}. One can also refer to \cite{WZZ1} for more details. 

\begin{remark}\label{Rmk: Formula}
We have the following representation formula of the wave operator and its inverse:
\begin{align}
\label{eq: defD}\bbD_{u,k}(\om)(k,y)&=d_1(k,y)\om(y)+u''(y)d_2(k,y)\int_0^1\f{e(k,y',y)}{u(y')-u(y)}\om(y')dy',\\
\bbD_{u,k}^1(\om)(k,y)&=d_1(k,y)\om(y)+d_2(k,y)\int_0^1\f{e(k,y',y)u''(y')}{u(y')-u(y)}\om(y')dy',\\
\label{eq: defD^-1}\bbD_{u,k}^{-1}(\om)(k,y)&=d_1(k,y)\om(y)+u''(y)\int_0^1\f{e(k,y,y')}{u(y)-u(y')}d_2(k,y')\om(y')dy'. 
\end{align}
where $d_1(k,y)$, $d_2(k,y)$ and $e(k,y',y)$ are defined in \eqref{eq:d_1}, \eqref{eq:d_2} and \eqref{eq:e}. 
\end{remark}
Let us now introduce the linear change of coordinates, namely $(t,x,y)\to (t,z,v)$:
 \ben\label{eq: linear-coordinate}
v=u(y),\quad 
z=x-tv.
\een
Under this linear change of coordinates, $\Delta_{x,y}$ becomes $\Delta_u$, which is defined by
\ben
\Delta_u\eqdef\pa_{zz}+(\tilde{u'})^2(\pa_v-t\pa_{z})^2+\widetilde{u''}(\pa_v-t\pa_{z})
\een 
where $\tilde{u'}=u'\circ u^{-1}(v)$ and $\widetilde{u''}=u''\circ u^{-1}(v)$. Here and below, for any $\varphi(y)$, $\tilde{\varphi}(v)=\varphi\circ u^{-1}(v)$. 

We also denote $\Delta_{u,k}=-k^2+(\tilde{u'})^2(\pa_v-itk)^2+\widetilde{u''}(\pa_v-itk)$ to be the Fourier tranform of $\Delta_u$ in the first direction. 

Let us now write \eqref{eq: defD} and \eqref{eq: defD^-1} in the $(z,v)$ coordinate. 
\begin{align*}
&\bfD_{u,k}\big(\mathcal{F}_{1}{f}(t,k,\cdot)\big)(t,k,v)\\
&=D_1(k,v)\mathcal{F}_{1}{f}(t,k,v)\\
&\quad+u''(u^{-1}(v))D_2(k,v)\int_{u(0)}^{u(1)}E(k,v_1,v)\f{\mathcal{F}_{1}{f}(t,k,v_1)e^{-i(v_1-v)tk}}{v_1-v}(u^{-1})'(v_1)dv_1,\\
&\bfD_{u,k}^{-1}\big(\mathcal{F}_{1}{f}(t,k,\cdot)\big)(t,k,v)\\
&=D_1(k,u)\mathcal{F}_{1}{f}(t,k,v)\\
&\quad+u''(u^{-1}(v))\int_{u(0)}^{u(1)}E(k,v,v_1)D_2(k,v_1)\f{\mathcal{F}_{1}{f}(t,k,v_1)e^{i(v_1-v)tk}}{v-v_1}(u^{-1})'(v_1)dv_1.
\end{align*}
We change $\bbD_{u,k}^1$ in the $(z,v)$ coordinate slightly (see Remark \ref{Rmk:dual-v}) to make the dual of $\bfD^{-1}_{u,k}$ in $v$ variable: 
\begin{align*}
&\bfD_{u,k}^1\big(\mathcal{F}_{1}{f}(t,k,\cdot)\big)(t,k,v)\\
&=D_1(k,v)\mathcal{F}_{1}{f}(t,k,v)\\
&\quad+D_2(k,v)\int_{u(0)}^{u(1)}E(k,v_1,v)\f{u''(u^{-1}(v_1))\mathcal{F}_{1}{f}(t,k,v_1)e^{-i(v_1-v)tk}}{v_1-v}(u^{-1})'(v)dv_1.
\end{align*}
Here $D_1(k,v)$, $D_2(k,v)$ and $E(k,v_1,v)$ satisfy
\beq\label{eq:DE}
\begin{split}
&D_1(k,u(y))=d_1(k,y),\quad D_2(k,u(y))=d_2(k,y),\\
&E(k,u(y_1),u(y))=e(k,y_1,y).
\end{split}\eeq
One can also refer to \eqref{eq: D_1}, \eqref{eq: D_2} and \eqref{eq: E} for the explicit formula. 

\begin{proposition}\label{prop:wave-coordinate}
For $k\neq 0$, let $\mathbf{R}_{u,k}=v\mathrm{Id}-u''( u^{-1}(v))\Delta_{u,k}^{-1}$ be a modified Rayleigh operator. 
Then it holds for any ${f}(t,z,v)$ and ${g}(t,z,v)$ that
\begin{align}
&\bfD_{u,k}\bfR_{u,k}=v\bfD_{u,k},\\
&\bfD_{u,k}\bfD_{u,k}^{-1}=\bfD_{u,k}^{-1}\bfD_{u,k}=\mathrm{Id},\ \bfD_{u,k}^{-1}=(\bfD_{u,k}^1)^*,\\
\label{eq:[pa_t,D]}
&[\pa_t,\bfD_{u,k}]\big(\mathcal{F}_{1}{{f}}(t,k,\cdot)\big)(t,k,v)
=[ikv,\bfD_{u,k}]\big(\mathcal{F}_{1}{{f}}(t,k,\cdot)\big)(t,k,v). \\
&\int\bfD_{u,k}\big(\mathcal{F}_{1}{f}(t,k,\cdot)\big)(t,k,v)\overline{\bfD_{u,k}^1\big(\mathcal{F}_{1}{g}(t,k,\cdot)\big)(t,k,v)}dv
=\int\mathcal{F}_{1}{f}(k,u)\overline{\mathcal{F}_{1}{g}(k,v)}dv,
\end{align}
and going back to $(x,y)$, it holds that
\begin{align*}
\bfD_{u,k}\big(\mathcal{F}_{1}{\udl{f}}(t,k,\cdot)\big)(t,k,v)e^{ikz}\bigg|_{z=x-u(y)t,\ u=u(y)}
&=\bfD_{u,k}\big(\mathcal{F}_{1}{\udl{f}}(t,k,\cdot)\big)(t,k,u(y))e^{ik(x-tu(y))}\\
&=\bbD_{u,k}\big(\mathcal{F}_{1}{f}(t,k,\cdot)\big)(t,k,y)e^{ikx},
\end{align*}
for any $f(t,x,y)=\udl{f}(t,z,v)\big|_{z=x-u(y)t,\ u=u(y)}$. 
\end{proposition}
\begin{remark}\label{Rmk:dual-v}
It holds that
\beno
\bfD_{u,k}^1\big(\mathcal{F}_{1}\udl{f}(t,k,\cdot)\big)(t,k,v)e^{ikz}\bigg|_{z=x-u(y)t,\ v=u(y)}
=\f{1}{u'(y)}\bbD_{u,k}^1\big(u'\mathcal{F}_{1}{f}(t,k,\cdot)\big)(t,k,y)e^{ikx}.
\eeno
for any $f(t,x,y)=\udl{f}(t,z,v)\big|_{z=x-tu(y),\ v=u(y)}$. 
\end{remark}
The proposition follows directly form Proposition \ref{prop: general-wave} and Remark \ref{Rmk: Formula}. We omit the proof. 

Although the wave operator is nonlocal both in physical space and frequency space, the next proposition shows that the wave operator does not move frequencies a lot. 

\begin{proposition}\label{prop: kernel-wave-op}
Recall $\tchi_2=\chi_2\circ u^{-1}$ be a smooth function with compact support such that $\mathrm{supp}\, \chi_2\subset [\f{\th_0}{2},1-\f{\th_0}{2}]$ and satisfies \eqref{eq: chi}. Then for any $0<|k|\leq k_{M}$, there exist $\mathcal{D}(t,k,\xi_1,\xi_2)$ and $\mathcal{D}^{-1}(t,k,\xi_1,\xi_2)$ such that 
\beno
\mathcal{F}_{2}\Big(\bfD_{u,k}\big(\tchi_2\mathcal{F}_{1}f(t,k,\cdot)\big)\Big)(t,k,\xi_1)=\int \mathcal{D}(t,k,\xi_1,\xi_2)\hat{f}_k(t,\xi_2)d\xi_2,
\eeno
and
\beno
\mathcal{F}_{2}\Big(\tchi_2\bfD_{u,k}^1\big(\tchi_2\mathcal{F}_{1}f(t,k,\cdot)\big)\Big)(t,k,\xi_1)=\int \mathcal{D}^1(t,k,\xi_1,\xi_2)\hat{f}_k(t,\xi_2)d\xi_2,
\eeno
and
\beno
\mathcal{F}_{2}\Big(\bfD_{u,k}^{-1}\big(\tchi_2\mathcal{F}_{1}f(t,k,\cdot)\big)\Big)(t,k,\xi_1)=\int \mathcal{D}^{-1}(t,k,\xi_1,\xi_2)\hat{f}_k(t,\xi_2)d\xi_2.
\eeno
Moreover, there exists $\la_{\mathcal{D}}=\la_{\mathcal{D}}(\la_0,\th_0,s_0,s_1, k_{M})$ independent of $t$ such that
\beno
\left|\mathcal{D}(t,k,\xi_1,\xi_2)\right|+\left|\mathcal{D}^1(t,k,\xi_1,\xi_2)\right|+\left|\mathcal{D}^{-1}(t,k,\xi_1,\xi_2)\right|\lesssim e^{-\la_{\mathcal{D}}|\xi_1-\xi_2|^{s_0}}. 
\eeno
\end{proposition}
This proposition is proved in Section \ref{Sec:Gevrey regularity}. 
\begin{corol}\label{corol: commutator}
There exists $\mathcal{D}^{com}(t,k,\xi,\xi_1)$ such that
\beno
\mathcal{F}_2\Big(\tchi_2\bfD_{u,k}(\tchi_2\mathcal{F}_1{f})-\tchi_2\bfD_{u,k}^1(\tchi_2\mathcal{F}_1{f})\Big)(\xi)=\int \mathcal{D}^{com}(t,k,\xi,\xi_1)\hat{f}_k(t,\xi_1)d\xi_1.
\eeno
Moreover, there exists $\la_{\mathcal{D}}=\la_{\mathcal{D}}(\la_0,\th_0,s_0,s_1, k_{M})$ independent of $t$ such that
\beno
\big|\mathcal{D}^{com}(t,k,\xi,\xi_1)\big|\lesssim \min\left(\f{e^{-\la_{\mathcal{D}}|\xi-\xi_1|^{s_0}}}{1+|\xi-kt|^2}, \f{e^{-\la_{\mathcal{D}}|\xi-\xi_1|^{s_0}}}{1+|\xi_1-kt|^2}\right).
\eeno
\end{corol}
The proof can be found in section \ref{Sec: comm}. 

\begin{remark}\label{eq: wave-A}
There exists $C=C(\la_{\mathcal{D}},\th_0,s,s_0, k_{M})$ independent of $t$ such that for any $\la<\la_{\mathcal{D}}$ and $k\leq k_{M}$ 
\beno
\big\|\mathbf{D}_{k,u}(\tchi_2\mathcal{F}_{1}f)\big\|_{\mathcal{G}^{\la,\s;s}}
+\left\|\mathbf{D}_{k,u}^{-1}(\tchi_2\mathcal{F}_{1}f)\right\|_{\mathcal{G}^{\la,\s;s}}
+\left\|\tchi_2\mathbf{D}_{k,u}^{1}(\tchi_2\mathcal{F}_{1}f)\right\|_{\mathcal{G}^{\la,\s;s}}\leq C\|\mathcal{F}_{1}f\|_{\mathcal{G}^{\la,\s;s}},
\eeno
and
\beno
\sum_{0<|k|\leq k_{M}}\left(\big\|\rmA_k\mathbf{D}_{k,u}(\tchi_2\mathcal{F}_{1}f)\big\|_{L^2}^2
+\left\|\rmA_k\mathbf{D}_{k,u}^{-1}(\tchi_2\mathcal{F}_{1}f)\right\|_{L^2}^2\right)\leq C\|\rmA f\|_{L^2}^2.
\eeno
\end{remark}
The proof is similar to the proof of {\it (3.39a)} in \cite{BM1}. 
\subsection{Nonlinear coordinate transform}
Let us recall the nonlinear change of coordinate from \cite{BM1}, namely $(t,x,y)\to (t,z,v)$: 
\beq\label{eq: coor-change}
\begin{split}
v(t,y)&=u(y)+\f{1}{t}\int_0^t<-\pa_y\psi>(t',y)\chi_1(y)dt'\\
&=u(y)-\f{\chi_1(y)}{t}\int_0^t\f{1}{2\pi}\int_{\bbT}\pa_y\psi(t',x,y)dxdt',\\
z(t,x,y)&=x-tv(t,y).
\end{split}
\eeq
Thus $Ran\, u=Ran\, v=[u(0),u(1)]$. This is the nonlinear version of \eqref{eq: linear-coordinate}. 

Let us also define
\begin{align*}
&\udl{\pa_tv}(t,v)=\pa_tv(t,y),\  
\udl{\pa_yv}(t,v)=\pa_yv(t,y),\ \udl{\pa_{yy}v}(t,v)=\pa_{yy}v(t,y),\\
&\udl{u''}(t,v)=u''(y),\ \udl{u'}(t,v)=u'(y). 
\end{align*}
Here and below, for any $\varphi(y)$, $\udl{\varphi}(t,v(t,y))=\varphi(y)$. 

Define $\Om(t,z,v)=\om(t,x,y)$ and $\Psi(t,z,v)=\psi(t,x,y)$, hence the original 2D Euler system \eqref{eq:vorticity} is expressed as
\beq\label{eq: Om}
\left\{
\begin{aligned}
&\pa_t\Om-\udl{u''}\pa_zP_{\neq}\Psi
+\mathrm{U}\cdot\na_{z,v}\Om=0,\\
&\Delta_t\Psi(t,z,v)\eqdef \pa_{zz}\Psi+(\udl{\pa_yv})^2(\pa_v-t\pa_z)^2\Psi+\udl{\pa_{yy}v}(\pa_v-t\pa_z)\Psi=\Om(t,z,v),
\end{aligned}\right.
\eeq 
where $\mathrm{U}=(0,\udl{\pa_tv})+\udl{\pa_yv}\Upsilon\na^{\bot}_{v,z}P_{\neq}(\Psi\Upsilon)$. 
Here we use a smoother cut-off function $\Upsilon(v)\in \mathcal{G}^{M(\th_0),\f{s+1}{2}}_{ph,1}$ which satisfies $\mathrm{supp}\, \Upsilon(v)\subset \left[u(\f{\th_0}{2}),u(1-\f{\th_0}{2})\right]$ and
$
\Upsilon(v)\equiv 1,\  \text{for }v\in [u({\th_0}),u(1-\th_0)]
$
to replace $\tchi_2$. 

We introduce $\udl{\Delta_t^{-1}}\Om=\Upsilon\Delta_t^{-1}(\Upsilon\Om)$, then by the compact support of $\Om$, we have
\beno
\Psi\Upsilon=\udl{\Delta_t^{-1}}\Om.
\eeno

Now let us introduce the good unknown
\beq\label{eq: def-f}
\begin{split}
f(t,z,v)
&\eqdef<\Om>(t,v)+P_{|k|\geq k_{M}}\Om(t,z,v)\\
&\quad+\sum_{0<|k|<k_{M}}\bfD_{u,k}(\mathcal{F}_1{\Om}(t,k,\cdot))(t,k,v)e^{izk},
\end{split}
\eeq
where 
\beno
P_{|k|\geq k_{M}}\Om(t,z,v)=\sum_{|k|\geq k_{M}}\mathcal{F}_1\Om(t,k,v)e^{ikz}.
\eeno
Here $k_{M}$ is a large constant depending only on the background shear flow $u(y)$ and will be determined in the proof.

\begin{remark}\label{Rmk: wave bdd-A}
It holds that
\beno
\Om(t,z,v)=<f>(t,v)+P_{|k|\geq k_{M}}f(t,z,v)+\sum_{0<|k|<k_{M}}\bfD_{u,k}^{-1}(\mathcal{F}_1{f}(t,k,\cdot))(t,k,v)e^{izk}.
\eeno
which gives us that there exist $C=C(k_M)\geq 1$ independent of $t$ such that for any $\la<\la_{\mathcal{D}}$,
\begin{align*}
{C}^{-1}\|\Om\|_{\mathcal{G}^{\la,\s;s}(\mathbb{T}\times [u(0),u(1)])}
\leq \|f\|_{\mathcal{G}^{\la,\s;s}(\mathbb{T}\times [u(0),u(1)])}
\leq C\|\Om\|_{\mathcal{G}^{\la,\s;s}(\mathbb{T}\times [u(0),u(1)])},
\end{align*}
and that there exist $C=C(k_M)\geq 1$ independent of $t$  such that for any $\la(t)<\la_{\mathcal{D}}$,
\beno
C^{-1}\|\rmA\Om\|_2^2\leq \|\rmA f\|_2^2\leq C\|\rmA\Om\|_2^2. 
\eeno
\end{remark}
The proof of the last inequality is similar to the proof of the {\it product Lemma 3.8 and (3.40)} in \cite{BM1}, which follows from {\it Lemma 3.5 and 3.4} in \cite{BM1}, we omit the details. 

An easy calculation gives that
\begin{align*}
\pa_tf
&=\pa_tP_0\Om+\pa_tP_{k\geq k_M}\Om\\
&\quad+\pa_t\sum_{0<|k|<k_{M}}\bfD_{u,k}(\mathcal{F}_1{\Om}(t,k,\cdot))(t,k,v)e^{iz k}\\
&=-P_{0}\left(\rmU\cdot\na_{z,v}\Om\right)
+\udl{u''}\pa_{z}P_{k\geq k_M}\big(\underline{\Delta_t^{-1}}\Om\big)\\
&\quad-P_{k\geq k_M}\Big(\rmU\cdot\na_{z,v}\Om\Big)\\
&\quad+\sum_{0<|k|<k_{M}}[\pa_t,\bfD_{u,k}]\left(\mathcal{F}_{1}\Om(t,k,\cdot)\right)(t,k,v)e^{iz k}\\
&\quad+\sum_{0<|k|<k_{M}}\bfD_{u,k}\left(\mathcal{F}_{1}\Big(\udl{u''}\pa_{z}\big(\Delta_t^{-1}\Om\big)\Big)\right)(t,k,v)e^{iz k}\\
&\quad-\sum_{0<|k|<k_{M}}\bfD_{u,k}\left(\mathcal{F}_{1}\Big(\rmU\cdot\na_{z,v}\Om\Big)\right)(t,k,v)e^{iz k}.
\end{align*}
By the fact that,
\begin{align*}
&\bfD_{u,k}\left(\mathcal{F}_{1}\Big(\udl{u''}\pa_{z}\big(\Delta_t^{-1}\Om\big)\Big)\right)(t,k,v)\\
&=\bfD_{u,k}\left(\mathcal{F}_{1}\Big(\widetilde{u''}\pa_{z}\big(\Delta_u^{-1}\Om\big)\Big)\right)(t,k,v)
+\bfD_{u,k}\left(\mathcal{F}_{1}\Big(\udl{u''}\pa_{z}\big(\Delta_t^{-1}\Om\big)-\widetilde{u''}\pa_{z}\big(\Delta_u^{-1}\Om\big)\Big)\right)(t,k,v)\\
&=\bfD_{u,k}\left(\mathcal{F}_{1}\Big(\widetilde{u''}\pa_{z}\big(\Delta_u^{-1}\Om\big)\Big)\right)(t,k,v)
+\bfD_{u,k}\left(\mathcal{F}_{1}\Big(\big(\udl{u''}-\widetilde{u''}\big)\pa_{z}\big(\Delta_t^{-1}\Om\big)-\widetilde{u''}\pa_{z}\big(\Delta_u^{-1}\Om-\Delta_t^{-1}\Om\big)\Big)\right)(t,k,v)\\
&=-ikv\bfD_{u,k}\left(\mathcal{F}_{1}\Om(t,k,\cdot)\right)(t,k,v)
+\bfD_{u,k}\left(ik\mathcal{F}_{1}v\Om(t,k,\cdot)\right)(t,k,v)\\
&\quad+\bfD_{u,k}\left(\mathcal{F}_{1}\Big(\big(\udl{u''}-\widetilde{u''}\big)\pa_{z}\big(\Delta_t^{-1}\Om\big)-\widetilde{u''}\pa_{z}\big(\Delta_u^{-1}\Om-\Delta_t^{-1}\Om\big)\Big)\right)(t,k,v),
\end{align*}
we obtain by \eqref{eq:[pa_t,D]} that
\beq\label{eq:f}
\begin{split}
\pa_tf
&=-P_{0}\left(\rmU\cdot\na_{z,v}\Om\right)
+\udl{u''}\pa_{z}P_{|k|\geq k_M}\big(\underline{\Delta_t^{-1}}\Om\big)\\
&\quad-P_{|k|\geq k_M}\Big(\rmU\cdot\na_{z,v}\Om\Big)\\
&\quad+\sum_{0<|k|<k_{M}}\bfD_{u,k}\left(\mathcal{F}_{1}\Big(\big(\udl{u''}-\widetilde{u''}\big)\pa_{z}\big(\Delta_t^{-1}\Om\big)-\widetilde{u''}\pa_{z}\big(\Delta_u^{-1}\Om-\Delta_t^{-1}\Om\big)\Big)\right)(t,k,v)e^{izk}\\
&\quad-\sum_{0<|k|<k_{M}}\bfD_{u,k}\left(\mathcal{F}_{1}\Big(\rmU\cdot\na_{z,v}\Om\Big)\right)(t,k,v)e^{iz k}.
\end{split}
\eeq
We also have 
\begin{align*}
\pa_tv(t,y)=-\f{\chi_1(y)}{t^2}\int_0^t<-\pa_y\psi>(t',y)dt'+\f{\chi_1(y)}{t}<-\pa_y\psi>(t,y),
\end{align*}
and
\begin{align*}
\pa_{tt}v(t,y)=\pa_t\big(\udl{\pa_tv}(t,v(t,y))\big)=\pa_t\big(\udl{\pa_tv}\big)(t,v(t,y))+\pa_tv(t,y)\pa_v\big(\udl{\pa_tv}\big)(t,v(t,y))
\end{align*}
which gives that
\begin{align*}
\pa_{tt}v(t,y)&=\f{2\chi_1(y)}{t^3}\int_0^t<-\pa_y\psi>(t',y)dt'-\f{2\chi_1(y)}{t^2}<-\pa_y\psi>(t,y)+\f{\chi_1(y)}{t}<\pa_tU^x>(t,y)\\
&=-\f{2}{t}\pa_tv(t,y)+\f{\chi_1(y)}{t}<\pa_tU^x>(t,y).
\end{align*}
By using the fact that
\begin{align*}
\pa_tU^x+u(y)\pa_xU^x+u''(y)U^y+\pa_xp+U^x\pa_xU^x+U^y\pa_yU^x=0,
\end{align*}
we obtain that
\begin{align}\label{eq: v_t}
\pa_t(\udl{\pa_tv})+\udl{\pa_tv}\pa_v(\udl{\pa_tv})+\f{2}{t}(\udl{\pa_tv})=-\f{\udl{\chi_1}}{t}\udl{\pa_yv}<\na_{z,v}^{\bot}P_{\neq}(\Psi\Upsilon)\cdot\na_{z,v}\widetilde{U^x}>,
\end{align}
where $\widetilde{U^x}(t,z(t,x,y),v(t,y))=U^x(t,x,y)$. 

We also have 
\begin{align*}
\pa_yv(t,y)=u'(y)-\f{1}{t}\int_0^t<\om>(t',y)dt',
\end{align*}
then it holds that
\begin{align*}
\pa_t\Big(t\big(\pa_yv(t,y)-u'(y)\big)\Big)=-<\om>(t,y). 
\end{align*}
Let $h(t,v(t,y))=\pa_yv(t,y)-u'(y)$, then $h(t,v)$ satisfies
\beq\label{eq: h}
\pa_th+\udl{\pa_tv}\pa_vh=\f{1}{t}\big(-P_0\Om-h\big)\eqdef \bar{h}=\udl{\pa_yv}\pa_v\udl{\pa_tv}.
\eeq
By \eqref{eq: v_t} and \eqref{eq: Om}, we have
\begin{align}\label{eq: bar{h}}
\pa_t\bar{h}=-\f{\bar{h}}{t}-\f{1}{t}(\pa_tP_0\Om+\pa_th)
=-\f{\bar{2h}}{t}-\udl{\pa_tv}\pa_v\bar{h}+\f{\udl{\pa_yv}}{t}<\na^{\bot}_{v,z}P_{\neq}(\Psi\Upsilon)\cdot\na_{z,v}\Om>.
\end{align}
\begin{remark}
Recall that $\udl{\varphi}(t,v(t,y))=\varphi(y)$ and $\widetilde{\varphi}(u(y))=\varphi(y)$, then $\udl{\varphi}(t,v)$ satisfies the following transport equation
\beno
\pa_t\udl{\varphi}(t,v)+\udl{\pa_tv}(t,v)\pa_v\udl{\varphi}(t,v)=0,
\eeno
and $\varphi_{\d}=\udl{\varphi}(t,v)-\widetilde{\varphi}(v)$ satisfies
\beno
\pa_t\varphi_{\d}+\udl{\pa_tv}(t,v)\pa_v\varphi_{\d}=-\udl{\pa_tv}(t,v)\pa_v\widetilde{\varphi}(v). 
\eeno
\end{remark}
Thus we have 
\begin{align}
\label{eq: chi-chi}
&\big(\pa_t+\udl{\pa_tv}(t,v)\pa_v\big)(\udl{\chi_1}-\tchi_1)=-\udl{\pa_tv}(t,v)\pa_v\tchi_1(v),\\ 
\label{eq: u'-u'}
&\big(\pa_t+\udl{\pa_tv}(t,v)\pa_v\big)(\udl{u'}-\widetilde{u'})=-\udl{\pa_tv}(t,v)\pa_v\widetilde{u'}(v),\\ 
\label{eq: u''-u''}
&\big(\pa_t+\udl{\pa_tv}(t,v)\pa_v\big)(\udl{u''}-\widetilde{u''})=-\udl{\pa_tv}(t,v)\pa_v\widetilde{u''}(v). 
\end{align}
Note that on the right hand side, the functions $\pa_v\tchi_1(v)$, $\pa_v\widetilde{u'}(v)$ and $\pa_v\widetilde{u''}(v)$ only depend on the background $u$ and are smoother than $\udl{\pa_tv}(t,v)$. 

\subsection{Main energy estimate}
Let us now introduce the main energy for large time:
\beq\label{eq: main-energy}
\rmE(t)=\f12\left\|\rmA(t,\na)f\right\|_2^2+\rmE_{v}(t)+\rmE_d(t)
\eeq
and for some constants $\rmK_v$, $\rmK_{\rmD}$ depending only on $s$, $\la_M$, $\la_{\infty}'$ fixed by the proof
\beq\label{eq: energy-E_v}
\rmE_{v}(t)=\langle t\rangle^{2+2s}\left\|\f{\rmA}{\langle\pa_v\rangle^s}\bar{h}(t)\right\|_2^2+\langle t\rangle^{4-\rmK_{\rmD}\ep}\left\|\udl{\pa_tv}(t)\right\|_{\mathcal{G}^{\la(t),\s-6}}^2+\f{1}{\rmK_v}\left\|\rmA^{\mathrm{R}}h(t)\right\|_2^2.
\eeq
and 
\beq\label{eq: energy-E_d}
\rmE_{d}(t)=\f{1}{\rmK_v}\Big(
\left\|\rmA^{\mathrm{R}}(\udl{\chi_1}-\tchi_1)\right\|_2^2
+\left\|\rmA^{\mathrm{R}}(\udl{u'}-\widetilde{u'})\right\|_2^2
+\left\|\rmA^{\mathrm{R}}(\udl{u''}-\widetilde{u''})\right\|_2^2\Big).
\eeq
We refer to {\it Lemma 2.1} in \cite{BM1} for the local wellposedness for 2D Euler in Gevrey spaces also see \cite{BB1977}. By this way, we may safely ignore the time interval $[0,1]$ by further restricting the size of the initial data.  

The goal is next to prove by a continuity argument that this energy $\rmE(t)$ (together with some related quantities) is uniformly bounded for all time if $\ep$ is sufficiently small. 
 
We define the following controls referred to in the sequel as the bootstrap hypotheses for $t\geq 1$ and some constant $C(k_{M})$,
\begin{itemize}
\item[(B1)]  $\rmE(t)\leq 10C(k_{M})\ep^2$; 
\item[(B2)] Compact support of $\Om$: $\mathrm{supp}\,\Om\subset \Big[2.5\th_0-10C(k_{M})\ep,1-2.5\th_0+10C(k_{M})\ep\Big]$;
\item[(B3)] `$\mathrm{CK}$' integral estimates
\begin{align*}
\int_{1}^t\bigg[&\mathrm{CK}_{\la}(\tau)+\mathrm{CK}_{w}(\tau)+\mathrm{CK}_{g}(\tau)+\mathrm{CK}_w^{v,2}(\tau)+\mathrm{CK}_{\la}^{v,2}(\tau)\\
&+\rmK_{v}^{-1}\left(\mathrm{CK}_{w}^{v,1}(\tau)+\mathrm{CK}_{\la}^{v,1}(\tau)\right)
+\rmK_{v}^{-1}\sum_{i=1}^2\left(\mathrm{CCK}_{w}^{i}(\tau)+\mathrm{CCK}_{\la}^{i}(\tau)\right)\bigg]d\tau
\leq 20C(k_{M})\ep^2.
\end{align*}
\end{itemize}
The $\mathrm{CK}$ terms above arise from the time derivates of $\rmA(t)$ and $\rmA^{\rmR}(t)$ and are naturally controlled by the energy estimates we are making. See \eqref{eq:CK_la}, \eqref{eq:CK_w}, \eqref{eq:CCK^i} for the definitions.

\begin{proposition}[Bootstrap]\label{prop: btsp}
Let $k_{M}$ be fixed and sufficiently large. There exists an $\ep_0\in (0,\f12)$ depending only on $k_{M}$, $\la_{0}$, $\la'$, $s$ and $\s$ such that if $\ep<\ep_0$ and on $[1, T^*]$ the bootstrap hypotheses (B1)-(B3) hold, then for $\forall t\in [1, T^*]$,
\begin{itemize}
\item[1.]  $\rmE(t)\leq 8C(k_{M})\ep^2$, 
\item[2.] Compact support of $\Om$: $\mathrm{supp}\,\Om\subset \Big[2.5\th_0-8C(k_{M})\ep,1-2.5\th_0+8C(k_{M})\ep\Big]$,
\item[3.] and the $\mathrm{CK}$ controls satisfy:
\begin{align*}
\int_{1}^t\bigg[&\mathrm{CK}_{\la}(\tau)+\mathrm{CK}_{w}(\tau)+\mathrm{CK}_{g}(\tau)+\mathrm{CK}_w^{v,2}(\tau)+\mathrm{CK}_{\la}^{v,2}(\tau)\\
&+\rmK_{v}^{-1}\left(\mathrm{CK}_{w}^{v,1}(\tau)+\mathrm{CK}_{\la}^{v,1}(\tau)\right)
+\rmK_{v}^{-1}\sum_{i=1}^2\left(\mathrm{CCK}_{w}^{i}(\tau)+\mathrm{CCK}_{\la}^{i}(\tau)\right)
\bigg]d\tau\leq 18C(k_{M})\ep^2,
\end{align*}
\end{itemize}
form which it follows that $T^*=+\infty$. 
\end{proposition}
The remainder of this section is devoted to the proof of Proposition \ref{prop: btsp}, the primary step being to show that on $[1, T^*]$, we have
\begin{align}
\label{eq:energy-goal}
\rmE(t)+\f12\int_{1}^t\bigg[&\mathrm{CK}_{\la}(\tau)+\mathrm{CK}_{w}(\tau)+\mathrm{CK}_{g}(\tau)+\mathrm{CK}_w^{v,2}(\tau)+\mathrm{CK}_{\la}^{v,2}(\tau)\\
\nonumber
&+\rmK_{v}^{-1}\left(\mathrm{CK}_{w}^{v,1}(\tau)+\mathrm{CK}_{\la}^{v,1}(\tau)\right)
+\rmK_{v}^{-1}\sum_{i=1}^2\left(\mathrm{CCK}_{w}^{i}(\tau)+\mathrm{CCK}_{\la}^{i}(\tau)\right)\bigg]d\tau\\
\nonumber
&\leq \rmE(1)+C(k_{M})\ep^2+K\ep^3. 
\end{align}
for some constant $K$ which is independent of $\ep$ and $T^*$. If $\ep$ is sufficiently small then \eqref{eq:energy-goal} implies Proposition \ref{prop: btsp}. Indeed, the control of the compact support of $\Om$ is based on a standard argument using the Lagrangian coordinates and the inviscid damping result. The proof can be found at the end of the section. 

It is natural to compute the time evolution of $\rmE(t)$. Indeed we have
\begin{align}\label{eq:energy2}
&\f{1}{2}\f{d}{dt}\left\|\rmA(t,\na)f\right\|_2^2\\
\nonumber
&=-\mathrm{CK}_{\la}-\mathrm{CK}_{w}
-\f{1}{2\pi}\int\rmA_0(\eta)\bar{\widehat{\Om}}_0(\eta)
\rmA_0(\eta)\mathcal{F}_2P_0\left(\rmU\cdot\na_{z,v}\Om\right)d\eta\\
\nonumber
&\quad+\sum_{|k|\geq k_M}\f{1}{2\pi}\int\rmA_k(\eta)\bar{\hat{f}}_k(\eta)\rmA_k(\eta)\widehat{\left(\udl{u''}\pa_{z}\big(\underline{\Delta_t^{-1}}f\big)\right)}_k(\eta)d\eta\\
\nonumber
&\quad-\sum_{|k|\geq k_{M}}\f{1}{2\pi}\int\rmA_k(\eta)\bar{\widehat{\Om}}_k(\eta)\rmA_k(\eta)\widehat{\Big(\rmU\cdot\na_{z,v}\Om\Big)}_k(\eta)d\eta\\
\nonumber
&\quad+\sum_{0<|k|<k_{M}}\f{1}{2\pi}\int\rmA_k(\eta)\bar{\hat{f}}_k(\eta)\rmA_k(\eta)\mathcal{F}_2\Big(ik\bfD_{u,k}\left(\mathcal{F}_{1}\Big(\big(\udl{u''}-\widetilde{u''}\big)\big(\Delta_t^{-1}\Om\big)\Big)\right)(t,k,\cdot)\Big)d\eta\\
\nonumber
&\quad-\sum_{0<|k|<k_{M}}\f{1}{2\pi}\int\rmA_k(\eta)\bar{\hat{f}}_k(\eta)\rmA_k(\eta)\mathcal{F}_2\Big(ik\bfD_{u,k}\left(\mathcal{F}_{1}\Big(\widetilde{u''}\big(\Delta_u^{-1}\Om-\Delta_t^{-1}\Om\big)\Big)\right)(t,k,\cdot)\Big)d\eta\\
\nonumber
&\quad-\sum_{0<|k|<k_{M}}\f{1}{2\pi}\int\rmA_k(\eta)\bar{\hat{f}}_k(\eta)\rmA_k(\eta)\mathcal{F}_2\Big(\bfD_{u,k}\left(\mathcal{F}_{1}\Big(\rmU\cdot\na_{z,v}\Om\Big)\right)(t,k,\cdot)\Big)_k(\eta)d\eta\\
\label{eq:energy-2}&=-\mathrm{CK}_{\la}-\mathrm{CK}_{w}
+\f{1}{2\pi}\Pi^{tr1}+\f{1}{2\pi}\Pi^{u''}+\f{1}{2\pi}\Pi^{tr2}+\f{1}{2\pi}\Pi^{u'',\ep1}+\f{1}{2\pi}\Pi^{u'',\ep2}+\f{1}{2\pi}\Pi,
\end{align}
where the $\mathrm{CK }$ stands for `Cauchy-Kovalevskaya' and expressed as
\begin{align}
\label{eq:CK_la}
&\mathrm{CK}_{\la}=-\dot{\la}(t)\left\||\na|^{s/2}\rmA f\right\|_2^2,\\
\label{eq:CK_w}
&\mathrm{CK}_{w}=\sum_k\int\f{\pa_tw_k(t,\eta)}{w_k(t,\eta)}e^{\la(t)|k,\eta|^s}\langle k,\eta\rangle^{\s}\f{e^{\mu|\eta|^{1/2}}}{w_k(t,\eta)}\rmA_k(t,\eta)|\widehat{f}_k(t,\eta)|^2d\eta.
\end{align}
We define
\begin{align}
&\tilde{\rmJ}_k(t,\eta)=\f{e^{\mu|\eta|^{1/2}}}{w_k(t,\eta)},\\
&\tilde{\rmA}_k(t,\eta)=e^{\la(t)|k,\eta|^s}\langle k,\eta\rangle^{\s}\tilde{\rmJ}_k(t,\eta).
\end{align}

\begin{remark}
Proposition \ref{prop: kernel-wave-op} gives us that for $\la_0$ and $\la'$ sufficiently small, 
\beno
-\dot{\la}(t)\left\||\na|^{s/2}\rmA \Om\right\|_2^2\lesssim \mathrm{CK}_{\la},\quad
\left\|\sqrt{\f{\pa_tw}{w}}\rmA\Om\right\|_2^2
\lesssim \mathrm{CK}_{w}+\mathrm{CK}_{\la}. 
\eeno
\end{remark}
\begin{proof}
The first inequality is easy to prove. One can refer to the proof of {\it (3.39a)} in \cite{BM1} for more details. 

One can also follow the proof of {\it (3.39b)} in \cite{BM1} to prove the second inequality. Here, we want to point out the difference. 
By applying Proposition \ref{prop: kernel-wave-op}, we get that
\begin{align*}
\left|\sqrt{\f{\pa_tw_k(\eta)}{w_k(\eta)}}\rmA_k(\eta)\widehat{\Om}_k(\eta)\right|
&\lesssim \sum_{\rmM\geq 8}\int_{\R}\sqrt{\f{\pa_tw_k(\eta)}{w_k(\eta)}}\rmA_k(\eta)(e^{-\la_{\mathcal{D}}|\eta-\xi|^{s_0}})_{\rmM}|\hat{f}_k(\xi)_{<\rmM/8}|d\xi\\
&\quad+\sum_{\rmM\geq 8}\int_{\R}\sqrt{\f{\pa_tw_k(\eta)}{w_k(\eta)}}\rmA_k(\eta)(e^{-\la_{\mathcal{D}}|\eta-\xi|^{s_0}})_{<\rmM/8}|\hat{f}_k(\xi)_{\rmM}|d\xi\\
&\quad+\sum_{\rmM'\in\mathcal{D}}\sum_{\f18\rmM'\rmM\leq 8\rmM'}\int_{\R}\sqrt{\f{\pa_tw_k(\eta)}{w_k(\eta)}}\rmA_k(\eta)(e^{-\la_{\mathcal{D}}|\eta-\xi|^{s_0}})_{\rmM'}|\hat{f}_k(\xi)_{\rmM}|d\xi\\
&=I_{HL}+I_{LH}+I_{R}. 
\end{align*}
One can follow the same argument as in the proof of {\it (3.39b)} in \cite{BM1} and obtain the estimates of $I_{LH}$ and $I_{R}$. To treat $I_{HL}$, we use the fact that 
\beno
\sqrt{\f{\pa_tw_k(\eta)}{w_k(\eta)}}\lesssim 1_{t\leq 2|\eta|}\lesssim \f{|\eta|}{t},\ \text{for}\ t\geq 1,
\eeno
which implies that
\beno
\|I_{HL}\|_{L^2}^2\lesssim \f{\|f\|_{\mathcal{G}^{c\la,\s;s}}^2}{\langle t\rangle^2}\lesssim \mathrm{CK}_{\la}. 
\eeno
Thus we complete the proof of the remark. 
\end{proof}

Now we control \eqref{eq:energy-2}. The most difficult part is to treat the $\Pi$ term. The key idea is to extract the transport structure from $\Pi$ which can be treated with $\Pi^{tr1}$ and $\Pi^{tr2}$. That is 
\begin{align*}
\Pi=\Pi^{tr3}+\Pi^{com},\quad \text{also see \eqref{eq: Pi^com}.}
\end{align*}
with 
\beno
\Pi^{tr3}=-\sum_{0<|k|< k_{M}}\int\rmA_k(\eta)\bar{\widehat{\Om}}_k(\eta)\rmA_k(\eta)\widehat{\Big(\rmU\cdot\na_{z,v}\Om\Big)}_k(\eta)d\eta.
\eeno

The estimate of $\Pi^{tr}=\Pi^{tr1}+\Pi^{tr2}+\Pi^{tr3}$ is similar to the {\it (2.23)} in \cite{BM1}. We start with integrating by parts, 
\beno
\f{1}{2\pi}\Pi^{tr}=\int \rmA\left(\mathrm{U}\cdot\na_{z,v}\Om\right)\rmA \Om dx
=-\f12\int \na\cdot\mathrm{U}|\rmA\Om|^2dx
+\int\rmA\Om\left[\rmA(\mathrm{U}\cdot\na_{z,v}\Om)-\mathrm{U}\cdot\na_{z,v}\rmA\Om\right]dx.
\eeno
By Sobolev embedding, $\s>5$ and the bootstrap hypotheses, we have 
\ben\label{eq: 2.25}
\left|\int \na\cdot\mathrm{U}|\rmA\Om|^2dx\right|
\leq \|\na\mathrm{U}\|_{\infty}\|\rmA\Om\|_2^2\lesssim \f{\ep}{\langle t\rangle^{2-\rmK_{\rmD}\ep/2}}\|\rmA\Om\|_2^2\lesssim \f{\ep^3}{\langle t\rangle^{2-\rmK_{\rmD}\ep/2}}.
\een

Similar to the proof in \cite{BM1}, we use a paraproduct decomposition to deal with the commutator. 
\begin{align*}
&\mathrm{T}_{\rmN}=2\pi \int \rmA \Om \left[\rmA\left(\mathrm{U}_{<\rmN/8}\cdot\na_{z,v}\Om_{\rmN}\right)-\mathrm{U}_{<\rmN/8}\cdot\na_{z,v}\rmA\Om_{\rmN}\right] dx\\
&\mathrm{R}_{\rmN}=2\pi \int \rmA \Om \left[\rmA\left(\mathrm{U}_{\rmN}\cdot\na_{z,v}\Om_{<\rmN/8}\right)-\mathrm{U}_{\rmN}\cdot\na_{z,v}\rmA\Om_{<\rmN/8}\right] dx\\
&\mathcal{R}=2\pi\sum_{\rmN\in \bf{D}}\sum_{\f18\rmN\leq \rmN'\leq 8\rmN}\int \rmA \Om \left[\rmA\left(\mathrm{U}_{\rmN}\cdot\na_{z,v}\Om_{\rmN'}\right)-\mathrm{U}_{\rmN}\cdot\na_{z,v}\rmA\Om_{\rmN'}\right] dx.
\end{align*}

The treatment of the transport term is similar to {\it section 5} in \cite{BM1}. 
\begin{proposition}[Transport]\label{prop: transport}
Under the bootstrap hypotheses, 
\beno
\sum_{\rmN\geq 8}|\rmT_{\rmN}|\lesssim \ep \mathrm{CK}_{\la}+\ep\mathrm{CK}_{w}+\f{\ep^3}{\langle t\rangle^{2-\rmK_{\rmD}\ep/2}}.
\eeno
\end{proposition}

To obtain the estimate of the reaction term, one can follow step by step the proof in {\it Section 6} of \cite{BM1}. We need to mention that we should replace $\udl{\pa_yv}-1$ in \cite{BM1} by $\udl{\pa_yv}-\widetilde{u'}$ and we will use the fact that $\udl{\pa_yv}-\widetilde{u'}=h+(\udl{u'}-\widetilde{u'})$. One may refer to \cite{BM1} for more details. 
\begin{proposition}[Reaction]\label{prop:reaction}
Under the bootstrap hypotheses, 
\beq\label{eq:reaction}
\begin{split}
\sum_{\rmN\geq 8}|\rmR_{\rmN}|\lesssim 
&\ep \mathrm{CK}_{\la}+\ep\mathrm{CK}_{w}+\f{\ep^3}{\langle t\rangle^{2-\rmK_{\rmD}\ep/2}}
+\ep\mathrm{CK}_{\la}^{v,1}+\ep\mathrm{CK}_{w}^{v,1}\\
&+\ep\left\|\left\langle\f{\pa_v}{t\pa_z}\right\rangle^{-1}(\pa_z^2+(\pa_v-t\pa_z)^2)\left(\f{|\na|^{\f{s}{2}}}{\langle t\rangle^s}\rmA+\sqrt{\f{\pa_t w}{w}}\tilde{\rmA}\right)P_{\neq}(\Psi\Upsilon)\right\|_2^2. 
\end{split}
\eeq
\end{proposition}

The elliptic estimate given below is slightly different from \cite{BM1}. We will regard $\Delta_t$ as a perturbation of $\Delta_u$ instead of $\Delta_L$ in \cite{BM1}: 
\begin{proposition}[Precision elliptic control] \label{prop:elliptic}
Under the bootstrap hypotheses,
\beq\label{eq:elliptic}
\begin{split}
&\left\|\left\langle\f{\pa_v}{t\pa_z}\right\rangle^{-1}(\pa_z^2+(\pa_v-t\pa_z)^2)\left(\f{|\na|^{\f{s}{2}}}{\langle t\rangle^s}\rmA+\sqrt{\f{\pa_t w}{w}}\tilde{\rmA}\right)P_{|k|\geq k_{M}}(\Psi\Upsilon)\right\|_2^2\\
&\leq C\Big[ \mathrm{CK}_{\la}+\mathrm{CK}_{w}
+\ep^2\Big(\sum_{i=1}^2\mathrm{CCK}_{\la}^i+\mathrm{CCK}^i_{w}\Big)\Big],
\end{split}
\eeq
where the `coefficient Cauchy-Kovalevskaya' terms are given by
\beq\label{eq:CCK^i}
\begin{split}
\mathrm{CCK}_{\la}^1&
=-\dot{\la}(t)\left\||\pa_v|^{s/2}\rmA^{\mathrm{R}}\left((u'\circ u^{-1})^2-(\udl{\pa_yv})^2\right)\right\|_2^2,\\
\mathrm{CCK}_{w}^1&
=\left\|\sqrt{\f{\pa_tw}{w}}\rmA^{\mathrm{R}}\left((u'\circ u^{-1})^2-(\udl{\pa_yv})^2\right)\right\|_2^2,\\
\mathrm{CCK}_{\la}^2&
=-\dot{\la}(t)\left\||\pa_v|^{s/2}\f{\rmA^{\mathrm{R}}}{\langle\pa_v\rangle}\left((u''\circ u^{-1})-\udl{\pa_{yy}v}\right)\right\|_2^2,\\
\mathrm{CCK}_{w}^2&
=\left\|\sqrt{\f{\pa_tw}{w}}\f{\rmA^{\mathrm{R}}}{\langle\pa_v\rangle}\left((u''\circ u^{-1})-\udl{\pa_{yy}v}\right)\right\|_2^2.
\end{split}
\eeq
The constant $C$ is independent of $k_{M}$. \\
Generally, it holds that
\beq\label{eq:elliptic}
\begin{split}
&\left\|\left\langle\f{\pa_v}{t\pa_z}\right\rangle^{-1}(\pa_z^2+(\pa_v-t\pa_z)^2)\left(\f{|\na|^{\f{s}{2}}}{\langle t\rangle^s}\rmA+\sqrt{\f{\pa_t w}{w}}\tilde{\rmA}\right)P_{\neq}(\Psi\Upsilon)\right\|_2^2\\
&\lesssim_{k_{M}} \mathrm{CK}_{\la}+\mathrm{CK}_{w}
+\ep^2\Big(\sum_{i=1}^2\mathrm{CCK}_{\la}^i+\mathrm{CCK}^i_{w}\Big),
\end{split}
\eeq
\end{proposition}
The proof is in Section \ref{Sec: elliptic}.

The estimate of the remainder term is similar to the proof of {\it Proposition 2.6} in \cite{BM1}: 
\begin{proposition}[Remainders]\label{prop: remainder}
Under the bootstrap hypotheses,
\beno
\mathcal{R}\lesssim \f{\ep^3}{\langle t\rangle^{2-\rmK_{\rmD}\ep/2}}.
\eeno
\end{proposition}

The next 3 terms are new. The $\Pi^{com}$ comes from the application of the wave operator to the nonlinear terms. 
\begin{proposition}\label{prop:com}
Under the bootstrap hypotheses, 
\beno
\begin{split}
|\Pi^{com}|\lesssim 
&\ep \mathrm{CK}_{\la}+\ep\mathrm{CK}_{w}+\f{\ep^3}{\langle t\rangle^{2-\rmK_{\rmD}\ep/2}}
+\ep\mathrm{CK}_{\la}^{v,1}+\ep\mathrm{CK}_{w}^{v,1}\\
&+\ep\left\|\left\langle\f{\pa_v}{t\pa_z}\right\rangle^{-1}(\pa_z^2+(\pa_v-t\pa_z)^2)\left(\f{|\na|^{\f{s}{2}}}{\langle t\rangle^s}\rmA+\sqrt{\f{\pa_t w}{w}}\tilde{\rmA}\right)P_{\neq}(\Psi\Upsilon)\right\|_2^2. 
\end{split}
\eeno
\end{proposition}
The term $\mathrm{II}^{u''}$ was not present in \cite{BM1}. The key observation for this term is as follows: 1. we can treat it as the reaction term; 2. we can get an extra smallness if $|k|\geq k_{M}$ with $k_{M}$ sufficiently large.
\begin{proposition}\label{prop: nonlocal}
Under the bootstrap hypotheses, for any $\kappa_0>0$ there exists $k_{M}$ such that 
\begin{align*}
|\Pi^{u''}|
\leq \kappa_0&\bigg(\f{\ep^2}{\langle t\rangle^2}
+\mathrm{CK}_{\la}+\mathrm{CK}_w
+ \left\|\left\langle\f{\pa_v}{t\pa_z}\right\rangle^{-1}\f{|\na|^{s/2}}{\langle t\rangle^{s}}\Delta_{L}P_{|k|\geq k_{M}}\rmA (\Psi\Upsilon)\right\|_2^2\\
&\quad +\left\|1_{R}\sqrt{\f{\pa_tw}{w}}\Delta_{L}\tilde{\rmA}P_{|k|\geq k_{M}}(\Psi\Upsilon)\right\|_2^2\bigg). 
\end{align*}
\end{proposition} 
The treatment of the nonlocal error terms $\Pi^{u'',\ep_1}$ and $\Pi^{u'',\ep_2}$ is also inspired from the reaction term in {\it section 6} of \cite{BM1}. 
\begin{proposition}\label{prop:nonlocal-ep}
Under the bootstrap hypotheses, 
\beno
\begin{split}
|\Pi^{u'',\ep_1}|+|\Pi^{u'',\ep_2}|\lesssim 
&\ep \mathrm{CK}_{\la}+\ep\mathrm{CK}_{w}+\f{\ep^3}{\langle t\rangle^{2-\rmK_{\rmD}\ep/2}}
+\ep\mathrm{CK}_{\la}^{v,1}+\ep\mathrm{CK}_{w}^{v,1}\\
&+\ep\left\|\left\langle\f{\pa_v}{t\pa_z}\right\rangle^{-1}(\pa_z^2+(\pa_v-t\pa_z)^2)\left(\f{|\na|^{\f{s}{2}}}{\langle t\rangle^s}\rmA+\sqrt{\f{\pa_t w}{w}}\tilde{\rmA}\right)P_{\neq}(\Psi\Upsilon)\right\|_2^2. 
\end{split}
\eeno
\end{proposition}

The next step in the bootstrap is to provide good estimates on the coordinate system. 
\begin{proposition}[Coordinate system controls]\label{prop: coordinate} Under the bootstrap hypotheses, for $\ep$ sufficiently small and $\rmK_v$ sufficiently large there is a $\rmK>0$ such that
\begin{align}
\label{(2.30a)}
&\|\rmA^{\rmR}h\|_2^2+\rmE_d(t)
+\f12\int_{1}^t\sum_{i=1}^2\mathrm{CCK}_{w}^i(\tau)d\tau
+\f12\int_{1}^t\sum_{i=1}^2\mathrm{CCK}_{\la}^i(\tau)d\tau\leq \f{1}{2}\rmK_v\ep^2,\\
\label{(2.30b)}
&\langle t\rangle^{2+2s}\left\|\f{\rmA}{\langle\pa_v\rangle^s}\bar{h}\right\|_2^2
+\f12\int_{1}^t\mathrm{CK}_{\la}^{v,2}(\tau)+\mathrm{CK}_{w}^{v,2}(\tau)d\tau\leq C(k_{M})\ep^2+\rmK\ep^3,\\
\label{(2.30c)}
&\langle t\rangle^{4-\rmK_{\rmD}\ep}\|\udl{\pa_tv}\|_{\mathcal{G}^{\la(t),\s-6;s}}^2\leq C(k_{M})\ep^2+\rmK\ep^3,\\
\label{(2.30d)}
&\int_{1}^t\mathrm{CK}_{\la}^{v,1}(\tau)+\mathrm{CK}_{w}^{v,1}(\tau)d\tau\leq \f{1}{2}\rmK_{v}\ep^2,
\end{align}
where the $\mathrm{CK}^{v,i}$ terms are given by
\beq\label{eq:2.31}
\begin{split}
&\mathrm{CK}_{w}^{v,2}(\tau)=\langle\tau\rangle^{2+2s}\left\|\sqrt{\f{\pa_tw}{w}}\f{\rmA}{\langle\pa_v\rangle^s}\bar{h}\right\|_2^2,\\
&\mathrm{CK}_{\la}^{v,2}(\tau)=\langle\tau\rangle^{2+2s}(-\dot{\la}(\tau))\left\||\pa_v|^s\f{\rmA}{\langle\pa_v\rangle^s}\bar{h}\right\|_2^2,\\
&\mathrm{CK}_{w}^{v,1}(\tau)=\langle\tau\rangle^{2+2s}\left\|\sqrt{\f{\pa_tw}{w}}\f{\rmA}{\langle\pa_v\rangle^s}\udl{\pa_tv}\right\|_2^2,\\
&\mathrm{CK}_{\la}^{v,1}(\tau)=\langle\tau\rangle^{2+2s}(-\dot{\la}(\tau))\left\||\pa_v|^s\f{\rmA}{\langle\pa_v\rangle^s}\udl{\pa_tv}\right\|_2^2. 
\end{split}
\eeq
\end{proposition}
Compared to {\it Proposition 2.5} in \cite{BM1}, the main difference is the estimate of $\rmE_d(t)$. 
Note that the equations of $\udl{\chi_1}-\tchi_1$, $\udl{u'}-\widetilde{u'}$ and $\udl{u''}-\widetilde{u''}$ \eqref{eq: chi-chi}-\eqref{eq: u''-u''} have the same transport structure as the equation of $h$:
\beq
(\pa_t+\udl{\pa_tv}\pa_v)\mathcal{U}=\mathcal{F},
\eeq
where $\mathcal{F}$ represents one of the forcing terms $\udl{\pa_tv}\pa_v\tchi_1$, $\udl{\pa_tv}\pa_v\widetilde{u'}$ or $\udl{\pa_tv}\pa_v\widetilde{u''}$. Due to the fact that $\tchi_1$, $\widetilde{u'}$ and $\widetilde{u''}$ are smoother given functions compared to the solutions, $\mathcal{F}$ has the same behavior as $\udl{\pa_tv}$ which is better than the force term $\bar{h}=\udl{\pa_yv}\pa_v\udl{\pa_tv}$ in the equation of $h$ \eqref{eq: h}. By following the estimate of $h$ in {\it Section 8.2} of \cite{BM1}, one can obtain the estimate of $\rmE_d(t)$ in \eqref{(2.30a)} easily. We omit the proof of this proposition. 
\begin{remark}\label{Rmk: 2identity}
It holds that
\beno
&&(\widetilde{u'})^2-(\udl{\pa_yv})^2
=-2(\widetilde{u'})\Upsilon(v)\left(h+(\udl{u'}-\widetilde{u'})\right)-\left(h+(\udl{u'}-\widetilde{u'})\right)^2,\\
&&\widetilde{u''}-\udl{\pa_{yy}v}
=\left(\widetilde{u''}-\udl{u''}\right)-\udl{\pa_yv}\Upsilon(v)\pa_vh. 
\eeno
\end{remark}

\subsection{Conclusion of proof}
Now, we prove the compact support result in Proposition \ref{prop: btsp} and conclude the proof. Under the bootstrap assumption, we have
\beno
\Delta\psi(t,x,y)=\om(t,x,y)=\Om(t,x-tv(t,y),v(t,y)), \quad 
\psi(x,0)=\psi(x,1)=0. 
\eeno
Taking the Fourier transform in $x$, we get
\ben\label{eq:psi}
\mathcal{F}_1\psi(t,k,y)=\int_0^1G_k(y,z)\Om(t,k,v(t,z))e^{-iktv(t,z)}dz,
\een
and
\ben\label{eq:u}
\pa_y\mathcal{F}_1\psi(t,k,y)=\int_0^1\pa_yG_k(y,z)\Om(t,k,v(t,z))e^{-iktv(t,z)}dz.
\een
Integrating by parts in $z$ in the identities \eqref{eq:psi}(twice) and \eqref{eq:u} (once), we get that
\ben\label{eq: inviscid-damping}
|U^y(t,x,y)|\lesssim \f{\ep}{\langle t\rangle^2},\quad |U^x-<U^x>|\lesssim \f{\ep}{\langle t\rangle}. 
\een
Then let us define
\beno
\left\{\begin{aligned}
&\f{d(X^1,X^2)}{dt}(t,x,y)=(u(y)+U^x(t,X^1,X^2),U^y(t,X^1,X^2)),\\
&(X^1,X^2)\big|_{t=0}=(x,y).\\
\end{aligned}\right.
\eeno
Then $\f{d}{dt}\om(t,X^1,X^2)=(u''\pa_x\psi)(t,X_1,X_2)$ and $|X^2(t,x,y)-y|\leq C\ep$, thus the for $y\notin [4\th_0-C\ep, 1-4\th_0+C\ep]$, $\f{d}{dt}\om(t,X^1,X^2)=0$. Thus $\om(t)$ will always be away from the boundary. 

This ends the proof of the bootstrap Proposition \ref{prop: btsp}.

Let us now prove \eqref{eq: Scattering}, applying the same method in {\it Section 2.4} of \cite{BM1}, we get for $\om_1(t,z,y)=\Om(t,z,v)=\om(t,x,y)$ and $\psi_1(t,z,y)=\Psi(t,z,v)=\psi(t,x,y)$, 
\beno
\pa_t\om_1+\na^{\bot}_{z,y}P_{\neq}\psi_1\cdot\na_{z,y}\om_1-u''\pa_z\psi_1=0. 
\eeno
Then we can define
\beno
f_{\infty}=\om_1(1)-\int_{1}^{\infty}\na^{\bot}_{z,y}P_{\neq}\psi_1(s)\cdot\na_{z,y}\om_1(s)ds+\int_{1}^{\infty}u''\pa_z\psi_1(s)ds. 
\eeno
Therefore there exists $f_{\infty}$ such that, 
\beno
\left\|\om(t,x+tu(y)+\Phi(t,y)\chi_1(y),y)-f_{\infty}(x,y)\right\|_{\mathcal{G}^{\la_{\infty}}}\lesssim \f{\ep}{\langle t\rangle}.
\eeno
Note that $\om(t,x,y)=0$ for $y\notin [2\th_0,1-2\th_0]$ so is $f_{\infty}$, thus we can replace $\Phi(t,y)\chi_1(y)$ by $\Phi(t,y)$ which gives \eqref{eq: Scattering}. 

Now, we prove \eqref{eq: inviscid damping}. We have 
\beno
-\pa_y<U^x>=<\om>,
\quad
\pa_t<U^x>+<U^y\om>=0,
\eeno
which implies that $<U^x>(t,y)=C_U$ for any $t\geq 1$ and $y\in [0,2\th_0]\cup [1-2\th_0,1]$. 
Let $\tilde{U}^x(t,z,y)=U^x(t,x,y)$, then $\Big(<\tilde{U}^x>(t,y)-C_{U}\Big)$ has the same compact support as $\om_1$ and satisfies
\begin{align*}
\pa_t\Big(<\tilde{U}^x>(t,y)-C_{U}\Big)
&=-<\na^{\bot}_{z,y}P_{\neq }\psi_1\cdot\na_{z,y}\tilde{U}^x>\chi_1(y)\\
&=-<-\pa_yP_{\neq }(\Upsilon\psi_1)\pa_z(\pa_y-t\pa_yv\pa_z)(\Upsilon\psi_1)+\pa_zP_{\neq }(\Upsilon\psi_1)\pa_y\pa_z(\Upsilon\psi_1)>. 
\end{align*}
Then \eqref{eq: inviscid damping} follows from the fact that 
\ben\label{eq:decay-psi_1}
\|\Upsilon\psi_1\|_{\mathcal{G}^{\la_{\infty}',2;s}}\lesssim \f{\ep}{\langle t\rangle^2}. 
\een
This ends the proof of Theorem \ref{Thm: main}. 

\subsection{Proof of Corollary \ref{Rmk: better-scattering}}
Now let us start the proof of Corollary \ref{Rmk: better-scattering}. Let $\rmF(t,z,v)=(\bbD_{u}\om)(t,x+tu(y)+\Phi(t,y),v(t,y))=(\bbD_{u}\om)(t,x+tv(t,y),y)$ with $z=x-tv(t,y)$, then 
\beno
\rmF(t,z,v)=<\om_1>+\sum_{k\neq 0}^{\infty}\bbD_{u,k}(\mathcal{F}_1{\om}_k)(t,k,y)e^{ikx-iktv(t,y)}
\eeno 
and
\beq\label{eq:f}
\begin{split}
\pa_t\rmF
&=-P_{0}\left(\rmU\cdot\na_{z,v}\Om\right)\\
&\quad+\sum_{|k|>0}\bfD_{u,k}\left(\mathcal{F}_{1}\Big(\big(\udl{u''}-\widetilde{u''}\big)\pa_{z}\big(\Delta_t^{-1}\Om\big)-\widetilde{u''}\pa_{z}\big(\Delta_u^{-1}\Om-\Delta_t^{-1}\Om\big)\Big)\right)(t,k,v)e^{izk}\\
&\quad-\sum_{|k|>0}\bfD_{u,k}\left(\mathcal{F}_{1}\Big(\rmU\cdot\na_{z,v}\Om\Big)\right)(t,k,v)e^{iz k}.
\end{split}
\eeq
By \eqref{eq:decay-psi_1} and the uniform $L^2$ to $L^2$ boundedness of $\bfD_{u,k}$, we have that
\begin{align*}
&\left\|\sum_{|k|>0}\bfD_{u,k}\left(\mathcal{F}_{1}\Big(\big(\udl{u''}-\widetilde{u''}\big)\pa_{z}\big(\Delta_t^{-1}\Om\big)-\widetilde{u''}\pa_{z}\big(\Delta_u^{-1}\Om-\Delta_t^{-1}\Om\big)\Big)\right)(t,k,v)e^{izk}\right\|_2^2\\
&\lesssim \sum_{|k|>0}\left\|\mathcal{F}_{1}\Big(\big(\udl{u''}-\widetilde{u''}\big)\pa_{z}\big(\Delta_t^{-1}\Om\big)-\widetilde{u''}\pa_{z}\big(\Delta_u^{-1}\Om-\Delta_t^{-1}\Om\big)\Big)_k\right\|_2^2\lesssim \f{\ep^2}{\langle t\rangle^2}, \\
&\left\|\sum_{|k|>0}\bfD_{u,k}\left(\mathcal{F}_{1}\Big(\rmU\cdot\na_{z,v}\Om\Big)\right)(t,k,v)e^{iz k}\right\|_2^2
\lesssim \sum_{|k|>0}\left\|\mathcal{F}_{1}\Big(\rmU\cdot\na_{z,v}\Om\Big)_k\right\|_2^2
\lesssim \f{\ep^2}{\langle t\rangle^2}.
\end{align*}
Let us define $\rmf_{\infty}(t,z,y)=\rmF_{\infty}(t,z,v)$ with
\begin{align*}
\rmF_{\infty}(z,v)
&=\rmF(1,z,v)-\int_1^{\infty}P_{0}\left(\rmU\cdot\na_{z,v}\Om\right)(\tau)d\tau\\
&\quad+\int_1^{\infty}\sum_{|k|>0}\bfD_{u,k}\left(\mathcal{F}_{1}\Big(\big(\udl{u''}-\widetilde{u''}\big)\pa_{z}\big(\Delta_t^{-1}\Om\big)-\widetilde{u''}\pa_{z}\big(\Delta_u^{-1}\Om-\Delta_t^{-1}\Om\big)\Big)\right)(\tau,k,v)e^{izk}d\tau\\
&\quad-\int_1^{\infty}\sum_{|k|>0}\bfD_{u,k}\left(\mathcal{F}_{1}\Big(\rmU\cdot\na_{z,v}\Om\Big)\right)(\tau,k,v)e^{iz k}d\tau.
\end{align*}
Together with \eqref{eq: inviscid damping}, it gives Corollary \ref{Rmk: better-scattering}. 

\section{Elliptic estimate}\label{Sec: elliptic}
The purpose of this section is to provide a thorough analysis of $\udl{\Delta_u^{-1}}$ and $\udl{\Delta_t^{-1}}$. 
\subsection{Basic estimate of the Green's function}
In this section, we study the operator $\udl{\Delta_u^{-1}}$ in the linear change of coordinates \eqref{eq: linear-coordinate}. Here, 
\beno
\psi(t,x,y)=\tpsi(x-tu(y),u(y)),\quad 
\om(t,x,y)=\tom(x-tu(y),u(y)),
\eeno 
hence,  $\Delta\psi=\om\chi_2$ is equivalent to $\Delta_u\tpsi=\tom\tchi_2$. 

Therefore we get that for $k\neq 0$,
\begin{align*}
\mathcal{F}_1{\psi}(k,y_1)=\int_0^1G_k(y_1,y_2)\chi_2(y_2)\mathcal{F}_1{\om}(k,y_2)dy_2,
\end{align*}
and for $k=0$,
\begin{align*}
\mathcal{F}_1{\psi}(0,y_1)=\int_0^1G_0(y_1,y_2)\chi_2(y_2)\mathcal{F}_1{\om}(0,y_2)dy_2,
\end{align*}
where 
\beno
G_k(y_1,y_2)=\f{1}{k\sinh k}\times\left\{\begin{aligned}
&\sinh(k(1-y_2))\sinh ky_1,\quad y_1\leq y_2,\\
&\sinh(ky_2)\sinh(k(1-y_1)),\quad y_1\geq y_2,
\end{aligned}\right.
\eeno
and
\beno
G_0(y_1,y_2)=\left\{\begin{aligned}
&(1-y_2)y_1,\quad y_1\leq y_2,\\
&y_2(1-y_1),\quad y_1\geq y_2,
\end{aligned}\right.
\eeno

Let $\tG_k(u(y_1),u(y_2))=G_k(y_1,y_2)$, then we have
\begin{align*}
&\mathcal{F}_1\udl{\Delta_u^{-1}}\tom
=\mathcal{F}_1(\tchi_2\Delta_u^{-1}(\tchi_2\tom))\\
&=\int_{u(0)}^{u(1)}\tG_k(u_1,u_2)(u^{-1})'(u_2)\tchi_2(u_1)\tchi_2(u_2)e^{-ikt(u_2-u_1)} \mathcal{F}_1{\tom}(k,u_2)du_2. 
\end{align*}
Let us denote the Fourier transform of the $t$-independent part of the kernel by 
\beno
\mathbf{G}(k,\xi_1,\xi_2)=\f{1}{2\pi}\int_{\R^2}\tG_k(u_1,u_2)(u^{-1})'(u_2)\tchi_2(u_1)\tchi_2(u_2)e^{-iu_1\xi_1-iu_2\xi_2}du_1du_2.
\eeno
Thus we have 
\ben\label{eq: Laplace-inver-form}
\widehat{\udl{\Delta_u^{-1}}\tom}(t,k,\xi_1)=\f{1}{2\pi}\int_{\R}\mathbf{G}(k, \xi_1-kt,kt-\xi_2)\widehat{\tom}(t, k,\xi_2)d\xi_2.
\een
We have the following lemma. 
\begin{lemma}[\cite{Jia2}] \label{eq: kernel-elliptic}
Suppose $u$ and $\chi_2$ are defined as above with $s_0\geq s$, then there exist $\la_{\Delta}=\la_{\Delta}(K,\th_0,s_0)>0$, $C>0$ such that for all $k$,
\beno
\left|\mathbf{G}(k,\xi_1,\xi_2)\right|\leq C\min\left\{\f{e^{-\la_{\Delta}\langle \xi_1+\xi_2\rangle^{s_0}}}{1+k^2+\xi_1^2},\f{e^{-\la_{\Delta}\langle \xi_1+\xi_2\rangle^{s_0}}}{1+k^2+\xi_2^2}\right\}.
\eeno
\end{lemma}
\begin{remark}\label{Rmk: Laplace}
To treat $\udl{\Delta_t^{-1}}$, we also need another Green's function $\mathbf{G}_2(k,\xi_1,\xi_2)$ with a different cut-off function $\Upsilon$, 
\beno
\mathbf{G}_2(k,\xi_1,\xi_2)=\f{1}{2\pi}\int_{\R^2}\tG_k(u_1,u_2)(u^{-1})'(u_2)\Upsilon(u_1)\Upsilon(u_2)e^{-iu_1\xi_1-iu_2\xi_2}du_1du_2.
\eeno
Thus we have 
\ben\label{eq: Laplace-inver-form-2}
\mathcal{F}_2\mathcal{F}_1\Big(\Upsilon\Delta_u^{-1}(\Upsilon\tom)\Big)(t,k,\xi_1)=\f{1}{2\pi}\int_{\R}\mathbf{G}_2(k, \xi_1-kt,kt-\xi_2)\widehat{\tom}(t, k,\xi_2)d\xi_2.
\een
Moreover, there exist $\la_{\Delta}=\la_{\Delta}(K,\th_0,s_0)>0$, $C>0$ such that for all $k$,
\beno
\left|\mathbf{G}_2(k,\xi_1,\xi_2)\right|\leq C\min\left\{\f{e^{-\la_{\Delta}\langle \xi_1+\xi_2\rangle^{s_0}}}{1+k^2+\xi_1^2},\f{e^{-\la_{\Delta}\langle \xi_1+\xi_2\rangle^{s_0}}}{1+k^2+\xi_2^2}\right\}.
\eeno
\end{remark}
The key idea of the proof is to use the definition of Gevrey class in physical side (see \cite{Jia2}). 

\subsubsection{Lossy estimate}
The following is the fundamental estimate on $\udl{\Delta_u^{-1}}$ or $\udl{\Delta_t^{-1}}$ which allows to trade the regularity of $\tom$ in a high norm for decay of the stream function in a slightly lower norm. By the estimate of the kernel $\mathbf{G}$ and $\mathbf{G}_2$, $\Delta_{L}\udl{\Delta_u^{-1}}\tom$ has the same regularity as $\tom$, where
\ben
\Delta_{L}=\pa_{zz}+(\pa_v-t\pa_{z})^2. 
\een 
We have the following lemmas.
\begin{lemma}\label{lem: lossy-elliptic u}
Under the same assumption of Lemma \ref{eq: kernel-elliptic}, it holds that
\ben\label{eq: lossy-elliptic u1}
\|P_{\neq}\pa_{z}^2\udl{\Delta_u^{-1}}\tom\|_{\mathcal{G}^{\la(t),\s-2;s}}\lesssim \f{\|\tom\|_{\mathcal{G}^{\la(t),\s;s}}}{1+t^2},
\een
and
\ben\label{eq: lossy-elliptic u2}
\|\rmA P_{\neq}\pa_{z}^2\udl{\Delta_u^{-1}}\tom\|_2+\|\rmA P_{0}\langle\pa_u\rangle^2\udl{\Delta_u^{-1}}\tom\|_{L^2}\lesssim \|\rmA \tom\|_{L^2}.
\een
\end{lemma}
As can be easily seen from examining $\Delta_u$, we cannot expect to gain $O(t^{-2})$ decay without paying two derivatives. Notice that since the coefficient $\udl{\pa_{yy}v}$ effectively contains a derivative on $\Om$, the estimate below loses three derivatives. 
\begin{lemma}\label{lem: lossy-elliptic t}
Under the bootstrap hypotheses, for $\ep$ sufficiently small, 
\beno
\|P_{\neq}(\Upsilon\Psi)\|_{\mathcal{G}^{\la(t),\s-3;s}}=\|P_{\neq}\udl{\Delta_t^{-1}}\Om\|_{\mathcal{G}^{\la(t),\s-3;s}}\lesssim \f{\|\Om\|_{\mathcal{G}^{\la(t),\s-1;s}}}{1+t^2}.
\eeno
\end{lemma}
Let us now start to prove the lemmas.
\begin{proof}
\no{\bf Proof of Lemma \ref{lem: lossy-elliptic u}. }
Omitting the time-dependence in $\la$, $\tom$, for $\la$ chosen sufficiently small than $\la_{\Delta}$, by {\it Lemma  A.2} in \cite{BM1}, we obtain that
\begin{align}\label{eq:Delta_L-est}
&\|\Delta_L\udl{\Delta_u^{-1}}\tom\|_{\mathcal{G}^{\la,\s;s}}\\
\nonumber&\lesssim \sum_{k}\left\|(k^2+(\xi-kt)^2)e^{\la|k,\xi|^s}\langle k,\xi\rangle^{\s}\int_{\R}\mathbf{G}(k,\xi-kt,kt-\xi_2)\widehat{\tom}_k(\xi_2)d\xi_2\right\|_{L^2_{\xi}}^2\\
\nonumber&\lesssim \sum_{k}\left\|\int_{\R}e^{-\la_{\Delta}|\xi-\xi_2|^{s_0}}e^{(1+c)\la|\xi-\xi_2|^s}e^{\la|k,\xi_2|^s}\langle k,\xi_2\rangle^{\s}|\widehat{\tom}_k(\xi_2)|d\xi_2\right\|_{L^2_{\xi}}^2\lesssim \|\tom\|_{\mathcal{G}^{\la,\s;s}},
\end{align}
which also implies the first part of \eqref{eq: lossy-elliptic u2}.  
Next by the same argument as {\it (4.3)} in \cite{BM1}, we have that
\begin{align*}
\|P_{\neq}\pa_z^2\udl{\Delta_u^{-1}}\tom\|_{\mathcal{G}^{\la,\s-2;s}}\lesssim \f{1}{t^2+1}\|\Delta_LP_{\neq}\udl{\Delta_u^{-1}}\tom\|_{\mathcal{G}^{\la,\s;s}}, 
\end{align*}
which gives \eqref{eq: lossy-elliptic u1}.

Similarly, we have for any $\s_1\leq \s$
\beq\label{eq:elliptic-G_2}
\|P_{\neq}\Upsilon\Delta_u^{-1}(\Upsilon\tom)\|_{\mathcal{G}^{\la,\s_1-2;s}}\lesssim \f{1}{t^2+1}\|P_{\neq}\tom\|_{\mathcal{G}^{\la,\s_1;s}}. 
\eeq

To prove \eqref{eq: lossy-elliptic u2}, one can follow the proof in \eqref{eq:Delta_L-est} and obtain that
\begin{align}\label{eq:Delta_L-est-1}
&\|\langle\pa_u^2\rangle P_0\udl{\Delta_u^{-1}}\tom\|_{\mathcal{G}^{\la,\s;s}}\\
\nonumber&\lesssim \sum_{k}\left\|(1+\xi^2)\f{\rmA_0(\xi)}{\rmA_0(\xi_2)}\int_{\R}\mathbf{G}(k,\xi-kt,kt-\xi_2)\rmA_0(\xi_2)\widehat{\tom}_0(\xi_2)d\xi_2\right\|_{L^2_{\xi}}^2\\
\nonumber&\lesssim \sum_{k}\left\|\int_{\R}e^{-\la_{\Delta}|\xi-\xi_2|^{s_0}}e^{(1+c)\la|\xi-\xi_2|^s}\rmA_0(\xi_2)|\widehat{\tom}_0(\xi_2)|d\xi_2\right\|_{L^2_{\xi}}^2\lesssim \|\rmA P_0\tom\|_{L^2}.
\end{align} 
Thus we proved the lemma. 

\no{\bf Proof of Lemma \ref{lem: lossy-elliptic t}. }
Now we write $\Delta_t$ as a perturbation of $\Delta_{u}$. Recall $\Delta_t\Psi=\Om$ and
\begin{align}\label{eq: Delta_uPsi}
\Delta_u\Psi
&=\Om+h_1(t,v)\Upsilon(v)(\pa_v-t\pa_z)^2(\Upsilon\Psi)+h_2(t,v)\Upsilon(v)(\pa_v-t\pa_z)(\Upsilon\Psi),
\end{align}
where 
\ben\label{eq: h_1,h_2}
h_1(t,v)=(\widetilde{u'})^2-(\udl{\pa_yv})^2,
\quad
h_2(t,v)=\widetilde{u''}-\udl{\pa_{yy}v}.
\een
Here we use the fact that $v(t,y)\equiv u(y)$ for $y\notin [\th_0,1-\th_0]$, which implies that $v^{-1}(t,c)\equiv u^{-1}(c)$ for $c\notin [u(\th_0), u(1-\th_0)]$ and then 
\beq\label{eq: cut-off,h_1h_2}
h_j(t,v)=h_j(t,v)\Upsilon(v), \  h_j(t,v)\pa_v\Upsilon(v)=0,\ h_j(t,v)\pa_{vv}\Upsilon(v)=0\quad j=1,2.
\eeq
Therefore, 
\begin{align*}
\|\Delta_u\Psi\|_{\mathcal{G}^{\la,\s-1;s}}
&\lesssim \|\Om\|_{\mathcal{G}^{\la,\s-1;s}}
+\|h_1\|_{\mathcal{G}^{\la,\s-1;s}}\|\Delta_L(\Upsilon\Psi)\|_{\mathcal{G}^{\la,\s-1;s}}\\
&\quad+\|h_2\|_{\mathcal{G}^{\la,\s-1;s}}\|(\pa_v-t\pa_z)(\Upsilon\Psi)\|_{\mathcal{G}^{\la,\s-1;s}}\\
&\lesssim \|\Om\|_{\mathcal{G}^{\la,\s-1;s}}
+\ep\|\Delta_u\Psi\|_{\mathcal{G}^{\la,\s-1;s}}.
\end{align*}
Together with \eqref{eq:elliptic-G_2}, this implies Lemma \ref{lem: lossy-elliptic t} provided that $\ep$ is sufficiently small. 
\end{proof}

\subsection{Precision elliptic control}\label{sec:Precision elliptic control}
Now we turn to the proof of Proposition \ref{prop:elliptic}. 
By using \eqref{eq: Delta_uPsi}, we get that
\begin{align}\label{eq: PsiUpsilon}
\widehat{\Psi\Upsilon}(k,\xi_1)&=\int_{\R}\mathbf{G}_2(k, \xi_1-kt,kt-\xi_2)\widehat{\Om}(k,\xi_2)d\xi_2\\
\nonumber&\quad-\int_{\R^2}\mathbf{G}_2(k, \xi_1-kt,kt-\xi_2)\widehat{h_1}(t,\xi_2-\xi_3)(\xi_3-kt)^2\widehat{\Psi\Upsilon}(k,\xi_3)d\xi_3d\xi_2\\
\nonumber&\quad+i\int_{\R^2}\mathbf{G}_2(k, \xi_1-kt,kt-\xi_2)\widehat{h_2}(t,\xi_2-\xi_3)(\xi_3-kt)\widehat{\Psi\Upsilon}(k,\xi_3)d\xi_3d\xi_2. 
\end{align}
Define the multipliers
\beno
&&\mathcal{M}_1(t,k,\xi)=\left\langle\f{\xi}{kt}\right\rangle^{-1}\f{|k,\xi|^{s/2}}{\langle t\rangle^{s}}\rmA P_{\neq},\\
&&\mathcal{M}_2(t,k,\xi)=\left\langle\f{\xi}{kt}\right\rangle^{-1}\sqrt{\f{\pa_tw}{w}}\tilde{\rmA} P_{\neq}.
\eeno
By Lemma \ref{eq: kernel-elliptic}, we have 
\begin{align*}
&\sum_{i=1,2}\left\|\mathcal{M}_i(k^2+(\xi_1-kt)^2)\int_{\R}\mathbf{G}_2(k, \xi_1-kt,kt-\xi_2)\widehat{\Om}(k,\xi_2)d\xi_2\right\|_2^2\\
&\lesssim \left\|\int_{\R}e^{-\la_{\Delta}\langle\xi_1-\xi_2\rangle^{s_0}}\f{\rmA_k(\xi_1)}{\rmA_k(\xi_2)}\left(\f{|k,\xi_1|}{|k,\xi_2|}\right)^{\f{s}{2}}\mathcal{M}_1\widehat{\Om}(k,\xi_2)d\xi_2\right\|_2^2\\
&\quad+\left\|\int_{\R}e^{-\la_{\Delta}\langle\xi_1-\xi_2\rangle^{s_0}}\tilde{\rmA}_k(\xi_1)\sqrt{\f{\pa_tw_k(t,\xi_1)}{w_k(t,\xi_1)}}\widehat{\Om}(k,\xi_2)d\xi_2\right\|_2^2
\end{align*}
By {\it Lemma 3.6 and Lemma A.2} in \cite{BM1}, we have
\begin{align*}
\left\|\mathcal{M}_1(k^2+(\xi_1-kt)^2)\int_{\R}\mathbf{G}_2(k, \xi_1-kt,kt-\xi_2)\widehat{\Om}(k,\xi_2)d\xi_2\right\|_2^2
\lesssim \left\|\mathcal{M}_1\Om\right\|_2^2
\lesssim \f{1}{\langle t\rangle^{2s}}\left\||\na|^{s/2}\rmA\Om\right\|_2^2.
\end{align*}
By {\it Lemma 3.3, Remark 8, Lemma 3.4 and the proof of Lemma 3.6} in \cite{BM1}, we have 
\begin{align*}
&\left\|\int_{\R}e^{-\la_{\Delta}\langle\xi_1-\xi_2\rangle^{s_0}}\tilde{\rmA}_k(\xi_1)\sqrt{\f{\pa_tw_k(t,\xi_1)}{w_k(t,\xi_1)}}\widehat{\Om}(k,\xi_2)d\xi_2\right\|_2^2\\
&\lesssim \left\|\int_{\R}e^{-\la_{\Delta}\langle\xi_1-\xi_2\rangle^{s_0}}\langle\xi_1-\xi_2\rangle^{\f12}\tilde{\rmA}_k(\xi_1)\sqrt{\f{\pa_tw_k(t,\xi_2)}{w_k(t,\xi_2)}}\widehat{\Om}(k,\xi_2)d\xi_2\right\|_2^2\\
&\quad+\left\|\int_{\R}e^{-\la_{\Delta}\langle\xi_1-\xi_2\rangle^{s_0}}\langle\xi_1-\xi_2\rangle\tilde{\rmA}_k(\xi_1)\f{|\xi_2|^{s/2}}{\langle t\rangle^s}\widehat{\Om}(k,\xi_2)d\xi_2\right\|_2^2, 
\end{align*}
and
\begin{align*}
\f{\tilde{\rmJ}_k(\xi_1)}{\tilde{\rmJ}_k(\xi_2)}\lesssim e^{10\mu |\xi_1-\xi_2|^{1/2}}.
\end{align*}
Thus we get that
\beq\label{eq: inverse-Lap-est}
\sum_{i=1,2}\left\|\mathcal{M}_i(k^2+(\xi_1-kt)^2)\int_{\R}\mathbf{G}_2(k, \xi_1-kt,kt-\xi_2)\widehat{\Om}(k,\xi_2)d\xi_2\right\|_2^2
\lesssim \sum_{i=1,2}\left\|\mathcal{M}_i\Om\right\|_2^2.
\eeq
Define,
\ben\label{eq: T^1,T^2}
\mathrm{T}^1=h_1(\pa_v-t\pa_z)^2(\Psi\Upsilon),\quad
\mathrm{T}^2=h_2(\pa_v-t\pa_z)(\Psi\Upsilon).
\een
and divide each via a paraproduct decomposition in the $v$ variable only 
\begin{align}
\label{eq:T_1}\mathrm{T}^1&=\sum_{M\geq 8}(h_1)_{M}(\pa_v-t\pa_z)^2(\Psi\Upsilon)_{<M/8}
+\sum_{M\geq 8}(h_1)_{<M/8}(\pa_v-t\pa_z)^2(\Psi\Upsilon)_M\\
\nonumber
&\quad+\sum_{M\in \mathcal{D}}\sum_{\f18M\leq M'\leq 8M}(h_1)_{M'}(\pa_v-t\pa_z)^2(\Psi\Upsilon)_{M}\\
\nonumber
&=\mathrm{T}^1_{\mathrm{HL}}+\mathrm{T}^1_{\mathrm{LH}}+\mathrm{T}^1_{\mathcal{R}},\\
\label{eq:T_2}\mathrm{T}^2&=\sum_{M\geq 8}(h_2)_{M}(\pa_v-t\pa_z)(\Psi\Upsilon)_{<M/8}
+\sum_{M\geq 8}(h_2)_{<M/8}(\pa_v-t\pa_z)(\Psi\Upsilon)_M\\
\nonumber&\quad+\sum_{M\in \mathcal{D}}\sum_{\f18M\leq M'\leq 8M}(h_2)_{M'}(\pa_v-t\pa_z)(\Psi\Upsilon)_{M}\\
\nonumber&=\mathrm{T}^2_{\mathrm{HL}}+\mathrm{T}^2_{\mathrm{LH}}+\mathrm{T}^2_{\mathcal{R}}.
\end{align}
By following the argument of the {\it proof of Proposition 2.4} in \cite{BM1}, we have 
\begin{align*}
&\left\|\mathcal{M}_1\mathrm{T}^2_{\mathrm{LH}}\right\|_2^2+\left\|\mathcal{M}_1\mathrm{T}^1_{\mathrm{LH}}\right\|_2^2\lesssim \ep^2\left\|\mathcal{M}_1\Delta_{L}P_{\neq}(\Psi\Upsilon)\right\|_2^2\\
&\left\|\mathcal{M}_2\mathrm{T}^2_{\mathrm{LH}}\right\|_2^2+\left\|\mathcal{M}_2\mathrm{T}^1_{\mathrm{LH}}\right\|_2^2\lesssim \ep^2\left\|\mathcal{M}_1\Delta_{L}P_{\neq}(\Psi\Upsilon)\right\|_2^2+\ep^2\left\|\mathcal{M}_2\Delta_{L}P_{\neq}(\Psi\Upsilon)\right\|_2^2,
\end{align*}
and
\begin{align*}
&\left\|\mathcal{M}_1\mathrm{T}^2_{\mathrm{HL}}\right\|_2^2\lesssim \ep^2\left\|\mathcal{M}_1\Delta_{L}(\Psi\Upsilon)\right\|_2^2+\ep^2\mathrm{CCK}_{\la}^2,\\
&\left\|\mathcal{M}_1\mathrm{T}^1_{\mathrm{HL}}\right\|_2^2\lesssim \ep^2\left\|\mathcal{M}_1\Delta_{L}(\Psi\Upsilon)\right\|_2^2+\ep^2\mathrm{CCK}_{\la}^1,\\
&\left\|\mathcal{M}_2\mathrm{T}^1_{\mathrm{HL}}\right\|_2^2+\left\|\mathcal{M}_2\mathrm{T}^2_{\mathrm{HL}}\right\|_2^2\lesssim \ep^2\left\|\mathcal{M}_2\Delta_{L}(\Psi\Upsilon)\right\|_2^2+\ep^2\left\|\mathcal{M}_1\Delta_{L}(\Psi\Upsilon)\right\|_2^2\\
&\qquad\qquad\qquad\qquad\qquad\qquad\ 
+\ep^2\mathrm{CCK}_{\la}^1+\ep^2\mathrm{CCK}_{w}^1+\ep^2\mathrm{CCK}_{\la}^2+\ep^2\mathrm{CCK}_{w}^2,
\end{align*}
and
\begin{align*}
&\left\|\mathcal{M}_1\mathrm{T}_{\mathcal{R}}^1\right\|_2^2\lesssim \ep^2\left\|\mathcal{M}_1\Delta_{L}(\Psi\Upsilon)\right\|_2^2,\\
&\left\|\mathcal{M}_2\mathrm{T}_{\mathcal{R}}^1\right\|_2^2+\left\|\mathcal{M}_2\mathrm{T}_{\mathcal{R}}^2\right\|_2^2\lesssim \ep^2\left\|\mathcal{M}_1\Delta_{L}(\Psi\Upsilon)\right\|_2^2+\ep^2\left\|\mathcal{M}_2\Delta_{L}(\Psi\Upsilon)\right\|_2^2. 
\end{align*}
Together we have 
\begin{align*}
&\left\|\mathcal{M}_1\Delta_{L}(\Psi\Upsilon)\right\|_2^2+\left\|\mathcal{M}_2\Delta_{L}(\Psi\Upsilon)\right\|_2^2\\
&\lesssim \left\|\mathcal{M}_1\Om\right\|_2^2+\left\|\mathcal{M}_2\Om\right\|_2^2+\ep^2\left\|\mathcal{M}_1\Delta_{L}(\Psi\Upsilon)\right\|_2^2+\ep^2\left\|\mathcal{M}_2\Delta_{L}(\Psi\Upsilon)\right\|_2^2\\
&\quad+\ep^2\mathrm{CCK}_{\la}^1+\ep^2\mathrm{CCK}_{w}^1+\ep^2\mathrm{CCK}_{\la}^2+\ep^2\mathrm{CCK}_{w}^2.
\end{align*}
This completes the proof of the Proposition \ref{prop:elliptic}. 

\section{The $\Pi$ term}
In this section, we study $\Pi$ defined in \eqref{eq:energy-2}. For $0<|k|<k_{M}$, let 
\beno
\Pi_k=\int\rmA_k(\eta)\overline{\mathcal{F}_2\bfD_{u,k}(\mathcal{F}_{1}\Om)}_k(\eta)\rmA_k(\eta)\mathcal{F}_2\Big(\bfD_{u,k}\left(\mathcal{F}_{1}\Big(\rmU\cdot\na_{z,v}\Om\Big)\right)(t,k,\cdot)\Big)(\eta)d\eta
\eeno
then
\begin{align*}
\Pi_k&=\int\rmA_k(\eta)\overline{\mathcal{F}_2\bfD_{u,k}(\mathcal{F}_{1}\Om)}_k(\eta)\rmA_k(\eta)\mathcal{F}_2\Big(\bfD_{u,k}\left(\mathcal{F}_{1}\Big(\rmU\cdot\na_{z,v}\Om\Big)\right)(t,k,\cdot)\Big)(\eta)d\eta\\
&=\int\rmA_k(\eta)\overline{\mathcal{F}_2\bfD_{u,k}(\mathcal{F}_{1}\Om)}_k(\eta)\rmA_k(\eta)\mathcal{F}_2\Big(\tchi_2\bfD_{u,k}^1\left(\tchi_2\mathcal{F}_{1}\Big(\rmU\cdot\na_{z,v}\Om\Big)\right)(t,k,\cdot)\Big)(\eta)d\eta\\
&\quad+\int\rmA_k(\eta)\overline{\mathcal{F}_2\bfD_{u,k}(\mathcal{F}_{1}\Om)}_k(\eta)\rmA_k(\eta)
\mathcal{F}_2\Big[\big(\tchi_2\bfD_{u,k}-\tchi_2\bfD_{u,k}^1\big)\left(\tchi_2\mathcal{F}_{1}\Big(\rmU\cdot\na_{z,v}\Om\Big)\right)(t,k,\cdot)\Big](\eta)d\eta\\
&=\int\rmA_k(\eta)\overline{\widehat{\Om}}_k(\eta)\rmA_k(\eta)\widehat{\big(\rmU\cdot\na_{z,v}\Om\big)}_k(\eta)d\eta\\
&\quad+\int(\rmA_k(\eta)-\rmA_k(\xi))\overline{\mathcal{D}(t,k,\eta,\xi)\hat{\Om}_k(\xi)}\rmA_k(\eta)\mathcal{F}_2\Big(\tchi_2\bfD_{u,k}^1\left(\tchi_2\mathcal{F}_{1}\Big(\rmU\cdot\na_{z,v}\Om\Big)\right)(t,k,\cdot)\Big)(\eta)d\xi d\eta\\
&\quad+\sum_{l}\int\overline{\mathcal{F}_2\bfD_{u,k}\big(\mathcal{F}_2^{-1}(\rmA_k(\eta)\widehat{\Om}_k)\big)}(\eta)\\
&\qquad\qquad\qquad\qquad
\times(\rmA_k(\eta)-\rmA_k(\xi))\mathcal{D}^1(t,k,\eta,\xi)\Big(\widehat{\rmU}_{k-l}(\xi-\xi_1)\cdot(il,i\xi_1)\widehat{\Om}_l(\xi_1)\Big)d\xi_1 d\xi d\eta\\
&\quad+\int\rmA_k(\eta)\bar{\hat{f}}_k(\eta)\rmA_k(\eta)\mathcal{F}_2\Big[\big(\tchi_2\bfD_{u,k}-\tchi_2\bfD_{u,k}^1\big)\left(\tchi_2\mathcal{F}_{1}\Big(\rmU\cdot\na_{z,v}\Om\Big)\right)(t,k,\cdot)\Big](\eta)d\eta\\
&=\Pi_k^{tr3}+\Pi_k^{com1}+\Pi_k^{com2}+\Pi_k^{com3}=\Pi_k^{tr3}+\Pi_k^{com},
\end{align*}
and then
\ben\label{eq: Pi^com}
\Pi=-\sum_{0<|k|<k_M}\Pi_k=-\sum_{0<|k|<k_M}\Pi_k^{tr3}-\sum_{0<|k|<k_M}\Pi_k^{com}\eqdef \Pi^{tr3}+\Pi^{com}. 
\een

\subsection{Estimate of $\Pi^{com1}_k$}
By Proposition \ref{prop: kernel-wave-op}, we have
\begin{align*}
|\Pi_k^{com1}|&\lesssim \sum_l\int\Big(\f{\rmA_k(\eta)}{\rmA_k(\xi)}-1\Big)e^{-\la_{\mathcal{D}}|\eta-\xi|^{s_0}}\rmA_k(\xi)|\Om_k(\xi)|\rmA_k(\eta)e^{-\la_{\mathcal{D}}|\eta-\xi_1|^{s_0}}\\
&\qquad \qquad \qquad \qquad \qquad 
\times|\widehat{\rmU}_{k-l}(\xi_1-\xi_2)||l,\xi_2||\widehat{\Om}_l(\xi_2)|d\xi d\xi_2d\xi_1d\eta\\
&\lesssim \sum_l\int 1_{D}\Big(\f{\rmA_k(\eta)}{\rmA_k(\xi)}-1\Big)e^{-\la_{\mathcal{D}}|\eta-\xi|^{s_0}}\rmA_k(\xi)|\Om_k(\xi)|\rmA_k(\eta)e^{-\la_{\mathcal{D}}|\eta-\xi_1|^{s_0}}\\
&\qquad \qquad \qquad \qquad \qquad 
\times|\widehat{\rmU}_{k-l}(\xi_1-\xi_2)||l,\xi_2||\widehat{\Om}_l(\xi_2)|d\xi d\xi_2d\xi_1d\eta\\
&\quad+\sum_l\int 1_{D^c}\Big(\f{\rmA_k(\eta)}{\rmA_k(\xi)}-1\Big)e^{-\la_{\mathcal{D}}|\eta-\xi|^{s_0}}\rmA_k(\xi)|\Om_k(\xi)|\rmA_k(\eta)e^{-\la_{\mathcal{D}}|\eta-\xi_1|^{s_0}}\\
&\qquad \qquad \qquad \qquad \qquad 
\times|\widehat{\rmU}_{k-l}(\xi_1-\xi_2)||l,\xi_2||\widehat{\Om}_l(\xi_2)|d\xi d\xi_2d\xi_1d\eta\\
&=\Pi^{com1}_{k,1}+\Pi^{com1}_{k,2}.
\end{align*}
where $D=\left\{|\eta|\geq 100k_{M}, \ |\eta-\xi|\leq \f{1}{100}|\eta|,\  |\eta-\xi_1|\leq\f{1}{100}|\eta|\right\}$ and $1_{D^c}=1-1_{D}$. 
The second term is easy to deal with. We have that
\begin{align*}
\f{\rmA_k(\eta)}{\rmA_k(\xi)}\lesssim e^{C\la|\eta-\xi|^s},
\end{align*}
then by the fact that $\la$ is chosen much smaller than $\la_{\mathcal{D}}$, we have that for $|k|<k_{M}$, 
\begin{align}\label{eq:Pi_2^com1}
\Pi^{com1}_{k,2}
&\lesssim \|\rmA\Om\|_{L^2}\|\rmU\|_{\mathcal{G}^{c\la,3;s}}\|\Om\|_{\mathcal{G}^{c\la,3;s}}
\lesssim \f{\ep^3}{t^{2-\rmK_{\rmD}\ep/2}},
\end{align}
with $c\in (0,1)$. 

By using the Bony decomposition, we write 
\begin{align*}
\Pi^{com1}_{k,1}
&=\sum_{\rmN\geq 8}\sum_l\int 1_{D}\Big(\f{\rmA_k(\eta)}{\rmA_k(\xi)}-1\Big)e^{-\la_{\mathcal{D}}|\eta-\xi|^{s_0}}\rmA_k(\xi)|\Om_k(\xi)|\rmA_k(\eta)e^{-\la_{\mathcal{D}}|\eta-\xi_1|^{s_0}}\\
&\qquad \qquad \qquad \qquad \qquad 
\times|\widehat{\rmU}_{k-l}(\xi_1-\xi_2)_{<\rmN/8}||l,\xi_2||\widehat{\Om}_l(\xi_2)_{\rmN}|d\xi d\xi_2d\xi_1d\eta\\
&\quad+\sum_{\rmN\geq 8}\sum_l\int 1_{D}\Big(\f{\rmA_k(\eta)}{\rmA_k(\xi)}-1\Big)e^{-\la_{\mathcal{D}}|\eta-\xi|^{s_0}}\rmA_k(\xi)|\Om_k(\xi)|\rmA_k(\eta)e^{-\la_{\mathcal{D}}|\eta-\xi_1|^{s_0}}\\
&\qquad \qquad \qquad \qquad \qquad 
\times|\widehat{\rmU}_{k-l}(\xi_1-\xi_2)_{\rmN}||l,\xi_2||\widehat{\Om}_l(\xi_2)_{<\rmN/8}|d\xi d\xi_2d\xi_1d\eta\\
&\quad+\sum_{\rmN'\in\mathcal{D}}\sum_{\f18\rmN'\leq \rmN\leq 8\rmN'}\sum_l\int 1_{D}\Big(\f{\rmA_k(\eta)}{\rmA_k(\xi)}-1\Big)e^{-\la_{\mathcal{D}}|\eta-\xi|^{s_0}}\rmA_k(\xi)|\Om_k(\xi)|\rmA_k(\eta)e^{-\la_{\mathcal{D}}|\eta-\xi_1|^{s_0}}\\
&\qquad \qquad \qquad \qquad \qquad 
\times|\widehat{\rmU}_{k-l}(\xi_1-\xi_2)_{\rmN'}||l,\xi_2||\widehat{\Om}_l(\xi_2)_{\rmN}|d\xi d\xi_2d\xi_1d\eta\\
&=\Pi^{com1}_{k,1,\rmT}+\Pi^{com1}_{k,1,\rmR}+\Pi^{com1}_{k,1,\mathcal{R}}.
\end{align*}
The $\Pi^{com1}_{k,1,\rmR}$ term is better than the reaction term, since we gain derivatives from $\rmA_{k}(\eta)-\rmA_{k}(\xi)$. We omit the proof and show the result. One can also refer to {\it section 6} in \cite{BM1} for more details. 
\begin{align}\label{eq:com1_1,R}
\Pi^{com1}_{k,1,\rmR}\lesssim &\ep \mathrm{CK}_{\la}+\ep\mathrm{CK}_{w}+\f{\ep^3}{\langle t\rangle^{2-\rmK_{\rmD}\ep/2}}
+\ep\mathrm{CK}_{\la}^{v,1}+\ep\mathrm{CK}_{w}^{v,1}\\
\nonumber
&+\ep\left\|\left\langle\f{\pa_v}{t\pa_z}\right\rangle^{-1}(\pa_z^2+(\pa_v-t\pa_z)^2)\left(\f{|\na|^{\f{s}{2}}}{\langle t\rangle^s}\rmA+\sqrt{\f{\pa_t w}{w}}\tilde{\rmA}\right)P_{\neq}(\Psi\Upsilon)\right\|_2^2. 
\end{align}
The remainder term is also easy to deal with, we have
\ben\label{eq:com1_1,RR}
\Pi^{com1}_{k,1,\mathcal{R}}\lesssim \f{\ep^3}{\langle t\rangle^{2-\rmK_{\rmD}\ep/2}}. 
\een
Decompose the difference:
\begin{align}\label{eq:Decompose}
\rmA_{k}(\eta)-\rmA_{k}(\xi)&=\rmA_k(\xi)\left[e^{\la|k,\eta|^s-\la|k,\xi|^s}-1\right]\\
\nonumber
&\quad+\rmA_k(\xi)e^{\la|k,\eta|^s-\la|k,\xi|^s}\left[\f{\rmJ_k(\eta)}{\rmJ_{k}(\xi)}-1\right]\f{\langle k,\eta\rangle^{\s}}{\langle k,\xi\rangle^{\s}}\\
\nonumber
&\quad+\rmA_k(\xi)e^{\la|k,\eta|^s-\la|k,\xi|^s}\left[\f{\langle k,\eta\rangle^{\s}}{\langle k,\xi\rangle^{\s}}-1\right]. 
\end{align}
We write
\begin{align*}
\Pi^{com1}_{k,1,\rmT}&\lesssim \sum_{\rmN\geq 8}\sum_l\int 1_{D}e^{-\la_{\mathcal{D}}|\eta-\xi|^{s_0}}\rmA_k(\xi)|\Om_k(\xi)|\left[e^{\la|k,\eta|^s-\la|k,\xi|^2}-1\right]\rmA_k(\eta)e^{-\la_{\mathcal{D}}|\eta-\xi_1|^{s_0}}\\
&\qquad \qquad \qquad \qquad \qquad 
\times|\widehat{\rmU}_{k-l}(\xi_1-\xi_2)_{<\rmN/8}||l,\xi_2||\widehat{\Om}_l(\xi_2)_{\rmN}|d\xi d\xi_2d\xi_1d\eta\\
&\quad+\sum_{\rmN\geq 8}\sum_l\int 1_{D}e^{-\la_{\mathcal{D}}|\eta-\xi|^{s_0}}\rmA_k(\xi)|\Om_k(\xi)|e^{\la|k,\eta|^s-\la|k,\xi|^s}\left[\f{\rmJ_k(\eta)}{\rmJ_{k}(\xi)}-1\right]\f{\langle k,\eta\rangle^{\s}}{\langle k,\xi\rangle^{\s}}\rmA_k(\eta)e^{-\la_{\mathcal{D}}|\eta-\xi_1|^{s_0}}\\
&\qquad \qquad \qquad \qquad \qquad 
\times|\widehat{\rmU}_{k-l}(\xi_1-\xi_2)_{<\rmN/8}||l,\xi_2||\widehat{\Om}_l(\xi_2)_{\rmN}|d\xi d\xi_2d\xi_1d\eta\\
&\quad+\sum_{\rmN\geq 8}\sum_l\int 1_{D}e^{-\la_{\mathcal{D}}|\eta-\xi|^{s_0}}\rmA_k(\xi)|\Om_k(\xi)|e^{\la|k,\eta|^s-\la|k,\xi|^s}\left[\f{\langle k,\eta\rangle^{\s}}{\langle k,\xi\rangle^{\s}}-1\right]\rmA_k(\eta)e^{-\la_{\mathcal{D}}|\eta-\xi_1|^{s_0}}\\
&\qquad \qquad \qquad \qquad \qquad 
\times|\widehat{\rmU}_{k-l}(\xi_1-\xi_2)_{\rmN/8}||l,\xi_2||\widehat{\Om}_l(\xi_2)_{\rmN}|d\xi d\xi_2d\xi_1d\eta\\
&=\Pi^{com1}_{k,1,1,\rmT}+\Pi^{com1}_{k,1,2,\rmT}+\Pi^{com1}_{k,1,3,\rmT}. 
\end{align*}
First we treat $\Pi^{com1}_{k,1,1,\rmT}$. Since on the support of the integrand
\beno
&&|k,\xi|\approx |k,\xi_1|,\quad
\big||k,\xi_1|-|l,\xi_2|\big|\leq |k-l,\xi_1-\xi_2|\leq \f{6}{32}|l,\xi_2|,\\
&&\f{26}{32}|l,\xi_2|\leq |k,\xi_1|\leq \f{38}{32}|l,\xi_2|,
\eeno
it gives us that
\begin{align*}
\Pi^{com1}_{k,1,1,\rmT}\lesssim \la\sum_{\rmN\geq 8}\left\||\na|^{s/2}\rmA\Om_{\sim\rmN}\right\|_2\|U\|_{\mathcal{G}^{\la,\s-4;s}}\left\||\na|^{s/2}\rmA\Om_{\rmN}\right\|_2
\lesssim \f{\la\ep}{\langle t\rangle^{2-\rmK_{\rmD}\ep/2}}\left\||\na|^{s/2}\rmA\Om\right\|_2^2. 
\end{align*}
Similarly, we can treat $\Pi^{com1}_{k,1,3}$. Here we omit the proof. 
\beq\label{eq: Pi^com1_1,1}
|\Pi^{com1}_{k,1,3,\rmT}|\lesssim \f{\ep^3}{\langle t\rangle^{2-\rmK_{\rmD}\ep/2}}.
\eeq
We now treat $\Pi^{com1}_{k,1,2,\rmT}$. 
We divide the integral as follows
\begin{align*}
\Pi^{com1}_{k,1,2,\rmT}&\lesssim\sum_{\rmN\geq 8}\sum_l\int 1_{D}[1_S+1_{L}]e^{-\la_{\mathcal{D}}|\eta-\xi|^{s_0}}\rmA_k(\xi)|\Om_k(\xi)|e^{c\la|\eta-\xi|^s}\left[\f{\rmJ_k(\eta)}{\rmJ_{k}(\xi)}-1\right]\\
&\qquad \qquad \qquad
\times\rmA_k(\xi_1)e^{-c\la_{\mathcal{D}}|\eta-\xi_1|^{s_0}}|\widehat{\rmU}_{k-l}(\xi_1-\xi_2)_{<\rmN/8}||l,\xi_2||\widehat{\Om}_l(\xi_2)_{\rmN}|d\xi d\xi_2d\xi_1d\eta\\
&=\Pi^{com1}_{k,1,2,\rmT,S}+\Pi^{com1}_{k,1,2,\rmT,L},
\end{align*}
where $S=\{t\leq \f12\min(\sqrt{|\xi|},\sqrt{|\eta|})\}$, $1_{L}=1-1_{S}$. 

In $\Pi^{com1}_{k,1,2,\rmT,S}$, we apply {\it Lemma 3.7} in \cite{BM1} to gain $1/2$ derivatives. Indeed, we have
\begin{align*}
\Pi^{com1}_{k,1,2,\rmT,S}\lesssim \f{\ep}{\langle t\rangle^{2-\rmK_{\rmD}\ep/2}}\left\||\na|^{s/2}\rmA\Om\right\|_2^2+\f{\ep^3}{\langle t\rangle^{2-\rmK_{\rmD}\ep/2}}. 
\end{align*}
To treat $\Pi^{com1}_{k,1,2,\rmT,L}$, we divide into two cases based on the relative size of $|l|$ and $|\xi_2|$
\begin{align*}
\Pi^{com1}_{k,1,2,\rmT,L}&\lesssim\sum_{\rmN\geq 8}\sum_l\int 1_{D}1_{L}[1_{|l|>100|\xi_2|}+1_{|l|\leq 100|\xi_2|}]e^{-\la_{\mathcal{D}}|\eta-\xi|^{s_0}}\rmA_k(\xi)|\Om_k(\xi)|e^{c\la|\eta-\xi|^s}e^{10\mu|\eta-\xi|^{\f12}}\\
&\qquad \qquad \qquad
\times\rmA_k(\xi_1)e^{-c\la_{\mathcal{D}}|\eta-\xi_1|^{s_0}}|\widehat{\rmU}_{k-l}(\xi_1-\xi_2)_{<\rmN/8}||l,\xi_2||\widehat{\Om}_l(\xi_2)_{\rmN}|d\xi d\xi_2d\xi_1d\eta\\
&=\Pi^{com1}_{k,1,2,\rmT,L,z}+\Pi^{com1}_{k,1,2,\rmT,L,v}.
\end{align*}
In the first case $|k-l|\leq \f{1}{5}|l|$ and $|k|\leq k_{M}$, we get that 
\beno
\Pi^{com1}_{k,1,2,\rmT,L,z}\lesssim \f{\ep^3}{\langle t\rangle^{2-\rmK_{\rmD}\ep/2}}.
\eeno
Turn now to $\Pi^{com1}_{k,1,2,\rmT,L,v}$. Note that $|l,\xi_2|\lesssim |\xi_2|\approx |\xi|\lesssim t^2$. We have $|l,\xi_2|\lesssim |k,\xi|^{s/2}|l,\xi_2|^{s/2}t^{2-2s}$
\beno
\Pi^{com1}_{k,1,2,\rmT,L,v}\lesssim \f{\ep}{\langle t\rangle^{2s-\rmK_{\rmD}\ep/2}}\big\||\na|^{s/2}\rmA\Om\big\|_2^2.
\eeno
Thus we conclude that
\begin{align}\label{eq:com1_1,T}
\Pi^{com1}_{k,1,\rmT}\lesssim \f{\ep}{\langle t\rangle^{2-\rmK_{\rmD}\ep/2}}\left\||\na|^{s/2}\rmA\Om\right\|_2^2+\f{\ep^3}{\langle t\rangle^{2-\rmK_{\rmD}\ep/2}}+\f{\ep}{\langle t\rangle^{2s-\rmK_{\rmD}\ep/2}}\big\||\na|^{s/2}\rmA\Om\big\|_2^2. 
\end{align}
\subsection{Estimate of $\Pi_k^{com2}$}
Now we treat $\Pi_k^{com2}$. We have
\begin{align*}
\Pi^{com2}_k&\lesssim \sum_{l}\int\mathcal{D}(t,k,\eta,\xi_2)\rmA_k(\xi_2)\bar{\hat{\Om}}_k(\xi_2)(\rmA_k(\eta)-\rmA_k(\xi))\\
&\qquad\qquad\qquad\qquad
\times\mathcal{D}^1(t,k,\eta,\xi)\Big(\widehat{\rmU}_{k-l}(\xi-\xi_1)\cdot(il,i\xi_1)\widehat{\Om}_l(\xi_1)\Big)d\xi_2d\xi_1 d\xi d\eta\\
&\lesssim \sum_{l}\int [1_{D}+1_{D^c}]e^{-\la_{\mathcal{D}}|\eta-\xi_2|^{s_0}}\rmA_k(\xi_2)\bar{\hat{\Om}}_k(\xi_2)\Big[\f{\rmA_k(\eta)}{\rmA_k(\xi)}-1\Big]\\
&\qquad\qquad\qquad\qquad
\times e^{-\la_{\mathcal{D}}|\eta-\xi|^{s_0}}\rmA_k(\xi)\Big(\widehat{\rmU}_{k-l}(\xi-\xi_1)\cdot(il,i\xi_1)\widehat{\Om}_l(\xi_1)\Big)d\xi_2d\xi_1 d\xi d\eta\\
&=\Pi_{k,1}^{com2}+\Pi_{k,2}^{com2},
\end{align*}
where $D=\{|\eta-\xi|\leq \f{1}{100}|\xi|,\ |\xi|\geq 100k_{M}\}$ and $1_{D^c}=1-1_{D}$. 

Then we have 
\ben\label{eq:Pi^com2_2}
|\Pi_{k,2}^{com2}|\lesssim \f{\ep^3}{\langle t\rangle^{2-\rmK_{\rmD}\ep/2}}. 
\een

Similarly, we write
\begin{align*}
\Pi_{k,1}^{com2}&\lesssim \sum_{\rmN\geq 8}\sum_{l}\int 1_{D}e^{-\la_{\mathcal{D}}|\eta-\xi_2|^{s_0}}\rmA_k(\xi_2)\bar{\hat{\Om}}_k(\xi_2)\Big[\f{\rmA_k(\eta)}{\rmA_k(\xi)}-1\Big]\\
&\qquad\qquad\qquad\qquad
\times e^{-\la_{\mathcal{D}}|\eta-\xi|^{s_0}}\rmA_k(\xi)\Big(\widehat{\rmU}_{k-l}(\xi-\xi_1)_{<\rmN/8}|l,\xi_1|\widehat{\Om}_l(\xi_1)_{\rmN}\Big)d\xi_2d\xi_1 d\xi d\eta\\
&\quad+\sum_{\rmN\geq 8}\sum_{l}\int 1_{D}e^{-\la_{\mathcal{D}}|\eta-\xi_2|^{s_0}}\rmA_k(\xi_2)\bar{\hat{\Om}}_k(\xi_2)\Big[\f{\rmA_k(\eta)}{\rmA_k(\xi)}-1\Big]\\
&\qquad\qquad\qquad\qquad
\times e^{-\la_{\mathcal{D}}|\eta-\xi|^{s_0}}\rmA_k(\xi)\Big(\widehat{\rmU}_{k-l}(\xi-\xi_1)_{\rmN}|l,\xi_1|\widehat{\Om}_l(\xi_1)_{<\rmN/8}\Big)d\xi_2d\xi_1 d\xi d\eta\\
&\quad+\sum_{\rmN'\in\mathcal{D}}\sum_{\f18\rmN'\leq \rmN\leq 8\rmN'}\sum_{l}\int 1_{D}e^{-\la_{\mathcal{D}}|\eta-\xi_2|^{s_0}}\rmA_k(\xi_2)\bar{\hat{\Om}}_k(\xi_2)\Big[\f{\rmA_k(\eta)}{\rmA_k(\xi)}-1\Big]\\
&\qquad\qquad\qquad\qquad
\times e^{-\la_{\mathcal{D}}|\eta-\xi|^{s_0}}\rmA_k(\xi)\Big(\widehat{\rmU}_{k-l}(\xi-\xi_1)_{\rmN}|l,\xi_1|\widehat{\Om}_l(\xi_1)_{\rmN'}\Big)d\xi_2d\xi_1 d\xi d\eta\\
&=\Pi_{k,1,\rmT}^{com2}+\Pi_{k,1,\rmR}^{com2}+\Pi_{k,1,\mathcal{R}}^{com2}. 
\end{align*}
Similar to the $\Pi_{k,1,\rmR}^{com1}$ and $\Pi_{k,1,\mathcal{R}}^{com1}$, we can use the same argument as the estimate of the reaction term and remainder term in {\it section 6 and 7} of \cite{BM1}. We omit the proof and show the result.
\begin{align}\label{eq:Pi_1,RandPi_1,RRcom2}
|\Pi^{com2}_{k,1,\rmR}|+|\Pi^{com2}_{k,1,\mathcal{R}}|\lesssim &\ep \mathrm{CK}_{\la}+\ep\mathrm{CK}_{w}+\f{\ep^3}{\langle t\rangle^{2-\rmK_{\rmD}\ep/2}}
+\ep\mathrm{CK}_{\la}^{v,1}+\ep\mathrm{CK}_{w}^{v,1}\\
\nonumber
&+\ep\left\|\left\langle\f{\pa_v}{t\pa_z}\right\rangle^{-1}(\pa_z^2+(\pa_v-t\pa_z)^2)\left(\f{|\na|^{\f{s}{2}}}{\langle t\rangle^s}\rmA+\sqrt{\f{\pa_t w}{w}}\tilde{\rmA}\right)P_{\neq}(\Psi\Upsilon)\right\|_2^2. 
\end{align}
We treat $\Pi^{com2}_{k,1,\rmT}$. Since on the support of the integrand
\begin{align*}
&||k,\xi|-|l,\xi_1||\leq |k-l,\xi-\xi_1|\leq \f{6}{32}|l,\xi_1|,\\
&\f{26}{32}|l,\xi_1|\leq |k,\xi|\leq \f{38}{32}|l,\xi_1|,\ 0<|k|\leq k_{M}, \\
&|k,\xi_2|+|\xi_2|\geq |k,\eta|\approx |\eta|\approx|k,\xi|\approx |\xi|\geq 100k_{M}.
\end{align*}
By using the composition in \eqref{eq:Decompose} and following the same argument as the treatments of $\Pi_{k,1,1,\rmT}^{com1}$, $\Pi_{k,1,2,\rmT}^{com1}$ and $\Pi_{k,1,3,\rmT}^{com1}$, we get that
\begin{align}\label{eq:com2_1,T}
\Pi^{com2}_{k,1,\rmT}\lesssim \f{\ep}{\langle t\rangle^{2s-\rmK_{\rmD}\ep/2}}\left\||\na|^{s/2}\rmA\Om\right\|_2^2+\f{\ep^3}{\langle t\rangle^{2-\rmK_{\rmD}\ep/2}}. 
\end{align}

\subsection{Estimate of $\Pi_k^{com3}$}
By Corollary \ref{corol: commutator}, we get that
\begin{align*}
\Pi_k^{com3}&=i\sum_l\int\rmA_k(\eta)\bar{\hat{f}}_k(\eta)
\rmA_k(\eta)\mathcal{D}^{com}(t,k,\eta,\xi)\widehat{\rmU}_{k-l}(\xi-\xi_1)\cdot(l,\xi_1)\hat{\Om}_l(\xi_1)d\xi_1 d\xi d\eta\\
&\lesssim \sum_l\int 1_{D^c}\rmA_k(\eta)|\bar{\hat{f}}_k(\eta)|
\f{e^{-c\la_{\mathcal{D}}|\eta-\xi|^{s_0}}}{1+|\eta-kt|^2+|\xi-kt|^2}\rmA_{k}(\xi)\widehat{\rmU}_{k-l}(\xi-\xi_1)|l,\xi_1|\hat{\Om}_l(\xi_1)d\xi_1 d\xi d\eta\\
&\quad+\sum_{\rmN\geq 8}\sum_l\int 1_{D}\rmA_k(\eta)|\bar{\hat{f}}_k(\eta)|
\f{e^{-c\la_{\mathcal{D}}|\eta-\xi|^{s_0}}}{1+|\eta-kt|^2+|\xi-kt|^2}\\
&\qquad\qquad\qquad\qquad\qquad\qquad
\times \rmA_{k}(\xi)|\widehat{\rmU}_{k-l}(\xi-\xi_1)_{<\rmN/8}||l,\xi_1||\hat{\Om}_l(\xi_1)_{\rmN}|d\xi_1 d\xi d\eta\\
&\quad+\sum_{\rmN\geq 8}\sum_l\int 1_{D}\rmA_k(\eta)|\bar{\hat{f}}_k(\eta)|
\f{e^{-c\la_{\mathcal{D}}|\eta-\xi|^{s_0}}}{1+|\eta-kt|^2+|\xi-kt|^2}\\
&\qquad\qquad\qquad\qquad\qquad\qquad
\times \rmA_{k}(\xi)|\widehat{\rmU}_{k-l}(\xi-\xi_1)_{\rmN}||l,\xi_1||\hat{\Om}_l(\xi_1)_{<\rmN/8}|d\xi_1 d\xi d\eta\\
&\quad+\sum_{\rmN'\in \mathcal{D}}\sum_{\f18\rmN'\leq \rmN\leq 8\rmN'}\sum_l\int 1_{D}\rmA_k(\eta)|\bar{\hat{f}}_k(\eta)|
\f{e^{-c\la_{\mathcal{D}}|\eta-\xi|^{s_0}}}{1+|\eta-kt|^2+|\xi-kt|^2}\\
&\qquad\qquad\qquad\qquad\qquad\qquad
\times \rmA_{k}(\xi)|\widehat{\rmU}_{k-l}(\xi-\xi_1)_{\rmN}||l,\xi_1||\hat{\Om}_l(\xi_1)_{\rmN'}|d\xi_1 d\xi d\eta\\
&=\Pi^{com3}_{k,2}+\Pi^{com3}_{k,1,\rmT}+\Pi^{com3}_{k,1,\rmR}
+\Pi^{com3}_{k,1,\mathcal{R}},
\end{align*}
where $D=\{|\eta-\xi|\leq \f{1}{100}|\xi|,\  |\xi|\geq 100k_{M}\}$ and $1_{D^c}=1-1_{D}$. 

The term $\Pi^{com3}_{k,2}$ is easy to deal with, we have
\ben\label{eq:com3_2}
|\Pi^{com3}_{k,2}|\lesssim \f{\ep^3}{\langle t\rangle^{2-\rmK_{\rmD}\ep/2}}. 
\een
Similar to the above cases, $\Pi^{com3}_{k,1,\rmR}$ and $\Pi^{com3}_{k,1,\mathcal{R}}$ are easy to treat, by following the estimate in {\it sections 6 and 7} of \cite{BM1}. We have
\begin{align}\label{eq:Pi_1,RandPi_1,RRcom3}
|\Pi^{com3}_{k,1,\rmR}|+|\Pi^{com3}_{k,1,\mathcal{R}}|&\lesssim \ep \mathrm{CK}_{\la}+\ep\mathrm{CK}_{w}+\f{\ep^3}{\langle t\rangle^{2-\rmK_{\rmD}\ep/2}}
+\ep\mathrm{CK}_{\la}^{v,1}+\ep\mathrm{CK}_{w}^{v,1}\\
\nonumber
&+\ep\left\|\left\langle\f{\pa_v}{t\pa_z}\right\rangle^{-1}(\pa_z^2+(\pa_v-t\pa_z)^2)\left(\f{|\na|^{\f{s}{2}}}{\langle t\rangle^s}\rmA+\sqrt{\f{\pa_t w}{w}}\tilde{\rmA}\right)P_{\neq}(\Psi\Upsilon)\right\|_2^2. 
\end{align}
We treat $\Pi^{com3}_{k,1,\rmT}$. Since on the support of the integrand
\begin{align*}
&||k,\xi|-|l,\xi_1||\leq |k-l,\xi-\xi_1|\leq \f{6}{32}|l,\xi_1|,\\
&\f{26}{32}|l,\xi_1|\leq |k,\xi|\leq \f{38}{32}|l,\xi_1|,\ 0<|k|\leq k_{M}, \\
&|k,\eta|\approx |\eta|\approx|k,\xi|\approx |\xi|\geq 100k_{M}.
\end{align*}
We write 
\begin{align*}
\Pi^{com3}_{k,1,\rmT}&=\sum_{\rmN\geq 8}\sum_l\int 1_{D}[1_{R}+1_{NR}]\rmA_k(\eta)|\bar{\hat{f}}_k(\eta)|
\f{e^{-c\la_{\mathcal{D}}|\eta-\xi|^{s_0}}}{1+|\eta-kt|^2+|\xi-kt|^2}\\
&\qquad\qquad\qquad\qquad\qquad\qquad
\times \rmA_{k}(\xi)|\widehat{\rmU}_{k-l}(\xi-\xi_1)_{<\rmN/8}||l,\xi_1||\hat{\Om}_l(\xi_1)_{\rmN}|d\xi_1 d\xi d\eta\\
&=\Pi^{com3}_{k,1,\rmT,1}+\Pi^{com3}_{k,1,\rmT,2},
\end{align*}
where $R=\{t\in \mathrm{I}_{k,\eta}\cap\mathrm{I}_{k,\xi}\}$ and $1_{NR}=1-1_{R}$. \\
For $\Pi^{com3}_{k,1,\rmT,2}$, by the fact that $|\eta-kt|^2+|\xi-kt|^2\gtrsim \f{|\eta|^2}{k^2}\gtrsim |\eta|^2$, we have
\ben\label{com3_1,T,2}
|\Pi^{com3}_{k,1,\rmT,2}|\lesssim \f{\ep^3}{\langle t\rangle^{2-\rmK_{\rmD}\ep/2}}. 
\een
Now we treat $\Pi^{com3}_{k,1,\rmT,1}$, we have
\begin{align*}
\Pi^{com3}_{k,1,\rmT,1}&\lesssim \sum_{\rmN\geq 8}\sum_l\int 1_{D}1_{R}\rmA_k(\eta)|\bar{\hat{f}}_k(\eta)|
\f{e^{-c\la_{\mathcal{D}}|\eta-\xi|^{s_0}}|k,\eta|}{k^2+|\eta-kt|^2}\\
&\qquad\qquad\qquad\qquad\qquad\qquad
\times \rmA_{k}(\xi)|\widehat{\rmU}_{k-l}(\xi-\xi_1)_{<\rmN/8}||\hat{\Om}_l(\xi_1)_{\rmN}|d\xi_1 d\xi d\eta.
\end{align*}
It can be regarded as the action of the resonant modes $|\bar{\hat{f}}_k(\eta)|
\f{|k,\eta|}{k^2+|\eta-kt|^2}$ on $\hat{\Om}_l(\xi_1)$ which can be either non-resonant or resonant. Due to the fact that $\rmU$ decays as $\langle t\rangle^{-2+\rmK_{\rmD}\ep/2}$,
this term is better than $\rmR^{\rmN\rmR,\rmR}_{\rmN}$ and $\rmR^{\rmR,\rmR}_{\rmN}$ in {\it section 6.1.3 and 6.1.4} of \cite{BM1}. 
Thus we obtain that
\begin{align}\label{eq:Pi_1,Tcom3}
|\Pi^{com3}_{k,1,\rmT,1}|\lesssim& \ep \mathrm{CK}_{\la}+\ep\mathrm{CK}_{w}+\f{\ep^3}{\langle t\rangle^{2-\rmK_{\rmD}\ep/2}}. 
\end{align}

Together with \eqref{eq:Pi_2^com1}, \eqref{eq:com1_1,R}, \eqref{eq:com1_1,RR}, \eqref{eq:com1_1,T}, \eqref{eq:Pi^com2_2}, \eqref{eq:Pi_1,RandPi_1,RRcom2}, \eqref{eq:com2_1,T}, \eqref{eq:com3_2}, \eqref{eq:Pi_1,RandPi_1,RRcom3}, \eqref{com3_1,T,2} and \eqref{eq:Pi_1,Tcom3}, we prove Proposition \ref{prop:com}.

\section{Nonlocal terms}\label{Sec:non-local-largetime}
First let us deal with the nonlocal term $\mathrm{II}^{u''}$. We have
\begin{align*}
\Pi^{u''}&=\sum_{\rmN\in \mathcal{D},\, \rmN\geq 8}\int \rmA\left(\udl{u''}_{\rmN}1_{|k|\geq k_{M}}P_{\neq}\pa_z(\Psi\Upsilon)_{<\f18\rmN}\right)\rmA fdx\\
&\quad+\sum_{\rmN\in \mathcal{D},\, \rmN\geq 8}\int \rmA\left(\udl{u''}_{<\f18\rmN}1_{|k|\geq k_{M}}P_{\neq}\pa_z(\Psi\Upsilon)_{\rmN}\right)\rmA f dx\\
&\quad+\sum_{\rmN\in \mathcal{D}}\sum_{\f18\rmN\leq \rmN'\leq 8\rmN}\int \rmA\left(\udl{u''}_{\rmN}1_{|k|\geq k_{M}}P_{\neq}\pa_z(\Psi\Upsilon)_{\rmN'}\right)\rmA f dx\\
&=\Pi^{u''}_{\mathrm{HL}}+\Pi^{u''}_{\mathrm{LH}}+\Pi^{u''}_{\mathcal{R}}.
\end{align*}
In this section we use the fact that 
\beno
P_{|k|\geq k_{M}}\Psi\Upsilon=\udl{\Delta_t^{-1}}P_{|k|\geq k_{M}}f.
\eeno 
\subsection{High-low interaction and high-high interaction}
The estimate of $\Pi^{u''}_{\mathrm{HL}}$ and $\Pi^{u''}_{\mathcal{R}}$ are easy. We omit the proof and show the result that
\ben\label{eq:u''}
|\Pi^{u''}_{\mathrm{HL}}|+|\Pi^{u''}_{\mathcal{R}}|\lesssim \f{\ep^2}{k_{M}\langle t\rangle^2}. 
\een

\subsection{Low-high interaction}
We get that
\begin{align*}
|\Pi^{u''}_{\mathrm{LH}}|
&\lesssim \sum_{\rmN\in \mathcal{D},\, \rmN\geq 8}\sum_{k\neq 0}\left|\int_{\xi,\eta} \rmA \overline{\widehat{f}}_k(\eta)\rmA_{k}(t,\eta)\left(\widehat{\udl{u''}}(\eta-\xi)_{<\f18\rmN}1_{|k|\geq k_{M}}|k|\widehat{(\Psi\Upsilon)}_k(\xi)_{\rmN}\right) d\xi d\eta\right|\\
&\lesssim \sum_{\rmN\in \mathcal{D},\, \rmN\geq 8}\sum_{k\neq 0}\bigg|\int_{\xi,\eta} [1_{\mathrm{NR,NR}}+1_{\mathrm{NR,R}}+1_{\mathrm{R,NR}}+1_{\mathrm{R,R}}]1_{|k|\geq k_{M}}\\
&\qquad\qquad\qquad\times\rmA \overline{\widehat{f}}_k(\eta)\rmA_{k}(t,\eta)\left(\widehat{\udl{u''}}(\eta-\xi)_{<\f18\rmN}|k|\widehat{(\Psi\Upsilon)}_k(\xi)_{\rmN}\right) d\xi d\eta\bigg|\\
&=\mathrm{II}^{u'';\mathrm{NR,NR}}_{\mathrm{LH}}
+\mathrm{II}^{u'';\mathrm{NR,R}}_{\mathrm{LH}}
+\mathrm{II}^{u'';\mathrm{R,NR}}_{\mathrm{LH}}
+\mathrm{II}^{u'';\mathrm{R,R}}_{\mathrm{LH}}
\end{align*}

\no{\bf Treatment of $\mathrm{II}^{u'';\mathrm{NR,NR}}_{\mathrm{LH}}$. }

By {\it (3.31), (6.1) and (A.11)} in \cite{BM1}, we have
\begin{align*}
\mathrm{II}^{u'';\mathrm{NR,NR}}_{\mathrm{LH}}&\lesssim \sum_{\rmN\in \mathcal{D},\, \rmN\geq 8}\|\udl{u''}\|_{\mathcal{G}^{\la}}\||\na|^{s/2}\rmA f_{\sim\rmN}\|_2
\left\|\f{|\pa_z|}{\langle\na\rangle^{s/2}}1_{\mathrm{NR}} 1_{|k|\geq k_{M}}P_{\neq}\rmA(\Psi\Upsilon)_{\rmN}\right\|_2. 
\end{align*}
\begin{remark}\label{Rmk: modified 6.11}
For any $\kappa_0>0$, there exist $k_{M}$ such that for $|k|>k_{M}$, it holds that
\beno
\langle t\rangle^{s}\left\|\f{|\pa_z|}{\langle\na\rangle^{s/2}}1_{\mathrm{NR}} 1_{|k|\geq k_{M}}P_{\neq}\rmA(\Psi\Upsilon)\right\|_2
\leq \kappa_0 \left\|\left\langle\f{\pa_v}{t\pa_z}\right\rangle^{-1}\f{|\na|^{s/2}}{\langle t\rangle^{s}}\Delta_{L}P_{\neq}\rmA (\Psi\Upsilon)\right\|_2.
\eeno
\end{remark}
\begin{proof}
The remark follows from the following claim: \\
{\bf Claim: }For any $\kappa_0>0$, there exists $k_{M}$ such that for $|k|>k_{M}$ and $t\notin \mathbf{I}_{k,\xi}$, it holds that
\beq\label{eq: 6.12modify}
\f{\langle t\rangle^{s}|k|}{\langle k,\xi\rangle^{s/2}}1_{t\notin \mathbf{I}_{k,\xi}} 1_{|k|\geq k_{M}}\leq \kappa_0 \left\langle\f{\xi}{tk}\right\rangle^{-1}\f{|k,\xi|^{s/2}}{\langle t\rangle^{s}}(k^2+(\xi-kt)^2). 
\eeq
Let consider the following three cases:

Case 1: $\left|\f{\xi}{tk}\right|\leq \f12$, then $\left\langle\f{\xi}{tk}\right\rangle^{-1}\approx 1$ and $(k^2+(\xi-kt)^2)\approx k^2t^2$, thus \eqref{eq: 6.12modify} is equivalent to prove that 
\beno
\f{\langle t\rangle^{2s}|k|}{\langle k,\xi\rangle^{s}}
\leq \kappa_0 k^2t^2,
\eeno
which is obvious when $|k|\geq k_{M}$.

Case 2: $\left|\f{\xi}{tk}\right|\geq \f32$, then $\left\langle\f{\xi}{tk}\right\rangle^{-1}\approx \f{tk}{\xi}$, $\langle k,\xi\rangle^{s}\approx |\xi|^{s}$ and $(k^2+(\xi-kt)^2)\approx k^2+\xi^2$, thus \eqref{eq: 6.12modify} is equivalent to prove that
\beno
\f{|t|^{2s}|k||\xi|}{|kt||\xi|^{s}(k^2+\xi^2)}
\leq \kappa_0. 
\eeno
Indeed we have for the left hand side that
\beno
\f{|t|^{2s}|k||\xi|}{|kt||\xi|^{s}(k^2+\xi^2)}
\approx 
|t|^{2s-1}|\xi|^{-s-1}\lesssim |t|^{s-2}|k|^{-s-1}\leq \kappa_0. 
\eeno

Case 3: $\f12\leq \left|\f{\xi}{tk}\right|\leq \f32$, then $|\xi|\approx |kt|$ and $(k^2+(\xi-kt)^2)=k^2(1+|t-\f{\xi}{k^2}|)\geq k^2+\f{\xi^2}{k^2}\approx k^2+t^2$. Thus
\beno
\f{\langle t\rangle^{2s}|k|}{\langle k,\xi\rangle^{s}(k^2+(\xi-kt)^2)}\lesssim \f{\langle t\rangle^{s}|k|^{1-s}}{(k^2+t^2)}\leq \kappa_0. 
\eeno
Thus we proved the claim, which gives the remark. 
\end{proof}
Therefore we conclude that for $k_{M}$ sufficiently large, 
\ben\label{eq:LHNRNR}
|\mathrm{II}^{u'';\mathrm{NR,NR}}_{\mathrm{LH}}|
\leq \f{\kappa_0}{\langle t\rangle^{2s}} \||\na|^{s/2}\rmA f\|_2^2+\kappa_0 \left\|\left\langle\f{\pa_v}{t\pa_z}\right\rangle^{-1}\f{|\na|^{s/2}}{\langle t\rangle^{s}}\Delta_{L}P_{|k|\geq k_{M}}\rmA (\Psi\Upsilon)\right\|_2^2.
\een

\no{\bf Treatment of $\mathrm{II}^{u'';\mathrm{R,NR}}_{\mathrm{LH}}$.}

By {\it (6.1) and (A.7)} in \cite{BM1} and Remark \ref{Rmk: modified 6.11}, for some $c\in (0,1)$, and for any $\kappa_0>0$ there exists $k_{M}$ such that
\beq\label{eq:LH,RNR}
\begin{split}
|\mathrm{II}^{u'';\mathrm{R,NR}}_{\mathrm{LH}}|
&\lesssim \sum_{\rmN\in \mathcal{D},\, \rmN\geq 8}
\sum_{k\neq 0}\int_{\xi,\eta} |\rmA\overline{\widehat{f}}_k(\eta)|\rmJ_{k}(\eta)e^{\la |k,\xi|^s}e^{c\la|\eta-\xi|^s}\langle k,\xi\rangle^{\s} |k| 1_{t\in \mathrm{I}_{k,\eta}}1_{t\notin \mathrm{I}_{k,\xi}}\\
&\qquad\qquad \qquad\qquad\times\bigg|\widehat{\udl{u''}}(\eta-\xi)_{<\f18\rmN}1_{|k|\geq k_{M}}|k|\widehat{(\Psi\Upsilon)}_k(\xi)_{\rmN}\bigg| d\xi d\eta\\
&\lesssim \sum_{\rmN\in \mathcal{D},\, \rmN\geq 8}
\||\na|^{s/2}\rmA f_{\sim \rmN}\|_2\left\|\f{\pa_z}{|\na|^{s/2}}1_{\mathrm{NR}}1_{|k|\geq k_{M}}\rmA P_{\neq}(\Psi\Upsilon)_{\rmN}\right\|_2\\
&\leq \f{\kappa_0}{\langle t\rangle^{2s}} \||\na|^{s/2}\rmA f\|_2^2+\kappa_0 \left\|\left\langle\f{\pa_v}{t\pa_z}\right\rangle^{-1}\f{|\na|^{s/2}}{\langle t\rangle^{s}}\Delta_{L}P_{|k|\geq k_{M}}\rmA (\Psi\Upsilon)\right\|_2^2. 
\end{split}\eeq

\no{\bf Treatment of $\mathrm{II}^{u'';\mathrm{NR,R}}_{\mathrm{LH}}$.}

In this case $(k,\xi)$ is resonant and $(k,\eta)$ in non-resonant. It follows that $4|k|^2\leq |\xi|$ and since $\rmN\geq 8$, we have $|\xi|\approx |\eta|$. By {\it (3.16), (3.32), (A.7), (A.12), (A.11) and (6.6)} in \cite{BM1}, 
\begin{align*}
|\mathrm{II}^{u'';\mathrm{NR,R}}_{\mathrm{LH}}|&\lesssim \sum_{\rmN\in \mathcal{D},\, \rmN\geq 8}\sum_{k\neq 0}\int_{\xi,\eta}1_{\mathrm{NR,R}}1_{|k|\geq k_{M}}\bigg[\sqrt{\f{\pa_tw_k(t,\eta)}{w_k(t,\eta)}}+\f{|k,\eta|^{s/2}}{\langle t\rangle^s}\bigg]|\rmA\widehat{f}_k(\eta)_{\sim\rmN}|\\
&
\quad\times \rmJ_k(\xi)\f{w_{\mathrm{R}}(\xi)}{w_{\mathrm{NR}}(\xi)}\sqrt{\f{w_k(t,\xi)}{\pa_tw_k(t,\xi)}}e^{\la|k,\xi|^s}e^{\la|\eta-\xi|^s}\langle k,\xi\rangle^{\s}|k|P_{\neq}\widehat{(\Psi\Upsilon)}_{k}(\xi)_{N}\widehat{\langle\pa_v\rangle\udl{u''}}(\eta-\xi)_{<\f18N}d\eta d\xi.
\end{align*}
On the support of the integrand, we have $\rmA_k(\eta)\lesssim \tilde{\rmA}_k(\eta)$ and $\rmA_k(\xi)\lesssim \tilde{\rmA}_k(\xi)$. Moreover $|\xi|\approx |kt|$. Since $t\in\mathrm{I}_{k,\xi}$, thus by {\it (3.5), (3.8) and (3.14)} in \cite{BM1}, we have $|\xi|\gtrsim k^2$
\begin{align*}
&\sqrt{\f{w(t,\xi)}{\pa_tw(t,\xi)}}|k|\f{w_{\mathrm{R}}(t,\xi)}{w_{\mathrm{NR}}(t,\xi)}1_{|k|\geq k_{M}}1_{t\in \mathrm{I}_{k,\xi}}\\
&\approx \Big(1+\big|t-\f{\xi}{k}\big|\Big)^{\f12}\f{|k|^3\Big(1+\big|t-\f{\xi}{k}\big|\Big)}{\xi}1_{t\in \mathrm{I}_{k,\xi}}1_{|k|\geq k_{M}}\\
&\approx \f{1}{k_{M}}\big(k^2+(\xi-kt)^2\big)\sqrt{\f{\pa_tw(t,\xi)}{w(t,\xi)}}
\end{align*}
Thus for any $\kappa_0>0$, there is $k_{M}$, such that for $|k|\geq k_{M}$
\begin{align}\label{eq: LHNRR}
|\mathrm{II}^{u'';\mathrm{NR,R}}_{\mathrm{LH}}|
&\lesssim \sum_{\rmN\in \mathcal{D},\, \rmN\geq 8}
\left\|\pa_v\udl{u''}\right\|_{\mathcal{G}^{\la}}
\left(\left\|\sqrt{\f{\pa_tw}{w}}\tilde{\rmA}f_{\sim \rmN}\right\|_2+\left\|\f{|\na|^{s/2}}{\langle t\rangle^{s}}\rmA f_{\sim \rmN}\right\|_2\right)\\
\nonumber&\quad\qquad\qquad\times\left\|\sqrt{\f{w}{\pa_tw}}|\pa_z|\f{w_{\mathrm{R}}}{w_{\mathrm{NR}}}1_{\mathrm{R}}1_{|k|\geq k_{M}}\tilde{\rmA}(\Psi\Upsilon)_{\rmN}\right\|_2\\
\nonumber&\leq \kappa_0\sum_{\rmN\in \mathcal{D},\, \rmN\geq 8}
\left\|\sqrt{\f{\pa_tw}{w}}\tilde{\rmA}f_{\sim \rmN}\right\|_2^2+\left\|\f{|\na|^{s/2}}{\langle t\rangle^{s}}\rmA f_{\sim \rmN}\right\|_2^2+\left\|\sqrt{\f{\pa_tw}{w}}\Delta_{L}\tilde{\rmA}P_{|k|\geq k_{M}}(\Psi\Upsilon)_{\rmN}\right\|_2^2\\
\nonumber&\leq \kappa_0\left\|\sqrt{\f{\pa_tw}{w}}\tilde{\rmA} f\right\|_2^2+\kappa_0\left\|\f{|\na|^{s/2}}{\langle t\rangle^{s}}\rmA f\right\|_2^2+\kappa_0\left\|\sqrt{\f{\pa_tw}{w}}\Delta_{L}\tilde{\rmA}P_{|k|\geq k_{M}}(\Psi\Upsilon)\right\|_2^2.
\end{align}

\no{\bf Treatment of $\mathrm{II}^{u'';\mathrm{R,R}}_{\mathrm{LH}}$.}

We have
\begin{align*}
|\mathrm{II}^{u'';\mathrm{R,R}}_{\mathrm{LH}}|
&\lesssim \sum_{\rmN\in \mathcal{D},\, \rmN\geq 8}\sum_{|k|\geq k_{M}}\bigg|\int_{\xi,\eta} 1_{\mathrm{R,R}}1_{|k|\geq k_{M}}\\
&\qquad\qquad\qquad\times\rmA \overline{\widehat{f}}_k(\eta)\rmA_{k}(t,\eta)\left(\widehat{\udl{u''}}(\eta-\xi)_{<\f18\rmN}|k|\widehat{(\Psi\Upsilon)}_k(\xi)_{\rmN}\right) d\xi d\eta\bigg|.
\end{align*}
By {\it (3.31) and Lemma 3.4} in \cite{BM1} and the fact that $\rmA_{k}(\eta)\lesssim \tilde{\rmA}_k(\eta)$ and $\rmA_{k}(\xi)\lesssim \tilde{\rmA}_k(\xi)$, we have
\begin{align*}
|\mathrm{II}^{u'';\mathrm{R,R}}_{\mathrm{LH}}|
&\leq C \sum_{\rmN\in \mathcal{D},\, \rmN\geq 8}\sum_{|k|\geq k_{M}}\int_{\xi,\eta} |\rmA\overline{\widehat{f}}_k(\eta)|\f{\rmJ_{k}(t,\eta)}{\rmJ_{k}(t,\xi)}\sqrt{\f{\pa_tw_k(t,\eta)}{w_k(t,\eta)}}\sqrt{\f{w_k(t,\xi)}{\pa_tw_k(t,\xi)}}\langle \xi-\eta\rangle^{1/2}\\
&\qquad\qquad\times e^{\la|\xi|^s}\langle k,\xi\rangle^{\s}\rmJ_{k}(t,\xi)e^{c\la|\eta-\xi|^s}\left|\widehat{\udl{u''}}(\eta-\xi)_{<\f18\rmN}1_{|k|\geq k_{M}}|k|\widehat{(\Psi\Upsilon)}_k(\xi)_{\rmN}\right| d\xi d\eta\\
&\leq C \sum_{\rmN\in \mathcal{D},\, \rmN\geq 8}\sum_{|k|\geq k_{M}}\int_{\xi,\eta} |\rmA\overline{\widehat{f}}_k(\eta)|e^{10\mu|\eta-\xi|^{\f12}}\sqrt{\f{\pa_tw_k(t,\eta)}{w_k(t,\eta)}}\sqrt{\f{w_k(t,\xi)}{\pa_tw_k(t,\xi)}}\langle \xi-\eta\rangle^{1/2}\\
&\qquad\qquad\times e^{\la|\xi|^s}\langle k,\xi\rangle^{\s}\rmJ_{k}(t,\xi)e^{c\la|\eta-\xi|^s}\left|\widehat{\udl{u''}}(\eta-\xi)_{<\f18\rmN}1_{|k|\geq k_{M}}|k|\widehat{(\Psi\Upsilon)}_k(\xi)_{\rmN}\right| d\xi d\eta\\
&\leq C \sum_{\rmN\in \mathcal{D},\, \rmN\geq 8}\left\|\sqrt{\f{\pa_tw}{w}}\tilde{\rmA}f_{\sim \rmN}\right\|_2
\left\|\sqrt{\f{w}{\pa_tw}}|\pa_z|1_{\mathrm{R}}1_{|k|\geq k_{M}}\tilde{\rmA}P_{\neq} (\Psi\Upsilon)_{\rmN}\right\|_2\|\udl{u''}\|_{\mathcal{G}^{\la}}. 
\end{align*}
The constant $C$, is independent of $k_{M}$. 
By {\it (3.14)} in \cite{BM1}, 
\begin{align*}
\sqrt{\f{w_{k}(t,\xi)}{\pa_tw_{k}(t,\xi)}}|k|1_{t\in \mathrm{I}_{k,\xi}}1_{|k|\geq k_{M}}
&\lesssim |k|\sqrt{1+\Big|t-\f{\xi}{k}\Big|}1_{t\in \mathrm{I}_{k,\xi}}1_{|k|\geq k_{M}}\\
&\lesssim |k|\bigg(1+\Big|t-\f{\xi}{k}\Big|\bigg)\sqrt{\f{\pa_tw_{k}(t,\xi)}{w_{k}(t,\xi)}}1_{t\in \mathrm{I}_{k,\xi}}1_{|k|\geq k_{M}}
\end{align*}
which implies, 
\begin{align}\label{eq:LH,1RR}
|\mathrm{II}^{u'';\mathrm{R,R}}_{\mathrm{LH}}|
\leq \f{1}{k_{M}} \left\|\sqrt{\f{\pa_tw}{w}}\tilde{\rmA}f\right\|_2^2+\f{C}{k_{M}}
\left\|\sqrt{\f{\pa_tw}{w}}\Delta_{L}\tilde{\rmA}P_{|k|\geq k_{M}}1_{\rmR} (\Psi\Upsilon)\right\|_2^2,
\end{align}
for some constant $C$ independent of $k_{M}$. 

Finally, summing together \eqref{eq:u''}, \eqref{eq:LHNRNR}, \eqref{eq: LHNRR}, \eqref{eq:LH,RNR} and \eqref{eq:LH,1RR} and by taking $k_{M}$ sufficiently large, we proved Proposition \ref{prop: nonlocal}.

\subsection{Nonlocal error terms}
In this subsection, let us treat $\Pi^{u'',\ep_1}$ and $\Pi^{u'',\ep_2}$. 
By using the $\ep$-smallness of $\udl{u''}-\widetilde{u''}$ and the same argument as in {\it section 6} of \cite{BM1}, it is easy to obtain that
\beno
\begin{split}
|\Pi^{u'',\ep_1}|\lesssim 
&\ep \mathrm{CK}_{\la}+\ep\mathrm{CK}_{w}+\f{\ep^3}{\langle t\rangle^{2-\rmK_{\rmD}\ep/2}}
+\ep\mathrm{CK}_{\la}^{v,1}+\ep\mathrm{CK}_{w}^{v,1}\\
&+\ep\left\|\left\langle\f{\pa_v}{t\pa_z}\right\rangle^{-1}(\pa_z^2+(\pa_v-t\pa_z)^2)\left(\f{|\na|^{\f{s}{2}}}{\langle t\rangle^s}\rmA+\sqrt{\f{\pa_t w}{w}}\tilde{\rmA}\right)P_{\neq}(\Psi\Upsilon)\right\|_2^2. 
\end{split}
\eeno

To treat $\Pi^{u'',\ep_2}$, let us recall \eqref{eq: Delta_uPsi} and the definition of $\rmT^1$ and $\rmT^2$ \eqref{eq: T^1,T^2}.  
Thus we get that
\beno
\widetilde{u''}(\Delta_u^{-1}\Om-\Delta_t^{-1}\Om)
=-\widetilde{u''}\udl{\Delta_u^{-1}}\Big(\rmT^1+\rmT^2\Big).
\eeno
Therefore, by using the same argument as in {\it section 6} of \cite{BM1} and the estimate in section \ref{sec:Precision elliptic control}, we obtain that
\beno
\begin{split}
|\Pi^{u'',\ep_2}|\lesssim 
&\ep \mathrm{CK}_{\la}+\ep\mathrm{CK}_{w}+\f{\ep^3}{\langle t\rangle^{2-\rmK_{\rmD}\ep/2}}
+\ep\mathrm{CK}_{\la}^{v,1}+\ep\mathrm{CK}_{w}^{v,1}\\
&+\ep\left\|\left\langle\f{\pa_v}{t\pa_z}\right\rangle^{-1}(\pa_z^2+(\pa_v-t\pa_z)^2)\left(\f{|\na|^{\f{s}{2}}}{\langle t\rangle^s}\rmA+\sqrt{\f{\pa_t w}{w}}\tilde{\rmA}\right)P_{\neq}(\Psi\Upsilon)\right\|_2^2. 
\end{split}
\eeno

\section{Appendix: Gevery spaces}
The physical space characterization of Gevrey functions is also used in this paper. We start with a characterization of the Gevery spaces on physical side. 
See {\it Lemma A.1} in \cite{IonescuJia} for the elementary proof. 
\begin{lemma}[\cite{IonescuJia}]\label{Lem: A1}
Suppose that $d=1,2$, $0<s<1$, $K>1$ and $g\in C^{\infty}(\R^d)$ with $\mathrm{supp}\, g\subset [a,b]^d$ satisfies the bounds
\ben
|D^{\al}g(x)|\leq K^m(m+1)^{m/s}, \quad x\in \R^d
\een
for all integers $m\geq 0$ and multi-indeices $\al$ with $|\al|=m$. Then
\ben
|\hat{g}(\xi)|\lesssim_{K,s} Le^{-\mu |\xi|^s},
\een
for all $\xi\in \R^d$ and some $\mu=\mu(K,s)>0$. 

Conversely, assume that, for some $\mu>0$ and $s\in (0,1)$,
\ben
|\hat{g}(\xi)|\leq Le^{-\mu |\xi|^s},
\een
for all $\xi\in \R^d$. Then there is $K>1$ depending on $s$ and $\mu$ such that 
 \ben
|D^{\al}g(x)|\lesssim_{\mu,s} K^m(m+1)^{m/s}, \quad x\in \R^d
\een
for all intergers $m\geq 0$ and multi-indeices $\al$ with $|\al|=m$.
\end{lemma}

For $x\in [a,b]^d$ and parameters $s\in (0,1)$ and $M\geq 1$, we define the spaces
\beq
\mathcal{G}_{ph}^{M,s}([a,b]^d)\eqdef
\left\{g: [a,b]^d\to \mathbb{C}:~\|g\|_{\mathcal{G}_{ph}^{M,s}([a,b]^d)}<\infty\right\}. 
\eeq
where 
\beq
\|g\|_{\mathcal{G}_{ph}^{M,s}([a,b]^d)}\eqdef \sup_{x\in [a,b]^d,\, m\geq 0,\, |\al|\leq m}\f{|D^{\al}g(x)|}{(m+1)^{m/s}M^{m}}. 
\eeq
Here `{\it ph}' represents the physical side. 

We define the spaces
\beq
\mathcal{G}_{ph,1}^{M,s}([a,b]^d)\eqdef
\left\{g: [a,b]^d\to \mathbb{C}:~\|g\|_{\mathcal{G}_{ph,1}^{M,s}([a,b]^d)}<\infty\right\}. 
\eeq
where 
\beq
\|g\|_{\mathcal{G}_{ph,1}^{M,s}([a,b]^d)}\eqdef \sup_{x\in [a,b]^d,\, m\geq 0,\, |\al|\leq m}\f{|D^{\al}g(x)|}{\Gamma_s(m)M^{m}},
\eeq
with $\Gamma_s(m)=(m!)^{\f1s}(m+1)^{-2}$, also see \cite{Yamanaka} for more details. 

By the Stirling's formula $N!\sim \sqrt{2\pi N}(N/e)^{N}$, it is easy to check that there exist $K_1<K_2$ such that 
\beno
K_1^{m}(m+1)^{m/s}\lesssim \Gamma_s(m)\lesssim  K_2^{m}(m+1)^{m/s},
\eeno
which implies
\ben\label{eq: space-embed}
\mathcal{G}_{ph}^{K_1M,s}([a,b]^d)\subset \mathcal{G}_{ph,1}^{M,s}([a,b]^d)\subset \mathcal{G}_{ph}^{K_2M,s}([a,b]^d). 
\een
We also have
\begin{equation}\label{eq:3.9}
\sum_{j=0}^{k}\f{k!}{j!(k-j)!}\Gamma_s(j)\Gamma_s(k-j)<\Gamma_s(k),
\end{equation}
which implies
\ben
\|fg\|_{\mathcal{G}_{ph,1}^{M,s}([a,b]^d)}\leq \|f\|_{\mathcal{G}_{ph,1}^{M,s}([a,b]^d)}\|g\|_{\mathcal{G}_{ph,1}^{M,s}([a,b]^d)}. 
\een

\begin{remark}\label{Rmk: fourier-gevrey}
Suppose $g\in \mathcal{G}_{ph,1}^{M,s}(\R^d)$ with $\mathrm{supp}\, g\subset [a,b]^d$, then for all integers $m\geq 0$ and multi-indeices $\al$ with $|\al|=m$. Then
\ben
|\hat{g}(\xi)|\lesssim_{K,s} Le^{-\mu |\xi|^s},
\een
for all $\xi\in \R^d$ and some $\mu=\mu(K,s)>0$. 
\end{remark}

We also introduce the composition lemma in \cite{Yamanaka}. 
\begin{lemma}[\cite{Yamanaka}]\label{eq: composition}
Let $I$, $J$ be real open intervals and let $f:J\to \R$ be a $C^{\infty}$-function such that $f'\in \mathcal{G}_{ph,1}^{L,s}(J)$ and $g: I\to J$ be a $C^{\infty}$-function such that $g'\in \mathcal{G}_{ph,1}^{M,s}(I)$. Let $N$ be a real constant such that
\beno
N\geq \max(M,L\|g'\|_{\mathcal{G}_{ph,1}^{M,s}(I)}).
\eeno
Then the derivative $(f\circ g)'$ of the composite function $f\circ g$ belongs to $\mathcal{G}_{ph,1}^{N,s}(I)$ and satisfies
\beno
\|(f\circ g)'\|_{\mathcal{G}_{ph,1}^{N,s}(I)}\leq \f{N}{L}\|f'\|_{\mathcal{G}_{ph,1}^{L,s}(J)}.
\eeno
\end{lemma}
We also introduce the estimate of inverse function in Gevrey class. 
\begin{lemma}[\cite{Yamanaka}]\label{eq: inverse-gevrey}
Let $I$, $J$ be real open intervals. Let $f:I\to J$ be a $C^{\infty}$-surjection such that
\beno
|f'(x)|\geq \f{1}{A},\quad x\in I
\eeno
for some $A>0$ and such that $f''\in \mathcal{G}_{ph,1}^{L,s}(I)$. Then $f$ has a $C^{\infty}$-inverse $f^{-1}:J\to I$ such that $(f^{-1})'\in  \mathcal{G}_{ph,1}^{M,s}(J)$ for some $M>0$. 
\end{lemma}

\section{Appendix: the wave operator}
\subsection{The existence of wave operator}\label{sect: The existence of wave operator}
In this subsection, we aim to construct the wave operator and prove Proposition \ref{prop: general-wave}. We will refer to some details from \cite{WZZ1}. Let us first recall {\it Definition 4.3, Lemma 4.4, Proposition 4.5 and Proposition 5.1} in \cite{WZZ1} that, there exist $\phi(k,y,y_c)$ and $\phi_1(k,y,y_c)$(here to make the notation good in this paper, we use $\phi(k,y,y_c)$ and $\phi_1(k,y,y_c)$ rather than $\phi(k,y,c)$ and $\phi_1(k,y,c)$ in \cite{WZZ1}) such that $\phi(k,y,y_c)=(u(y)-u(y_c))\phi_1(k,y,y_c)$ and
\beq\label{eq: hom-Rei}
\pa_{yy}\phi-k^2\phi-\f{u''(y)}{u(y)-u(y_c)}\phi=0,\quad
\phi(k,y_c,y_c)=0,\ \pa_y\phi(k,y_c,y_c)=u'(y_c). 
\eeq
Moreover, it holds that
\beno
(\mathrm{Id}-k^2T)(\phi_1)(k,y,y_c)=1,
\eeno
and $(\mathrm{Id}-k^2T)$ is invertible in $L^{\infty}_{y,y_c}([0,1]^2)$, and 
\beno
\phi_1(k,y,y_c)\approx \f{\sinh k(y-y_c)}{k(y-y_c)},
\eeno
where 
\beno
T\phi_1=T_0\circ T_{2,2}\phi_1=\int_{y_c}^y\f{1}{(u(y')-u(y_c))^2}\int_{y_c}^{y'}(u(y'')-u(y_c))^2\phi_1(y'',y_c)dy''dy'
\eeno
with 
\beno
T_0\phi_1=\int_{y_c}^y\phi_1(y',y_c)dy',
\eeno 
and 
\beno
T_{2,2}\phi_1=\f{1}{(u(y)-u(y_c))^2}\int_{y_c}^{y}(u(y'')-u(y_c))^2\phi_1(y'',y_c)dy''. 
\eeno

Then by {\it (7.1), (7.2) with $t=0$ and (7.3)} in \cite{WZZ1}, we obtain that for $k\neq 0$, it holds that
\beq\label{eq:rep-form}
\psi(k,y)=-\int_0^1\f{\rho(y_c)\bbD_{u,k}\left(\Delta_k\psi\right)(k,y_c)e(k,y,y_c)}{\sqrt{\left(u'(y_c)\rho(y_c)p.v.\int_0^1\f{1}{\phi(k,y',y_c)^2}dy'\right)^2+\left(\pi\rho(y_c)\f{u''(y_c)}{u'(y_c)^2}\right)^2}}dy_c
\eeq
with
\ben\label{eq:e}
e(k,y,y_c)=
\left\{
\begin{aligned}
&\phi(k,y,y_c)\int_{0}^y\frac{1}{\phi(k,y',y_c)^2}dy'\quad 0\leq y<y_c,\\
&\phi(k,y,y_c)\int_1^y\frac{1}{\phi(k,y',y_c)^2}dy'\quad y_c<y\leq 1,
\end{aligned}
\right.
\een
where 
\begin{align*}
\rho(y_c)&=(u(y_c)-u(0))(u(1)-u(y_c)), \\
u'(y_c)\rho(y_c)p.v.\int_0^1\f{1}{\phi(k,y,y_c)^2}dy
&\eqdef u(0)-u(1)-\rho(y_c)p.v.\int_0^1\f{u'(y)-u'(y_c)}{(u(y)-u(y_c))^2}dy\\
&\quad +u'(y_c)\rho(y_c)\int_0^1\f{1}{(u(y)-u(y_c))^2}\Big(\f{1}{\phi_1(k,y,y_c)^2}-1\Big)dy,
\end{align*}
and
\beq\label{eq: wave1}
\begin{split}
\bbD_{u,k}(\om)(k,y_c)&\eqdef \f{u'(y_c)\rho(y_c)p.v.\int_0^1\f{1}{\phi(k,y,y_c)^2}dy \ \om(y_c)}{\sqrt{\left(u'(y_c)\rho(y_c)p.v.\int_0^1\f{1}{\phi(k,y',y_c)^2}dy'\right)^2+\left(\pi\rho(y_c)\f{u''(y_c)}{u'(y_c)^2}\right)^2}}\\
&\quad+\f{\rho(y_c)u''(y_c)p.v. \int_0^1\f{\int_{y_c}^{y'}\om(y'')\phi_1(k,y'',y_c)dy''}{\phi(k,y',y_c)^2}dy'}{\sqrt{\left(u'(y_c)\rho(y_c)p.v.\int_0^1\f{1}{\phi(k,y',y_c)^2}dy'\right)^2+\left(\pi\rho(y_c)\f{u''(y_c)}{u'(y_c)^2}\right)^2}}.
\end{split}
\eeq
Let us now define
\begin{align}
\label{eq:d_1}
d_1(k,y)&=\f{u'(y)\rho(y)p.v.\int_0^1\f{1}{\phi(k,y',y)^2}dy'}{\sqrt{\left(u'(y)\rho(y)p.v.\int_0^1\f{1}{\phi(k,y',y)^2}dy'\right)^2+\left(\pi\rho(y)\f{u''(y)}{u'(y)^2}\right)^2}},\\
\label{eq:d_2}
d_2(k,y)&=-\f{\rho(y)}{\sqrt{\left(u'(y)\rho(y)p.v.\int_0^1\f{1}{\phi(k,y',y)^2}dy'\right)^2+\left(\pi\rho(y)\f{u''(y)}{u'(y)^2}\right)^2}}.
\end{align}
Thus we get that
\begin{align}\label{eq: D_uk}
\bbD_{u,k}(f)(k,y)=d_1(k,y)f(y)+u''(y)d_2(k,y)\int_0^1\f{e(k,y',y)}{u(y')-u(y)}f(y')dy'.
\end{align}

Then we turn to the proof of Proposition \ref{prop: general-wave}. 
\begin{proof}
By {\it Lemma 6.1 and (8.6)} in \cite{WZZ1} and the no eigenvalue assumption of the Rayleigh operator, we have 
\ben\label{eq: Wronskian}
\left(u'(y_c)\rho(y_c)p.v.\int_0^1\f{1}{\phi(k,y,y_c)^2}dy\right)^2+\left(\pi\rho(y_c)\f{u''(y_c)}{u'(y_c)^2}\right)^2\approx 1+k^2\rho(y_c)^2,
\een
and then by {\it (8.1)} in \cite{WZZ1}, it holds that
\ben
\|\bbD_{u,k}\om\|_2\leq C \|\om\|_2,
\een
here the constant $C$ is independent of $k$. 

Let $g\in H^2(0,1)\cap H_0^1(0,1)$ and $\Delta_k\psi=(\pa_{yy}-k^2)\psi=\om$, then by \eqref{eq:rep-form}, it holds that
\begin{align*}
&\int_0^1\om(y)g(y)dy=\int_0^1\psi(k,y)(\Delta_k g)(k,y)dy\\
&=\int_{0}^{1}\f{\int_0^1\f{\int_{y_c}^{y'}\Delta_kg(k,y'')\phi(k,y'',y_c)dy''}{\phi(k,y',y_c)^2}dy'\rho(y_c)\bbD_{u,k}(\om)(k,y_c)}{\sqrt{\left(u'(y_c)\rho(y_c)p.v.\int_0^1\f{1}{\phi(k,y',y_c)^2}dy'\right)^2+\left(\pi\rho(y_c)\f{u''(y_c)}{u'(y_c)^2}\right)^2}}dy_c.
\end{align*}
Integrating by parts, we obtain that
\begin{align*}
&\int_{y_c}^{y'}\Delta_kg(k,y'')\phi(k,y'',y_c)dy''\\
&=\int_{y_c}^{y'}g(k,y'')\Delta_k\phi(k,y'',y_c)dy''
+\pa_yg(k,y')\phi(k,y',y_c)-g(k,y')\pa_y\phi(k,y',y_c)+u'(y_c)g(k,y_c), 
\end{align*}
which gives us 
\begin{align*}
&p.v.\int_0^1
\f{\int_{y_c}^{y'}\Delta_kg(k,y'')\phi(k,y'',y_c)dy''}{\phi(k,y',y_c)^2}dy'\\
&=p.v.\int_0^1
\f{\int_{y_c}^{y'}g(k,y'')\f{u''(y'')\phi(k,y'',y_c)}{u(y'')-u(y_c)}dy''}{\phi(k,y',y_c)^2}dy'\\
&\quad+p.v.\int_0^1\f{\pa_yg(k,y')\phi(k,y',y_c)-g(k,y')\pa_y\phi(k,y',y_c)+u'(y_c)g(k,y_c)}{\phi(k,y',y_c)^2}dy'\\
&=p.v.\int_0^1
\f{\int_{y_c}^{y'}g(k,y'')u''(y'')\phi_1(k,y'',y_c)dy''}{\phi(k,y',y_c)^2}dy'
+\f{g(k,y_c)}{\rho(y_c)}u'(y_c)\rho(y_c)p.v.\int_0^1\f{1}{\phi(k,y,y_c)^2}dy,
\end{align*}
which implies that
\beno
\int_0^1\bbD_{u,k}(\om)(k,y_c){\bbD}_{u,k}^1(g)(k,y_c)dy_c=\int_0^1\om(k,y)g(k,y)dy,
\eeno
where
\beq\label{eq: wave1-dual}
\begin{split}
{\bbD}_{u,k}^1(g)(k,y_c)&\eqdef \f{u'(y_c)\rho(y_c)p.v.\int_0^1\f{1}{\phi(k,y,y_c)^2}dy g(y_c)}{\sqrt{\left(u'(y_c)\rho(y_c)p.v.\int_0^1\f{1}{\phi(k,y',y_c)^2}dy'\right)^2+\left(\pi\rho(y_c)\f{u''(y_c)}{u'(y_c)^2}\right)^2}}\\
&\quad+\f{\rho(y_c)p.v. \int_0^1\f{\int_{y_c}^{y'}u''(y'')g(y'')\phi_1(k,y'',y_c)dy''}{\phi(k,y',y_c)^2}dy'}{\sqrt{\left(u'(y_c)\rho(y_c)p.v.\int_0^1\f{1}{\phi(k,y',y_c)^2}dy'\right)^2+\left(\pi\rho(y_c)\f{u''(y_c)}{u'(y_c)^2}\right)^2}}.
\end{split}
\eeq
Thus by {\it (8.1)} in \cite{WZZ1} and \eqref{eq: Wronskian}, we have
\ben
\|\bbD_{u,k}^1g\|_2\leq C \|g\|_2,
\een
here the constant $C$ is independent of $k$. 

Due to the fact that $H^2\cap H_0^1$ is dense in $L^2$, for any $g\in L^2$, there is a sequence $\{g_n\}_{n\geq 1}$ such that $H^2\cap H_0^1 \ni g_n\to g$ in $L^2$ and 
\begin{align*}
\int_{0}^1\bbD_{u,k}(f)(y_c){\bbD}_{u,k}^1(g_n)(y_c)dy_c=\int_0^1f(y)g_n(y)dy. 
\end{align*}
Then 
\begin{align*}
&\left|\int_{0}^1\bbD_{u,k}(f)(y_c){\bbD}_{u,k}^1(g)(y_c)dy_c-\int_0^1f(k,y)g(k,y)dy\right|\\
&=\left|\int_{0}^1\bbD_{u,k}(f)(y_c){\bbD}_{u,k}^1(g_n)(y_c)dy_c-\int_0^1f(k,y)(g(y)-g_n(y))dy\right|\\
&\leq C\|\bbD_{u,k}(f)\|_{L^2}\|{\bbD}_{u,k}^1(g-g_n)\|_{L^2}+\|f\|_{L^2}\|g-g_n\|_{L^2}\to 0.
\end{align*}
Thus we have proved \eqref{eq: id1} and \eqref{eq: est1}. 

In order to get \eqref{eq:DR=uD}, we integrate by parts and use $\phi(k,y_c,y_c)=0, \pa_y\phi(k,y_c,y_c)=u'(y_c),\psi(k,0)=\psi(k,1)=0$ to arrive at
\begin{align*}
&p.v.\int_{0}^1\frac{\int_{y_c}^yu''(y')\psi(k,y')\phi_1(k,y',y_c)dy'}{\phi(k,y,y_c)^2}dy
=p.v.\int_{0}^1\frac{\int_{y_c}^y \psi(k,y')(\pa_y^2-k^2)\phi(k,y',y_c)dy'}{\phi(k,y,y_c)^2}dy\\
&=p.v. \int_{0}^1\frac{\int_{y_c}^y (\pa_y^2-k^2)\psi(k,y')\phi(k,y',y_c)dy'}{\phi(k,y,y_c)^2}dy
+p.v.\int_{0}^1\f{\psi(k,y) \pa_y\phi(k,y,y_c)-\psi(k,y_c)u'(y_c)}{\phi(k,y,y_c)^2} dy\\
&\quad-\int_{0}^1\f{\pa_y\psi(k,y) \phi(k,y,y_c)}{\phi(y,c)^2} dy\\
&=p.v.\int_{0}^1\frac{\int_{y_c}^yu(y')\om(k,y')\phi_1(k,y',y_c)dy'}{\phi(k,y,y_c)^2}dy-u'(y_c)p.v.\int_{0}^1\frac{\int_{y_c}^y\om(k,y')\phi_1(k,y',y_c)dy'}{\phi(k,y,y_c)^2}dy\\
&\quad-\f{\psi(k,y_c)}{\rho(y_c)}u'(y_c)\rho(y_c)p.v.\int_{0}^1\f{1}{\phi(k,y,y_c)^2} dy,
\end{align*}
which gives \eqref{eq:DR=uD}.
\end{proof}

By the \eqref{eq: id1}, it is easy to obtain that the inverse of $\bbD_{u,k}$ exists and has the following formula
\begin{align}\label{eq: D_uk^{-1}}
\bbD_{u,k}^{-1}(f)(k,y)&=d_1(k,y)f(y')+u''(y)\int_0^1\f{e(k,y,y')}{u(y)-u(y')}d_2(k,y')f(y')dy',
\end{align}
which gives Remark \ref{Rmk: Formula}. 

\subsection{The Fourier transform of the integral operator}
In this section, we make preparations to study the Gevrey regularity of the nonlocal part in the wave operator. Indeed, we will write the nonlocal part into the following four types of integral operators:
\begin{align*}
&\Pi_m(F)(c)=\int_{\R}K(u,c)\mu_m(u-c)e^{-ikt(u-c)}F(u)du,\quad m=1,2,3,4,\\
&\Pi_m^*(F)(u)=\int_{\R}K(u,c)\mu_m(u-c)e^{-ikt(u-c)}F(c)dc,\quad m=1,2,3,4,
\end{align*}
where $K(u_1,u)$ represent a smooth kernel with compact support which may vary from one line to the other and  $\mu_1\in C^{\infty}(\R)$ with compact support such that 
\beno
\mu_1(u)=\left\{\begin{aligned}
&1,|u|\leq u(1)-u(0),\\
&0,|u|\geq 2(u(1)-u(0)),
\end{aligned}\right.
\eeno
$\mu_2(u)=\mu_1(u)\mathbf{1}_{\R^-}(v)$ and $\mu_3(u)=\mu_1(u)\ln|u|$, $\mu_4(u)=p.v.\f{1}{u}$. 

It is easy to check that $\Pi_m^*, \, m=1,2,3,4$ are the dual operators of $\Pi_m,\, m=1,2,3,4$. 

For a smooth kernel with compact support defined on $\R^2$, let $\widehat{\widehat{K}}$ be the Fourier transform in both variables. 
\begin{lemma}\label{lem: Fourier_type1}
Suppose $K(u,c)\in C^{\infty}(\R^2)$ with compact support, then it holds for $m=1,2,3,4$ that
\beno
\widehat{\Pi_m(F)}(\eta)=\f{1}{2\pi}\int_{\R^2}\hat{F}(\xi)\widehat{\widehat{K}}(-\xi-\xi', \eta+\xi')\widehat{\mu_m}(\xi'+kt)d\xi'd\xi,
\eeno
and
\beno
\widehat{\Pi_m^*(F)}(\xi)=\f{1}{2\pi}\int_{\R^2}\hat{F}(\eta)\widehat{\widehat{K}}(-\xi-\xi', \eta+\xi')\widehat{\mu_m}(\xi'+kt)d\xi'd\eta.
\eeno
\end{lemma}
\begin{proof}
Let $K_{L}(u-c,c)=K(u,c)$, then we have that 
\begin{align*}
\Pi_m(F)(c)&=\int_{\R}K_L(u-c,c)\mu_m(u-c)e^{-ikt(u-c)}F(u)dv\\
&=\f{1}{2\pi}\int_{\R}K_L(u-c,c)\mu_m(u-c)e^{-ikt(u-c)}\int_{\R}\hat{F}(\xi)e^{i\xi u}d\xi du\\
&=\f{1}{2\pi}\int_{\R}\hat{F}(\xi)e^{ic\xi}\int_{\R}K_L(u-c,c)\mu_m(u-c)e^{-ikt(u-c)}e^{i\xi (u-c)} dud\xi\\
&=\f{1}{2\pi}\int_{\R}\hat{F}(\xi)e^{ic\xi}\mathcal{F}_1\big(K_L(\cdot, c)\mu_m\big)(-\xi+kt)d\xi\\
&=\f{1}{2\pi}\int_{\R}\int_{\R}\hat{F}(\xi)e^{ic\xi}\widehat{K_L}(-\xi-\xi', c)\widehat{\mu_m}(\xi'+kt)d\xi'd\xi.
\end{align*}

Thus we get that
\begin{align*}
\widehat{\Pi_m(F)}(\eta)=\f{1}{2\pi}\int_{\R^2}\hat{F}(\xi)\widehat{\widehat{K_L}}(-\xi-\xi', \eta-\xi)\widehat{\mu_m}(\xi'+kt)d\xi'd\xi.
\end{align*}
We also have
\begin{align*}
\widehat{\widehat{K_L}}(\xi, \eta)
&=\int_{\R^2}K_L(w,c)e^{-iw\xi-ic\eta}dwdc\\
&=\int_{\R^2}K_L(v-c,c)e^{-i(u-c)\xi-ic\eta}dudc\\
&=\int_{\R^2}K(v,c)e^{-i(u-c)\xi-ic\eta}dudc
=\widehat{\widehat{K}}(\xi,\eta-\xi),
\end{align*}
which gives the lemma. 
\end{proof}
\begin{remark}\label{Rmk: fo-g-2}
Suppose $K(u,c)\in \mathcal{G}_{ph,1}^{M,s_0}(\R^2)$ with compact support in $[u(0),u(1)]^2$, then there exists $\la=\la(M,s_0)$, such that for $m=1,2,3,4,$
\beno
\left|\int\widehat{\widehat{K}}(-\xi-\xi', \eta+\xi')\widehat{\mu_m}(\xi'+kt)d\xi'\right|\lesssim e^{-\la|\eta-\xi|^{s_0}}. 
\eeno
\end{remark}
This remark follows directly from Remark \ref{Rmk: fourier-gevrey} and the fact that
\beno
\int_{\R}e^{-\la'(|\xi+\xi'|^2+|\eta+\xi'|^2)^{\f{s_0}{2}}}d\xi'
\lesssim \int_{\R}\f{e^{-\la|\xi-\eta|^{s_0}}}{1+|\xi+\xi'|^2+|\eta+\xi'|^2}d\xi'\lesssim e^{-\la|\xi-\eta|^{s_0}},
\eeno
holds for $0<\la<\la'$. 

\subsection{Gevrey regularity}\label{Sec:Gevrey regularity}
In this subsection, we will study the Gevrey regularity of the coefficients $d_1, d_2$ and the kernel $e$ of the wave operator and prove Proposition \ref{prop: kernel-wave-op}. 

A direct calculation gives us that
\begin{align*}
&\phi(k,y,y_c)\int_j^y\f{1}{\phi(k,y',y_c)^2}dy'\\
&=\phi(k,y,y_c)\int^y\f{1}{(u(y')-u(y_c))^2}\Big(\f{1}{\phi_1(k,y',y_c)^2}-1\Big)dy'\\
&\quad+\phi(k,y,y_c)\int^y\f{1}{(u(y')-u(y_c))^2}dy'\\
&=\phi(k,y,y_c)\int_j^y\f{1}{(u(y')-u(y_c))^2}\Big(\f{1}{\phi_1(k,y',y_c)^2}-1\Big)dy'\\
&\quad+\phi(k,y,y_c)\int_{u(j)}^{u(y)}\f{(u^{-1})'(u_1)-(u^{-1})'(u(y_c))-(u^{-1})''(u(y_c))(u_1-u(y_c))}{(u_1-u(y_c))^2}du_1\\
&\quad-\f{1}{u'(y_c)}\f{\phi(k,y,y_c)}{u(y')-u(y_c)}\bigg|_{y'=j}^{y'=y}
+\phi(k,y,y_c)\f{u''(y_c)}{u'(y_c)^3}\ln |u(y')-u(y_c)|\bigg|_{y'=j}^{y'=y}.
\end{align*}
Thus we have that 
\beno
\phi(k,y,y_c)\int_{j}^y\f{1}{\phi(k,y',y_c)^2}dy'
=-\f{1}{u'(y_c)}+\f{u''(y_c)}{u'(y_c)^3}\phi(k,y,y_c)\ln |u(y)-u(y_c)|+\phi_j^{re}(k,y,y_c),
\eeno
where 
\begin{align*}
&\phi_j^{re}(k,y,y_c)\\
&=\phi(k,y,y_c)\int_j^y\f{1}{(u(y')-u(y_c))^2}\Big(\f{1}{\phi_1(k,y',y_c)^2}-1\Big)dy'\\
&\quad+\phi(k,y,y_c)\int_{u(j)}^{u(y)}\f{(u^{-1})'(u_1)-(u^{-1})'(u(y_c))-(u^{-1})''(u(y_c))(u_1-u(y_c))}{(u_1-u(y_c))^2}du_1\\
&\quad+\f{1}{u'(y_c)}(1-\phi_1(k,y,y_c))
+\f{1}{u'(y_c)}\f{(u(y)-u(y_c))\phi_1(k,y,y_c)}{u(j)-u(y_c)}\\
&\quad-\phi(k,y,y_c)\f{u''(y_c)}{u'(y_c)^3}\ln |u(j)-u(y_c)|. 
\end{align*}
Therefore, 
\begin{align*}
e(k,y,y_c)&=-\f{1}{u'(y_c)}+\f{u''(y_c)}{u'(y_c)^3}\phi(k,y,y_c)\ln|u(y)-u(y_c)|\\
&\quad+\phi^{re}_0(k,y,y_c)(1_{\R^-}(y-y_c))+\phi^{re}_1(k,y,y_c)(1_{\R^+}(y-y_c)). 
\end{align*}
Let us define
\begin{align*}
\rmw_1(k,y_c)&=\f{-\f{1}{u'(y_c)}+\f{u''(y_c)}{u'(y_c)^3}\phi(k,1,y_c)\ln|u(1)-u(y_c)|+\phi_0^{re}(k,1,y_c)}{\phi_1(k,1,y_c)}(u(y_c)-u(0))u'(y_c),
\end{align*}
then, we get
\begin{align*}
d_1(k,y_c)=\f{\rmw_1(k,y_c)}{\sqrt{\rmw_1(k,y_c)^2+\left(\pi\rho(y_c)\f{u''(y_c)}{u'(y_c)^2}\right)^2}}.
\end{align*}
Making the change of coordinate $(y,y_c)\to (u,c)$ and using the definition \eqref{eq:DE}, we also have
\begin{align}
\label{eq: D_1}
D_1(k,c)&=\f{\rmW_1(k,c)}{\sqrt{\rmW_1(k,c)^2+\left(\pi\trho(c)\f{\widetilde{u''}(c)}{\widetilde{u'}(c)^2}\right)^2}},\\
\label{eq: D_2}
D_2(k,c)&=-\f{\trho(c)}{\sqrt{\rmW_1(k,c)^2+\left(\pi\trho(c)\f{\widetilde{u''}(c)}{\widetilde{u'}(c)^2}\right)^2}},\\
\label{eq: E}
E(k,u,c)&=-\f{1}{\widetilde{u'}(c)}+(u^{-1})''(c)(u-c)\Phi_1(k,u,c)\ln|u-c|\\
\nonumber
&\quad+\Phi^{re}_0(k,u,c)(1_{\R^-}(u-c))+\Phi^{re}_1(k,u,c)(1_{\R^+}(u-c)),
\end{align}
where 
\begin{align}
\nonumber
\trho(c)&=(u(1)-c)(c-u(0)),\ \widetilde{u'}(c)=u'\circ u^{-1}(c)=\f{1}{(u^{-1})'(c)},\  \widetilde{u''}(c)=u''\circ u^{-1}(c),\\
\nonumber
\rmW_1(k,c)&=\f{-c+u(0)}{\Phi_1(k,u(1),c)}+\f{(u^{-1})''(c)\trho(c)\ln|u(1)-c|}{(u^{-1})'(c)}+\f{\Phi_0^{re}(k,u(1),c)}{\Phi_1(k,u(1),c)}\f{c-u(0)}{(u^{-1})'(c)},\\
\label{eq: Phi^{re}_j}
\Phi^{re}_j(k,u,c)&=(u-c)\bigg[\Phi_1(k,u,c)\int_{u(j)}^u\f{1}{(u_1-c)^2}\Big(\f{1}{\Phi_1(k,u_1,c)^2}-1\Big)(u^{-1})'(u_1)du_1\\
\nonumber&\quad+\Phi_1(k,u,c)\int_{u(j)}^{u}\f{(u^{-1})'(u_1)-(u^{-1})'(c)-(u^{-1})''(c)(u_1-c)}{(u_1-c)^2}du_1\\
\nonumber&\quad+\f{1}{\tilde{u'}(c)}\f{1-\Phi_1(k,u,c)}{u-c}
+\f{1}{\tilde{u'}(c)}\f{\Phi_1(k,u,c)}{u(j)-c}
-\Phi_1(k,u,c)(u^{-1})''(c)\ln |u(j)-c|\bigg]\\
\label{eq: Phi^{re}_{j,1}}
&\eqdef(u-c)\Phi^{re}_{j,1}(k,u,c).
\end{align}
with $\Phi_1(k,u,c)$ satisfying 
\beno
\Phi_1(k,u(y),u(y_c))=\phi_1(k,y,y_c). 
\eeno

We divide the nonlocal part in the wave operator into four different types of integral operators whose kernels have different singularities:
\begin{align}\label{eq: bfD}
\bfD_{u,k}\big(\mathcal{F}_{1}\tilde{f}(t,k,\cdot)\big)(t,k,u)
=D_1(k,u)\mathcal{F}_{1}\tilde{f}(t,k,u)+\sum_{j=1}^4\Pi_j(\tilde{f}),
\end{align}
where
\begin{align*}
\Pi_1(\tilde{f})&=\widetilde{u''}(u)D_2(k,u)\int_{u(0)}^{u(1)}{\mathcal{F}_{1}\tilde{f}(t,k,u_1)e^{-i(u_1-u)tk}}(u^{-1})'(u_1)\Phi^{re}_{1,1}(k,u_1,u)du_1,\\
\Pi_2(\tilde{f})&=\widetilde{u''}(u)D_2(k,u)\int_{u(0)}^{u(1)}{\mathcal{F}_{1}\tilde{f}(t,k,u_1)e^{-i(u_1-u)tk}}(u^{-1})'(u_1)\Phi^{re}_{0,1}(k,u_1,u)(1_{\R^-}(u_1-u))du_1\\
&\quad-\widetilde{u''}(u)D_2(k,u)\int_{u(0)}^{u(1)}{\mathcal{F}_{1}\tilde{f}(t,k,u_1)e^{-i(u_1-u)tk}}(u^{-1})'(u_1)\Phi^{re}_{1,1}(k,u_1,u)1_{\R^-}(u_1-u)du_1,\\
\Pi_3(\tilde{f})&=\widetilde{u''}(u)D_2(k,u)\int_{u(0)}^{u(1)}{\mathcal{F}_{1}\tilde{f}(t,k,u_1)e^{-i(u_1-u)tk}}(u^{-1})'(u_1)(u^{-1})''(u)\Phi_1(k,u_1,u)\ln|u_1-u|du_1,\\
\Pi_4(\tilde{f})&=\widetilde{u''}(u)D_2(k,u)\int_{u(0)}^{u(1)}\f{\mathcal{F}_{1}\tilde{f}(t,k,u_1)e^{-i(u_1-u)tk}}{u_1-u}\f{(u^{-1})'(u_1)}{(u^{-1})'(u)}du_1.
\end{align*} 

To study the Gevrey regularity of the wave operator, we only need to study the smooth kernels $\Phi_1$, $\Phi_{0,1}^{re}$, $\Phi_{1,1}^{re}$ and the coefficients $D_1$, $D_2$. Indeed, all these functions are related to $\Phi_1$ closely. Once we obtain the estimate of $\Phi_1$, we can deduce the estimates of other functions easily. 

Now let us study the regularity of $\Phi_1$. 

Recall that $\Phi(k,u(y),u(y_c))=\phi(k,y,y_c)$, $\Phi_1(k,u(y),u(y_c))=\phi_1(k,y,y_c)$, $\Phi(k,u,c)=(u-c)\Phi_1(k,u,c)$ and
\ben\label{eq: Phi_1}
\begin{split}
\Phi_1(k,u,c)
&=1+k^2\int_{c}^u\f{(u^{-1})'(u_1)}{(u_1-c)^2}\int_c^{u_1}(u_2-c)^2(u^{-1})'(u_2)\Phi_1(k,u_2,c)d u_2 du_1\\
&=1+k^2\mathrm{T}_0\circ \mathrm{T}_{2,2}\Phi_1,
\end{split}
\een
with $\rmT_0\Phi_1(k,u,c)=\int_{c}^u\Phi_1(k,u_1,c)du_1$ and for $j_1\leq j_2+1$, $j_2=0,1,2,...,$
\beno
\rmT_{j_1,j_2}\Phi_1(k,u_1,c)=\f{(u^{-1})'(u_1)}{(u_1-c)^{j_1}}\int_c^{u_1}(u_2-c)^{j_2}(u^{-1})'(u_2)\Phi_1(k,u_2,c)d u_2.
\eeno

For any $k\neq 0$, let $M\geq 1$ be a large constant, then we define the following weighted Gevery spaces:
\ben
\mathcal{G}^{M|k|,s}_{ph,\cosh}([a,b]^2)\eqdef \{f\in C^{\infty}([a,b]^2):~\|f\|_{\mathcal{G}^{M|k|,s}_{ph,\cosh}([a,b]^2)}<\infty\}, 
\een
with 
\beq
\|f\|_{\mathcal{G}^{M|k|,s}_{ph,\cosh}([a,b]^2)}\eqdef
\sup_{(u,c)\in [u(0),u(1)]^2,\, m\geq m_1\geq 0}\f{|\pa_u^{m-m_1}(\pa_c+\pa_u)^{m_1}f(u,c)|}{\Gamma_s(m)M^m|k|^{m}\cosh M|k|(u-c)}. 
\eeq

\begin{remark}\label{rmk: Gevery1}
It holds that
\begin{align*}
\sup_{(u,c)\in [u(0),u(1)]^2,\, m\geq m_1\geq 0}\f{|\pa_u^{m-m_1}\pa_c^{m_1}f(u,c)|}{\Gamma_s(m)2^{m_1}M^m|k|^{m}\cosh M|k|(u-c)}\leq \|f\|_{\mathcal{G}^{M|k|,s}_{ph,\cosh}([a,b]^2)}. 
\end{align*}
\end{remark}
\begin{proof}
We have 
\begin{align*}
&\sup_{(u,c)\in [u(0),u(1)]^2,\, m\geq m_1\geq 0}\f{|\pa_u^{m-m_1}\pa_c^{m_1}f(u,c)|}{\Gamma_s(m)2^{m_1}M^m|k|^{m}\cosh M|k|(u-c)}\\
&\leq \sup_{m_1\geq 0}2^{-m_1}\sum_{m_2=0}^{m_1}\f{m_1!}{m_2!(m_1-m_2)!}\sup_{(u,c)\in [u(0),u(1)]^2,\, m\geq m_2\geq 0}\f{|\pa_u^{m-m_2}(\pa_c+\pa_u)^{m_2}f(u,c)|}{\Gamma_s(m)M^m|k|^{m}\cosh M|k|(u-c)}\\
&\leq \sup_{m_1\geq 0}2^{-m_1}\sum_{m_2=0}^{m_1}\f{m_1!}{m_2!(m_1-m_2)!}\|f\|_{\mathcal{G}^{M|k|,s}_{ph,\cosh}([a,b]^2)}\leq \|f\|_{\mathcal{G}^{M|k|,s}_{ph,\cosh}([a,b]^2)}. 
\end{align*}
Thus we proved the remark.
\end{proof}

\begin{remark}
It holds that
\begin{align*}
\|fg\|_{\mathcal{G}^{M|k|,s}_{ph,\cosh}([a,b]^2)}
\leq C\|f\|_{\mathcal{G}^{M|k|,s}_{ph,\cosh}([a,b]^2)}\|g\|_{\mathcal{G}^{M|k|,s}_{ph,1}([a,b]^2)}. 
\end{align*}
\end{remark}

\begin{proposition}\label{prop: T}
Suppose $(u^{-1})''\in \mathcal{G}^{K_u,s_0}_{ph,1}([u(0),u(1)])$, then there exists $M_0=M_0(K_u)\geq 1$, such that for any $M\geq M_0$, it holds that
\beno
\|\rmT_0\circ\rmT_{2,2}\|_{\mathcal{G}^{M|k|,s_0}_{ph,\cosh}([u(0),u(1)]^2)\to \mathcal{G}^{M|k|,s_0}_{ph,\cosh}([a,b]^2)}\leq \f{C}{M^2k^2}. 
\eeno
\end{proposition}
\begin{proof}
It is easy to check that
\begin{align*}
&(\pa_u+\pa_c)^{m}\rmT_0\circ\rmT_{2,2}f\\
&=\sum_{k_1+k_2\leq m}\f{m!}{(m-k_1-k_2)!k_1!k_2!}\\
&\qquad \times\int_{c}^u\f{(u^{-1})^{(k_1+1)}(u_1)}{(u_1-c)^2}\int_c^{u_1}(u_2-c)^2(u^{-1})^{(k_2+1)}(u_2)(\pa_{u_2}+\pa_c)^{m-k_1-k_2}f(k,u_2,c)d u_2 du_1,
\end{align*}
which gives us that
\begin{align*}
&\left|\f{(\pa_u+\pa_c)^{m}\rmT_0\circ\rmT_{2,2}f}{\Gamma_{s_0}(m)M^m|k|^m\cosh M|k|(u-c)}\right|\\
&\leq \sum_{k_1+k_2\leq m}\f{m!}{(m-k_1-k_2)!k_1!k_2!}
\Gamma_{s_0}(m-k_1-k_2)\Gamma_{s_0}(k_1)\Gamma_{s_0}(k_2)\\
&\qquad \times\int_{c}^u\f{1}{(u_1-c)^2}\int_c^{u_1}(u_2-c)^2\cosh M|k|(u_2-c)d u_2 du_1\\
&\qquad\times \sup_{0\leq m_1\leq m}\left|\f{(\pa_u+\pa_c)^{m_1}f}{\Gamma_{s_0}(m_1)M^{m_1}|k|^{m_1}\cosh M|k|(u-c)}\right|\|(u^{-1})'\|_{\mathcal{G}^{M|k|,s_0}_{ph,1}([u(0),u(1)])}\\
&\leq \f{C}{M^2k^2}\sup_{0\leq m_1\leq m}\left|\f{(\pa_u+\pa_c)^{m_1}f}{\Gamma_{s_0}(m_1)M^{m_1}|k|^{m_1}\cosh M|k|(u-c)}\right|. 
\end{align*}
which implies 
\beq
\sup_{m\geq 0}\left|\f{(\pa_u+\pa_c)^{m}\rmT_0\circ\rmT_{2,2}f}{\Gamma_{s_0}(m)M^m|k|^m\cosh M|k|(u-c)}\right|
\leq \f{C}{M^2k^2}\sup_{m\geq 0}\left|\f{(\pa_u+\pa_c)^{m}f}{\Gamma_{s_0}(m)M^{m}|k|^{m}\cosh M|k|(u-c)}\right|. 
\eeq
We also have 
\begin{align*}
&\pa_u\rmT_0\circ\rmT_{2,2}f=\rmT_{2,2}f\\
&\pa_u^2\rmT_0\circ\rmT_{2,2}f=(u-c)\f{(u^{-1})''(u)}{(u^{-1})'(u)}\rmT_{3,2}f-2\rmT_{3,2}f+(u^{-1})'(u)^2f(u,c)
\end{align*}
{\bf Claim:} It holds for any $s\in (0,1)$ and $j\geq 0$ that 
\beq
\begin{split}
&\sup_{m_1,\,m\geq 0}\left\|\f{\pa_u^{m_1}(\pa_u+\pa_c)^{m-m_1}\Big(\f{1}{(u-c)^{j+1}}\int_c^u(u'-c)^{j}f(u',c)du'\Big)}{\Gamma_{s}(m)M^m|k|^m\cosh M|k|(u-c)}\right\|_{L^{\infty}([u(0),u(1)]^2)}\\
&\qquad\leq C\sup_{m_1, m\geq 0}\left\|\f{(\pa_c+\pa_u)^{m-m_1}\pa_u^{m_1}f(u,c)}{\Gamma_{s}(m)M^m|k|^m\cosh M|k|(u-c)}\right\|_{L^{\infty}([u(0),u(1)]^2)}. 
\end{split}
\eeq
Let admit the claim and finish the proof of the proposition first. Indeed, it is easy to check that
\begin{align*}
\sum_{m=0}^{2}\sum_{m_1=0}^{m}\left\|\f{\pa_u^{m_1}\pa_c^{m-m_1}\rmT_0\circ\rmT_{2,2}f}{(M|k|)^m\cosh M|k|(u-c)}\right\|_{L^{\infty}([u(0),u(1)]^2)}
\leq \f{C}{M^2k^2}\sum_{m=0}^2\left\|\f{\pa_c^mf}{\cosh M|k|(u-c)}\right\|_{L^{\infty}([u(0),u(1)]^2)}. 
\end{align*}
For more details of the above estimate one can refer to \cite{WZZ1}. 

For $m_1\geq 2$, by using the claim, we have that 
\begin{align*}
&\left\|\f{(\pa_c+\pa_u)^{m-m_1}\pa_u^{m_1}\rmT_0\circ\rmT_{2,2}f}{\Gamma_{s_0}(m)(M|k|)^m\cosh M|k|(u-c)}\right\|_{L^{\infty}([u(0),u(1)]^2)}\\
&\leq \left\|\f{(\pa_c+\pa_u)^{m-m_1}\pa_u^{m_1-2}\Big((u-c)\f{(u^{-1})''(u)}{(u^{-1})'(u)}\rmT_{3,2}f\Big)}{\Gamma_{s_0}(m)(M|k|)^m\cosh M|k|(u-c)}\right\|_{L^{\infty}([u(0),u(1)]^2)}\\
&\quad+2\left\|\f{(\pa_c+\pa_u)^{m-m_1}\pa_u^{m_1-2}\rmT_{3,2}f}{\Gamma_{s_0}(m)(M|k|)^m\cosh M|k|(u-c)}\right\|_{L^{\infty}([u(0),u(1)]^2)}\\
&\quad+\left\|\f{(\pa_c+\pa_u)^{m-m_1}\pa_u^{m_1-2}((u^{-1})'(u)^2f(u,c))}{\Gamma_{s_0}(m)(M|k|)^m\cosh M|k|(u-c)}\right\|_{L^{\infty}([u(0),u(1)]^2)}\\
&\leq \f{C}{M^2k^2}\sup_{m_2\leq m_1-2,\, m_3\leq m-m_1}\left\|\f{(\pa_c+\pa_u)^{m_3}\pa_u^{m_2}f(u,c)}{\Gamma_{s_0}(m_2+m_3)(M|k|)^{m_2+m_3}\cosh M|k|(u-c)}\right\|_{L^{\infty}([u(0),u(1)]^2)}. 
\end{align*}

Finally, let us prove the claim. Indeed we have
\begin{align*}
&\pa_u^{m_1}(\pa_u+\pa_c)^{m-m_1}\left(\f{1}{(u-c)^{j+1}}\int_c^u(u'-c)^{j}f(u',c)du'\right)\\
&=\pa_u^{m_1}(\pa_u+\pa_c)^{m-m_1}\left(\int_0^1t^jf(c+(u-c)t,c)dt\right)\\
&=\int_0^1t^jt^{m_1}\big(\pa_u^{m_1}(\pa_u+\pa_c)^{m-m_1}f\big)(c+(u-c)t,c)dt,
\end{align*}
which gives that
\begin{align*}
&\left\|\f{\pa_u^{m_1}(\pa_u+\pa_c)^{m-m_1}\left(\f{1}{(u-c)^{j+1}}\int_c^u(u'-c)^{j}f(u',c)du'\right)}{\cosh M|k|(u-c)}\right\|_{L^{\infty}[u(0),u(1)]^2}\\
&\leq \f{\int_0^1t^jt^{m_1}\cosh tM|k|(u-c)dt}{\cosh M|k|(u-c)}\left\|\f{\pa_u^{m_1}(\pa_u+\pa_c)^{m-m_1}f}{\cosh M|k|(u-c)}\right\|_{L^{\infty}[u(0),u(1)]^2}\\
&\leq C\left\|\f{\pa_u^{m_1}(\pa_u+\pa_c)^{m-m_1}f}{\cosh M|k|(u-c)}\right\|_{L^{\infty}[u(0),u(1)]^2}.
\end{align*}
Thus we proved the proposition. 
\end{proof}

By \eqref{eq: Phi_1}, Proposition \ref{prop: T} and Remark \ref{rmk: Gevery1}, we have the following corollary:
\begin{corol}\label{corol: regular Phi}
Suppose $u$ is the same as in Proposition \ref{prop: T} and $\Phi_1$ satisfies \eqref{eq: Phi_1}, then there exist $M_0$ and $C(k)>0$ such that for any $|k|\neq 0$, for all integers $m\geq m_1\geq 0$ and $M\geq M_0$, 
\beno
\left|\pa_u^{m_1}\pa_c^{m-m_1}\Phi_1(k,u,c)\right|\leq C(k)\Gamma_{s_0}(m)(2M)^m|k|^m.
\eeno
Moreover there exists $C\geq 1$ such that
\beno
\left|\pa_u^{m_1}\pa_c^{m-m_1}\Big(\f{\Phi_1(k,u,c)-1}{(u-c)^2}\Big)\right|\leq C(k)\Gamma_{s_0}(m)(CM)^m|k|^m. 
\eeno
\end{corol}
\begin{proof}
The second inequality follows directly from the first inequality and the fact that $\Phi_1(k,u,c)\big|_{u=c}=1$ and $\pa_c\Phi_1(k,u,c)\big|_{u=c}=\pa_u\Phi_1(k,u,c)=0\big|_{u=c}$. So we only need to prove the first inequality. 

By taking $M_0$ large enough, it follows from \eqref{eq: Phi_1} and Proposition \ref{prop: T} that
\begin{align*}
\|\Phi_1\|_{\mathcal{G}^{M|k|,s_0}_{ph,\cosh}([a,b]^2)}
&\leq \|k^2\rmT_0\circ\rmT_{2,2}\Phi_1\|_{\mathcal{G}^{M|k|,s_0}_{ph,\cosh}([a,b]^2)}+C\\
&\leq \f{C}{M^2}\|\Phi_1\|_{\mathcal{G}^{M|k|,s_0}_{ph,\cosh}([a,b]^2)}+C,
\end{align*}
which gives us that
\beno
\|\Phi_1\|_{\mathcal{G}^{M|k|,s_0}_{ph,\cosh}([a,b]^2)}\leq C. 
\eeno
Then the corollary follows from Remark \ref{rmk: Gevery1}. 
\end{proof}

\begin{corol}\label{corol: Phi_j^re}
For $j=0,1$, let $\Phi_{j}^{re}$ and $\Phi^{re}_{j,1}$ be defined as in \eqref{eq: Phi^{re}_j} and \eqref{eq: Phi^{re}_{j,1}}, then there exists $C=C(\th_0,M,k)$ such that for any $|k|\neq 0$, for all integers $m\geq m_1\geq 0$,
\beno
\sup_{(c,u)\in [u(\f{\th_0}{2}),u(1-\f{\th_0}{2})]^2}\left|\pa_u^{m_1}\pa_c^{m-m_1}\Phi_{j}^{re}(k,u,c)\right|\leq C\Gamma_{s_0}(m)(CM)^m|k|^m,
\eeno
and
\beno
\sup_{(c,u)\in [u(\f{\th_0}{2}),u(1-\f{\th_0}{2})]^2}\left|\pa_u^{m_1}\pa_c^{m-m_1}\Big(\Phi_{j,1}^{re}(k,u,c)\Big)\right|\leq C\Gamma_{s_0}(m)(CM)^m|k|^m. 
\eeno
\end{corol}
\begin{proof}
By the second inequality in Corollary \ref{corol: regular Phi}, we obtain that for all integers $m\geq m_1\geq 0$, there exist $C$ such that for $j=0,1$
\begin{align*}
\sup_{(c,u)\in [u(0),u(1)]^2}&\left|\pa_u^{m_1}\pa_c^{m-m_1}\int_{u(j)}^u\f{1}{(u_1-c)^2}\Big(\f{1}{\Phi_1(k,u_1,c)^2}-1\Big)(u^{-1})'(u_1)du_1\right|\\
&\leq C\Gamma_{s_0}(m)(CM)^m|k|^m. 
\end{align*}
By the fact that 
\begin{align*}
&\f{(u^{-1})'(u)-(u^{-1})'(c)-(u^{-1})''(c)(u-c)}{(u-c)^2}\\
&=\f{1}{(u-c)^2}\int_c^{u}\int_{c}^{u_1}(u^{-1})'''(u_2)du_2du_1
=\int_0^1\int_0^1t(u^{-1})'''(c+st(u-c))dsdt,
\end{align*}
we obtain that there is $C\geq 1$ such that 
\beno
\left\|\f{(u^{-1})'(u)-(u^{-1})'(c)-(u^{-1})''(c)(u-c)}{(u-c)^2}\right\|_{\mathcal{G}_{ph,1}^{CK_u,s_0}([u(0),u(1)]^2)}\leq \|(u^{-1})''\|_{\mathcal{G}_{ph,1}^{K_u,s_0}([u(0),u(1)])},
\eeno
which gives us that for $j=0,1$ 
\beno
\left\|\int_{u(j)}^{u}\f{(u^{-1})'(u_1)-(u^{-1})'(c)-(u^{-1})''(c)(u_1-c)}{(u_1-c)^2}du_1\right\|_{\mathcal{G}_{ph,1}^{\tilde{C}K_u,s_0}([u(0),u(1)]^2)}\leq \|(u^{-1})''\|_{\mathcal{G}_{ph,1}^{K_u,s_0}([u(0),u(1)])}.
\eeno
We also have that for some $C(\th_0)$, 
\beno
\sup_{c\in [u(\f{\th_0}{2}),u(1-\f{\th_0}{2})]}\left(\left|\pa_c^m\Big(\f{1}{u(j)-c}\Big)\right|+\left|\pa_c^m\Big(\ln|u(j)-c|\Big)\right|\right)\leq C^mm!. 
\eeno
Thus together with Corollary \ref{corol: regular Phi}, Lemma \ref{eq: inverse-gevrey}, we obtain the Corollary. 
\end{proof}

Let us now turn to the proof of Proposition \ref{prop: kernel-wave-op}. 
\begin{proof}
By the definition of the Gevrey class $\mathcal{G}_{ph,1}^{M,s_0}([u(0),u(1)]^d)$ it is easy to check that for any $f\in\mathcal{G}_{ph,1}^{M,s_0}([u(0),u(1)]^d)$, its trace will be in the Gevrey class $\mathcal{G}_{ph,1}^{M,s_0}([u(0),u(1)]^{d-1})$ of the same regularity. 

By \eqref{eq: bfD} and Lemma \ref{lem: Fourier_type1}, we get that there exists $\mathcal{D}(t,k,\xi_1,\xi_2)$ such that
\beno
\mathcal{F}_{2}\Big(\bfD_{u,k}\big(\tchi_2\mathcal{F}_{1}f(t,k,\cdot)\big)\Big)(t,k,\xi_1)=\int \mathcal{D}(t,k,\xi_1,\xi_2)\hat{f}_k(t,\xi_2)d\xi_2. 
\eeno
The behavior of $\mathcal{D}(t,k,\xi_1,\xi_2)$ depends only on the regularity of $D_1$, $D_2$ and $E$. 
Indeed, by Corollary \ref{corol: regular Phi}, Corollary \ref{corol: Phi_j^re}, Lemma \ref{eq: inverse-gevrey} and \eqref{eq: Wronskian}, we obtain that there exist $M_u$ such that for all integers $m\geq m_1\geq 0$ it holds that
\begin{align*}
&\left\|\left(\rmW_1(k,c)^2+\left(\pi\trho(c)\f{\widetilde{u''}(c)}{\widetilde{u'}(c)^2}\right)^2\right)^{-\f12}\right\|_{\mathcal{G}_{ph,1}^{M_u,s_0}\left(\left[u\big(\f{\th_0}{2}\big),u\big(1-\f{\th_0}{2}\big)\right]\right)}\\
&+\left\|\Phi_{j,1}^{re}\right\|_{\mathcal{G}_{ph,1}^{M_u,s_0}\left(\left[u\big(\f{\th_0}{2}\big),u\big(1-\f{\th_0}{2}\big)\right]^2\right)}\leq C\left(k,\, \th_0,\, \|(u^{-1})''\|_{\mathcal{G}_{ph,1}^{K_u,s_0}([u(0),u(1)])}\right),
\end{align*}
which gives us that
\begin{align*}
&\left\|D_1(k,\cdot)\tchi_2\right\|_{\mathcal{G}_{ph,1}^{M_u,s_0}\left(\R\right)}
+\left\|D_2(k,\cdot)\tchi_2\right\|_{\mathcal{G}_{ph,1}^{M_u,s_0}\left(\R\right)}
+\left\|\Phi_{j,1}^{re}(k,u,c)\tchi_2(u)\tchi_2(c)\right\|_{\mathcal{G}_{ph,1}^{M_u,s_0}\left(\R^2\right)}\\
&\leq C\left(k,\, \th_0,\, \|(u^{-1})''\|_{\mathcal{G}_{ph,1}^{K_u,s_0}([u(0),u(1)])}\right). 
\end{align*}
Therefore by Lemma \ref{lem: Fourier_type1}, Remark \ref{Rmk: fourier-gevrey} and Remark \ref{Rmk: fo-g-2}, we obtain Proposition \ref{prop: kernel-wave-op}. 
\end{proof}

\subsection{Commutator}\label{Sec: comm}
In this section, let us study the difference between $\bfD_{u,k}^1$ and $\bfD_{u,k}$. We have
\begin{align*}
&\tchi_2\bfD_{u,k}(\mathcal{F}_1\tilde{f})-\tchi_2\bfD_{u,k}^1(\mathcal{F}_1\tilde{f})\\
&=\tchi_2D_2(k,u)\int_{u(0)}^{u(1)}\f{\mathcal{F}_{1}\tilde{f}(t,k,u_1)e^{-i(u_1-u)tk}}{u_1-u}\Big[\f{(u^{-1})'(u_1)\widetilde{u''}(u)}{(u^{-1})'(u)}-\widetilde{u''}(u_1)\Big]du_1\\
&\quad+\tchi_2D_2(k,u)\int_{u(0)}^{u(1)}(u^{-1})''(u)\Phi_1(k,u_1,u)\ln|u_1-u|{\mathcal{F}_{1}\tilde{f}(t,k,u_1)e^{-i(u_1-u)tk}}\\
&\qquad\qquad\qquad\qquad\qquad\qquad\qquad\qquad
\times\Big[(u^{-1})'(u_1)\widetilde{u''}(u)-(u^{-1})'(u)\widetilde{u''}(u_1)\Big]du_1\\
&\quad+\tchi_2D_2(k,u)\int_{u(0)}^{u(1)}{\mathcal{F}_{1}\tilde{f}(t,k,u_1)e^{-i(u_1-u)tk}}\Phi^{re}_{0,1}(k,u_1,u)(1_{\R^-}(u_1-u))\\
&\qquad\qquad\qquad\qquad\qquad\qquad\qquad\qquad
\times\Big[(u^{-1})'(u_1)\widetilde{u''}(u)-(u^{-1})'(u)\widetilde{u''}(u_1)\Big]du_1\\
&\quad-\tchi_2D_2(k,u)\int_{u(0)}^{u(1)}{\mathcal{F}_{1}\tilde{f}(t,k,u_1)e^{-i(u_1-u)tk}}\Phi^{re}_{1,1}(k,u_1,u)1_{\R^-}(u_1-u)\\
&\qquad\qquad\qquad\qquad\qquad\qquad\qquad\qquad
\times\Big[(u^{-1})'(u_1)\widetilde{u''}(u)-(u^{-1})'(u)\widetilde{u''}(u_1)\Big]du_1\\
&\quad+\tchi_2D_2(k,u)\int_{u(0)}^{u(1)}{\mathcal{F}_{1}\tilde{f}(t,k,u_1)e^{-i(u_1-u)tk}}\Phi^{re}_{1,1}(k,u_1,u)\\
&\qquad\qquad\qquad\qquad\qquad\qquad\qquad\qquad
\times\Big[(u^{-1})'(u_1)\widetilde{u''}(u)-(u^{-1})'(u)\widetilde{u''}(u_1)\Big]du_1. 
\end{align*}
Let $\mathcal{D}^{com}(t,k,\xi,\xi_1)$ be the Fourier kernel of the operator $\tchi_2\bfD_{u,k}(\tchi_2\mathcal{F}_1\tilde{f})-\tchi_2\bfD_{u,k}^1(\tchi_2\mathcal{F}_1\tilde{f})$, which means that
\beno
\mathcal{F}_2\Big(\tchi_2\bfD_{u,k}(\tchi_2\mathcal{F}_1\tilde{f})-\tchi_2\bfD_{u,k}^1(\tchi_2\mathcal{F}_1\tilde{f})\Big)(\xi)=\int \mathcal{D}^{com}(t,k,\xi,\xi_1)\hat{f}_k(t,\xi_1)d\xi_1.
\eeno
Then by the same argument as in the proof of Proposition \ref{prop: kernel-wave-op}, we get that there exists $\la_{\mathcal{D}}$ such that
\beno
\big|\mathcal{D}^{com}(t,k,\xi,\xi_1)\big|\lesssim e^{-\la_{\mathcal{D}}|\xi-\xi_1|^{s_0}}. 
\eeno
Let us also study the derivate $(\pa_u-itk)$ acting on $\tchi_2\bfD_{u,k}(\mathcal{F}_1\tilde{f})-\tchi_2\bfD_{u,k}^1(\mathcal{F}_1\tilde{f})$:
\begin{align*}
&(\pa_u-itk)\Big(\tchi_2\bfD_{u,k}(\mathcal{F}_1\tilde{f})-\tchi_2\bfD_{u,k}^1(\mathcal{F}_1\tilde{f})\Big)\\
&=\tchi_2D_2(k,u)\int_{u(0)}^{u(1)}(u^{-1})''(u)\Phi_1(k,u_1,u)\ln|u-u_1|{\mathcal{F}_{1}\tilde{f}(t,k,u_1)e^{-i(u_1-u)tk}}\\
&\qquad\qquad\qquad\qquad\qquad\qquad\qquad\qquad
\times\pa_u\Big[(u^{-1})'(u_1)\widetilde{u''}(u)-(u^{-1})'(u)\widetilde{u''}(u_1)\Big]du_1\\
&\quad+\tchi_2D_2(k,u)\int_{u(0)}^{u(1)}{\mathcal{F}_{1}\tilde{f}(t,k,u_1)e^{-i(u_1-u)tk}}\Phi^{re}_{0,1}(k,u_1,u)(1_{\R^-}(u_1-u))\\
&\qquad\qquad\qquad\qquad\qquad\qquad\qquad\qquad
\times\pa_u\Big[(u^{-1})'(u_1)\widetilde{u''}(u)-(u^{-1})'(u)\widetilde{u''}(u_1)\Big]du_1\\
&\quad-\tchi_2D_2(k,u)\int_{u(0)}^{u(1)}{\mathcal{F}_{1}\tilde{f}(t,k,u_1)e^{-i(u_1-u)tk}}\Phi^{re}_{1,1}(k,u_1,u)1_{\R^-}(u_1-u)\\
&\qquad\qquad\qquad\qquad\qquad\qquad\qquad\qquad
\times\pa_u\Big[(u^{-1})'(u_1)\widetilde{u''}(u)-(u^{-1})'(u)\widetilde{u''}(u_1)\Big]du_1\\
&\quad+\text{good terms}.
\end{align*}
We say the rest terms are good, because 
\beno
&&\f{\f{(u^{-1})'(u_1)\widetilde{u''}(u)}{(u^{-1})'(u)}-\widetilde{u''}(u_1)}{u-u_1}\in \mathcal{G}_{ph,1}^{K_u,s_0}([u(0),u(1)]^2),\\
&&\f{(u^{-1})'(u_1)\widetilde{u''}(u)-(u^{-1})'(u)\widetilde{u''}(u_1)}{u-u_1}\in \mathcal{G}_{ph,1}^{K_u,s_0}([u(0),u(1)]^2),
\eeno
for some $K_u>0$.

Therefore, there exists $\mathcal{D}^{com,1}(t,k,\xi,\xi_1)$ such that
\beno
(i\xi-ikt)\mathcal{F}_2\Big(\tchi_2\bfD_{u,k}(\tchi_2\mathcal{F}_1\tilde{f})-\tchi_2\bfD_{u,k}^1(\tchi_2\mathcal{F}_1\tilde{f})\Big)(\xi)=\int \mathcal{D}^{com,1}(t,k,\xi,\xi_1)\hat{f}_k(t,\xi_1)d\xi_1.
\eeno
Moreover, Remark \ref{Rmk: fo-g-2} gives us that there exists $\la_{\mathcal{D}}'$ such that 
\beno
|\mathcal{D}^{com,1}(t,k,\xi,\xi_1)|\lesssim e^{-\la_{\mathcal{D}}'|\xi-\xi_1|^{s_0}}, 
\eeno
which together with the fact that 
\beno
i(\xi-kt)\mathcal{D}^{com}(t,k,\xi,\xi_1)=\mathcal{D}^{com,1}(t,k,\xi,\xi_1),
\eeno
gives us that
\beno
\big|\mathcal{D}^{com}(t,k,\xi,\xi_1)\big|\lesssim \f{e^{-\la_{\mathcal{D}}'|\xi-\xi_1|^{s_0}}}{1+|\xi-kt|}. 
\eeno
We can repeat the above argument once more and get that 
\begin{align*}
&(\pa_u-itk)^2\Big(\tchi_2\bfD_{u,k}(\mathcal{F}_1\tilde{f})-\tchi_2\bfD_{u,k}^1(\mathcal{F}_1\tilde{f})\Big)\\
&=\tchi_2D_2(k,u)\int_{u(0)}^{u(1)}(u^{-1})''(u)\Phi_1(k,u_1,u)\f{1}{u-u_1}{\mathcal{F}_{1}\tilde{f}(t,k,u_1)e^{-i(u_1-u)tk}}\\
&\qquad\qquad\qquad\qquad\qquad\qquad\qquad\qquad
\times\Big[(u^{-1})'(u_1)\pa_u\widetilde{u''}(u)-(u^{-1})''(u)\widetilde{u''}(u_1)\Big]du_1\\
&\quad+\tchi_2D_2(k,u)\mathcal{F}_{1}\tilde{f}(t,k,u)\Phi^{re}_{0,1}(k,u,u)
\Big[(u^{-1})'(u)\pa_u\widetilde{u''}(u)-(u^{-1})''(u)\widetilde{u''}(u)\Big]\\
&\quad-\tchi_2D_2(k,u){\mathcal{F}_{1}\tilde{f}(t,k,u)}\Phi^{re}_{1,1}(k,u,u)\Big[(u^{-1})'(u)\pa_u\widetilde{u''}(u)-(u^{-1})''(u)\widetilde{u''}(u)\Big]\\
&\quad+\text{good terms}.
\end{align*}
Therefore, there exists $\mathcal{D}^{com,2}(t,k,\xi,\xi_1)$ such that
\beno
(i\xi-ikt)^2\mathcal{F}_2\Big(\tchi_2\bfD_{u,k}(\tchi_2\mathcal{F}_1\tilde{f})-\tchi_2\bfD_{u,k}^1(\tchi_2\mathcal{F}_1\tilde{f})\Big)(\xi)=\int \mathcal{D}^{com,2}(t,k,\xi,\xi_1)\hat{f}_k(t,\xi_1)d\xi_1.
\eeno
Moreover there exists $\la_{\mathcal{D}}'$ such that 
\beno
|\mathcal{D}^{com,2}(t,k,\xi,\xi_1)|\lesssim e^{-\la_{\mathcal{D}}'|\xi-\xi_1|^{s_0}}, 
\eeno
which together with the fact that 
\beno
-(\xi-kt)^2\mathcal{D}^{com}(t,k,\xi,\xi_1)=\mathcal{D}^{com,2}(t,k,\xi,\xi_1),
\eeno
gives us that
\beno
\big|\mathcal{D}^{com}(t,k,\xi,\xi_1)\big|\lesssim \f{e^{-\la_{\mathcal{D}}'|\xi-\xi_1|^{s_0}}}{1+|\xi-kt|^2}. 
\eeno
By the fact that 
\beno
\f{1+|\xi_1-kt|^2}{1+|\xi-kt|^2}\lesssim \f{1+|\xi-kt|^2+|\xi-\xi_1|^2}{1+|\xi-kt|^2}\lesssim \langle\xi-\xi_1\rangle^2,
\eeno
we obtain Corollary \ref{corol: commutator} for some $\la_{\mathcal{D}}<\la_{\mathcal{D}}'$. 

Finally, let us make a remark about the optimality of the 2 derivatives gain in the commutator. The commutator of a smooth function $u$ and the Hilbert transform $[u,H]f(x)=\f{1}{2\pi}\int_{a}^b\f{u(x)-u(y)}{x-y}f(y)dy$ is smooth as long as $f\in L^1$. However if there is a lower singularity in the integral for example $T(f)=\int_{a}^b\ln|x-y|f(y)dy$, then the commutator $[u,T]f=\int_{a}^b(u(x)-u(y))\ln|x-y|f(y)dy$ is not as good. Indeed, after taking derivatives twice, we have
\beno
\pa_x^2[u,T]f(x)=2\int_{a}^b\f{u'(x)}{x-y}f(y)dy+\int_{a}^bu''(x)\ln|x-y|f(y)dy+\int_a^b\pa_x\Big(\f{u(x)-u(y)}{x-y}\Big)f(y)dy.
\eeno
Thus the regularity of $\pa_x^2[u,T]f$ will depend on the regularity of $f$ because of the singular integral.

\end{CJK*}

\end{document}